\pdfoutput=1
\documentclass{amsart}[11pt,draft]
\usepackage{hyperref}
\usepackage{amssymb}

\usepackage{amsmath}
\usepackage{amscd}
\usepackage{amsthm}
\usepackage{graphics}
\usepackage{epsfig}
\usepackage{url}
\usepackage{mdwlist}
\usepackage{enumitem}

\usepackage[linesnumbered,norelsize]{algorithm2e}

\usepackage{pgf}
\usepackage{tikz}
\usetikzlibrary{arrows,patterns,plotmarks,shapes,snakes,er,3d,automata,backgrounds,topaths,trees,petri,mindmap}

\usepackage{verbatim} 

\usepackage[foot]{amsaddr}

\usepackage{chngcntr}
\counterwithin{table}{section}

\usepackage{placeins}

\newif\ifnever\neverfalse

\newcommand*{\CopyCounter}[2]{%
  \expandafter\def\csname c@#2\endcsname{\csname c@#1\endcsname}
  \expandafter\def\csname p@#2\endcsname{\csname p@#1\endcsname}
  \expandafter\def\csname the#2\endcsname{\csname the#1\endcsname}}

\numberwithin{Theorem}{section}
\CopyCounter{Theorem}{Proposition}
\CopyCounter{Theorem}{ProposedProblem}
\CopyCounter{Theorem}{Property}
\CopyCounter{Theorem}{Claim}
\CopyCounter{Theorem}{Lemma}
\CopyCounter{Theorem}{Corollary}
\CopyCounter{Theorem}{Conjecture}
\CopyCounter{Theorem}{Definition}
\CopyCounter{Theorem}{Example}
\CopyCounter{Theorem}{Remark}
\CopyCounter{Theorem}{Question}
\CopyCounter{Theorem}{Condition}
\CopyCounter{Theorem}{Criterion}
\CopyCounter{Theorem}{Observation}
\theoremstyle{plain}

\newtheorem{prop}[Proposition]{Proposition}

\newtheorem{property}[Property]{Property}

\newtheorem{lemma}[Lemma]{Lemma}

\theoremstyle{definition}

\newtheorem{remark}[Remark]{Remark}

\newif\ifnever\neverfalse

\newif\iffullversion
\fullversiontrue

\newif\ifoldstylefootnotes
\oldstylefootnotestrue

\newcommand{\bm}[1]{\mbox{\boldmath${#1}$}}

\newcommand{\dist}[2]{\| {#1} - {#2} \|}

\newcommand{\domain}{\Omega}
\newcommand{\cdomain}{\widebar{\Omega}}
\newcommand{\boundary}{\partial \domain}

\newcommand{\x}{\bm{x}}
\newcommand{\T}{\bm{t}}

\newcommand{\ba}{\bm{a}}

\newcommand{\bb}{\bm{b}}

\newcommand{\bv}{\bm{v}}

\newcommand{\y}{\bm{y}}
\newcommand{\z}{\bm{z}}

\newcommand{\widebar}{\overline}
\newcommand{\R}{\bm{R}}

\newcommand{\fB}{\bm{f}}

\newcommand{\M}{{\mathcal{M}}}
\newcommand{\B}{{\mathcal{B}}}

\newcommand{\D}{{\mathcal{D}}}

\newcommand{\I}{{\mathcal{I}}}

\DeclareMathOperator{\Forall}{\;\forall\,}

\newcommand{\safe}{\mathcal{S}}
\newcommand{\nsafe}{\mathcal{U}}
\newcommand{\del}{\partial}
\newcommand{\A}{{\mathcal{A}}}
\newcommand{\J}{{\mathcal{J}}}

\newcommand{\control}{\bm{a}}
\newcommand{\runcost}{K}
\newcommand{\runcostb}{\hat{K}}
\newcommand{\fb}{f}
\newcommand{\val}{u} 
\newcommand{\term}{q} 
\newcommand{\maxb}{B}
\newcommand{\mfl}{v}
\newcommand{\edomain}{\domain_e}
\newcommand{\rdomain}{\domain_r}
\newcommand{\udomain}{\domain_{\nsafe}}
\newcommand{\sdomain}{\domain_{\safe}}
\newcommand{\bud}{b}
\newcommand{\budp}{\beta}
\newcommand{\bq}{\begin{equation}}
\newcommand{\eq}{\end{equation}}
\newcommand{\bal}{\begin{align}}
\newcommand{\eal}{\end{align}}
\newcommand{\bmat}{\left[\begin{matrix}}
\newcommand{\emat}{\end{matrix}\right]}
\newcommand{\h}{h}

\newcommand{\target}{\mathcal{T}}
\newcommand{\grad}{\nabla}

\newcommand{\reach}{\mathcal{R}}
\newcommand{\grid}{\mathcal{G}}
\newcommand{\Val}{W}

\newcommand{\vale}{w}
\newcommand{\valr}{w}

\newcommand{\interface}{\Gamma}
\newcommand{\valu}{w_1}
\newcommand{\vals}{w_2}
\newcommand{\safereach}{\safe_{\mathcal{R}}}

\newcommand{\nsafereach}{{\nsafe}_{\mathcal{R}}}

\newcommand{\Valu}{W_1}
\newcommand{\Vals}{W_2}

\newcommand{\Bd}{\mathcal{B}}
\newcommand{\Nb}{N_{\bud}}
\newcommand{\mflv}{{\tilde u}}
\newcommand{\mflV}{{\tilde U}}
\newcommand{\Mfl}{V}
\newcommand{\aux}{g}
\newcommand{\Aux}{G}

\newcommand{\Line}{\mathcal{L}}
\newcommand{\Valn}{U}
\newcommand{\domainreach}{{\domain}_{\reach}}

\newcommand{\marginfix}{
\setlength{\parskip}{0.01cm}
\setlength{\textwidth}{6.0in}
\setlength{\oddsidemargin}{-0.0 in}
\setlength{\evensidemargin}{0.0 in}
\setlength{\topmargin}{-0.5in}
\setlength{\textheight}{9.0 in}
}

\marginfix

  \renewenvironment{thebibliography}[1]{%
    \begin{oldthebibliography}{#1}%
      \setlength{\parskip}{.3ex}%
      \setlength{\itemsep}{.3ex}%
  }%
  {%
    \end{oldthebibliography}%
  }

\ifoldstylefootnotes

\iffullversion
\title{Optimal Control with Reset-Renewable Resources.}
\else
\title{Optimal Control with Budget Constraints and Resets.}
\fi

\author{R. Takei$^{a,e}$}
\author{W. Chen$^{b,f}$}
\author{Z. Clawson$^{c,f}$}
\author{S. Kirov$^{d,f}$}
\author{A. Vladimirsky$^{b,c,g}$}

\address{$^A$ Department of Electrical Engineering and Computer Sciences, UC Berkeley, CA 94720} 	
\address{$^B$ Department of Mathematics, Cornell University, Ithaca, NY 14853}
\address{$^C$ Center for Applied Mathematics, Cornell University, Ithaca, NY 14853}
\address{$^D$ Department of Mathematical Sciences, Carnegie Mellon University, Pittsburgh, PA 15213}
\address{$^E$ Supported in part by ONR grants N0014-03-1-0071, N00014-07-1-0810, N00014-08-1-1119, DOE grant DE-FG02-05ER25710, an ARO MURI through
Rice University and the CHASE MURI grant 556016.}
\address{$^F$ Supported in part by the National Science Foundation through the Research
Experiences for Undergraduates Program at Cornell.} 	
\address{$^G$ Supported in part by the National Science Foundation grant 
DMS-1016150.}

\renewcommand*{\thefootnote}{\arabic{footnote}}
\fi

\begin{document}

\ifoldstylefootnotes
\maketitle

\fi

\vspace*{.1in}
\begin{abstract}
\noindent
We consider both discrete and continuous control problems constrained
by a fixed budget of some resource, which may be renewed upon entering
a preferred subset of the state space.  In the discrete case, we consider
\iffullversion
both deterministic and stochastic
\fi
shortest path problems on graphs with a full
budget reset in all preferred nodes.
In the continuous case, we derive augmented PDEs of optimal control,
which are then solved numerically
on the extended state space with a full/instantaneous
budget reset on the preferred subset.  We introduce an iterative
algorithm for solving these problems efficiently.
The method's performance is demonstrated on
a range of computational examples, including
optimal path planning with constraints on
prolonged visibility by a static enemy observer.

\iffullversion
In addition, we also develop an
algorithm that works
on the original state space to solve a related but simpler problem:
finding the subsets of the domain ``reachable-within-the-budget''.

This manuscript is an extended version of the paper accepted for publication by SIAM J. on Control and Optimization.
In the journal version, Section 3 and the Appendix were omitted due to space limitations.
\fi
\end{abstract}

\section{Introduction.}
\label{s:intro}



Dynamic programming provides a convenient framework
for finding provably ``optimal'' strategies to control
both discrete and continuous systems.
The optimality is usually defined with respect to
a single criterion or cost (e.g., money, fuel, or time
needed to implement each particular control).
Given a set of possible system configurations $\domain$,
the {\em value function} is defined as the cost of such
optimal control for each starting position $\x \in \domain$,
and the
Dynamic Programming 
equations are then solved to recover
this value function.

Realistic applications usually involve several different
criteria for evaluating control strategies and the notion
of ``optimality'' becomes more complicated.
One natural approach is to focus on a single ``primary'' cost
to be optimized, while treating all other (``secondary'')
costs as constraints.
A typical application of this type is to find the fastest path
to the target subject to
constraints on the maximum use of energy along that path.
Another possible constraint is on the maximum
total amount of time that a robot can be visible to
an unfriendly observer while moving to a target.
Kumar and Vladimirsky have recently introduced
efficient numerical methods for these and similar
constrained-optimal control problems
in continuous domains \cite{KumarVlad}.
We will refer to such problems
as ``budget-constrained'' since
the original state space will have to be expanded to keep track of
the remaining ``resource-budget'' to satisfy the constraint.
(E.g., two different starting energy-budgets
may well result in very different ``constrained time-optimal''
paths, even when the starting physical position is the same.)

This state space expansion leads to a significantly larger
system of augmented dynamic programming equations or a higher-dimensional
domain in the continuous case.
Thus, the computational efficiency of numerical methods
becomes particularly important.
In many applications it is natural to assume that
the resource budgets are non-renewable and
that every change in the system state results
in an immediate budget decrease.
This introduces a natural ``causal ordering'' on
the expanded state space:
the knowledge of energy-constrained time-optimal
controls for the starting energy-budget $\bb_1$
can be used
in computing such constrained-optimal
controls for a higher budget $\bb_2 > \bb_1$.
The efficiency of methods introduced in
\cite{KumarVlad} hinges on this explicit causality.

In the current paper we focus on a more general situation,
where the budget can be {\em instantaneously reset}
to its maximum value by visiting a special part of
the state space $\safe \subset \domain$ and the budget
is decreasing when moving through the rest of the state
space $\nsafe = \domain \backslash \safe.$
If the limited resource is fuel, $\safe$ can be thought
of as a discrete set of gas stations.  On the other hand,
if the secondary cost is the vehicle's exposure
to an unfriendly observer, the $\safe$ can be interpreted
as a ``safe'' part of the domain protected from observation,
and the constraint is on the maximum {\em contiguous} period of
time that the system is allowed to spend in an
``unsafe'' (observable) set $\nsafe$.
To the best of our knowledge, our formulation of
continuous versions of such problems is new and
no efficient methods for these were previously available.
We show that such ``budget-resets'' result in a much more subtle
implicit causality in the expanded state space.
Nevertheless, under mild technical assumptions,
non-iterative methods are available for deterministic
(section \ref{s:deterministic_SP_reset})
and even for certain stochastic
\iffullversion
(section \ref{s:SSP})
\fi
budget-reset problems on graphs.
In the continuous case, this problem is described by
a controlled hybrid system (section \ref{s:continuous_reset}),
whose value function is a discontinuous viscosity solution of two
different Hamilton-Jacobi PDEs posed on $\safe$
and $\nsafe \times \B$, where $\B$ is the set of possible budget
levels available to satisfy secondary constraints.
The characteristics of these PDEs
coincide with the constrained-optimal trajectories and define
the direction of information flow in the continuous state space.
Unfortunately, the most natural semi-Lagrangian discretization of
these PDEs is no longer causal, making iterative numerical methods
unavoidable (section \ref{ss:Numerics}).
The key contributions of this paper are the
study of properties of the value functions for budget reset problems
(section \ref{ss:br_pde})
and the fast iterative algorithm for approximating such value functions
(sections \ref{ss:iterative_brp}-\ref{ss:mfl}).
Our general approach is equally efficient even in
the presence of strong inhomogeneities
and anisotropies
in costs and/or dynamics of the system.
In section \ref{sec:numresults}, we provide numerical evidence
of the method's convergence and illustrate the key properties
on several optimal control examples\footnote{
Matlab/C++ codes used to produce all experimental results in this paper
are 
available from \mbox{\url{http://www.math.cornell.edu/~vlad/papers/b_reset/}}}, including
prolonged-visibility-avoidance problems.

We note that for budget-reset problems, finding constrained-optimal controls
is significantly harder than just identifying the ``reachable''
subset of $\domain$; i.e., the set of states from which it's possible to
steer the system to the target without violating the constraints, provided
one starts with the maximum budget.  A more efficient algorithm for the
latter problem
\iffullversion
is presented in
Appendix \ref{s:Reachable}.
\fi
We conclude by discussing the limitations of our approach and the
future work in section \ref{s:conclusions}.

\section{Deterministic SP on graphs with renewable resources.}
\label{s:deterministic_SP_reset}

The classical problem of finding the shortest path (SP) in a directed graph with
non-negative arclengths is among those most exhaustively studied.
Fast iterative (``label-correcting'') and fast non-iterative (``label-setting'')
methods for it are widely used in a variety of applications.

Consider a directed graph on a set of nodes
$X = \{ \x_1, \ldots, \x_M, \x_{M+1}=\T \}$, where the last node $\T$ is the target.
For each node $\x_i \in X$, there is a set of immediate neighbors
$N_i = N(\x_i) \subset X \backslash x_i$ and for each neighbor $\x_j \in N_i$
there is a known transition cost $C_{ij} = C(\x_i,\x_j) \geq 0$.
For convenience, we will take $C_{ij} = +\infty$ if $\x_j \not \in N_i.$
We will further assume that our graph is relatively sparse; i.e.,
$\max_i |N_i| \leq \kappa \ll M.$
If $\y = (\y_0=\x_i, \y_1, ..., \y_r = \T)$ is a path from $\x_i$ to the target,
its total cost is defined as $\J(\y) = \sum_{k=0}^{r-1} C(\y_k,\y_{k+1}).$
The value function $U_i = U(\x_i)$ is defined as the cost along an optimal path.
Clearly $U_{\T} = 0$ and for all other nodes the Bellman optimality principle yields
a system of $M$ non-linear coupled equations:
\begin{equation}
\label{eq:primary}
U_i = \min\limits_{\x_j \in N_i} \left\{ C_{ij} + U_j \right\};
\qquad i \in \I = \{1, \ldots, M\}.
\end{equation}
(This definition makes $U_i = +\infty$, whenever there is no path from $\x_i$
to $\T$, including the cases where $N_i = \emptyset$.  Throughout this paper, we will
use the convention that the minimum over an empty set is $+\infty$.)
Dijkstra's classical method \cite{Diks} can be used to solve the system
\eqref{eq:primary} non-iteratively
in $O(M \log M)$ operations.  A detailed discussion of this and other label-setting and
label-correcting algorithms can be found in standard references; e.g. \cite{Ahuja, Bertsekas_NObook}.

In many applications, a natural extension of this problem requires keeping track of
several different types of cost (e.g., money, fuel, time) associated with each transition.
The goal is then to either find all Pareto optimal paths or to treat one criterion/cost
as primary (to be minimized) and all others as secondary, providing the constraints
to restrict the set of allowable paths.  (E.g., what is {\em the quickest} path given
the current amount of fuel in our tank?)
A number of algorithms for such multi-objective dynamic programming
are also well-known \cite{Hansen, Martins, Jaffe}.

One natural approach to bicriterion problems
involves finding a simple, single-criterion
optimal path but in an expanded graph with the node set $\widehat{X} = X \times \B$.
We begin with a similar technique adopted to our ``constrained resources'' scenario.
Our description emphasizes the causal properties of the model, also reflected by
the numerical methods for the continuous case in section \ref{ss:Numerics}.
We assume the secondary resource-cost for each transition is specified as
$c_{ij} = c(\x_i, \x_j)$,
whereas the primary cost of that transition will be still denoted as $C_{ij}$.
For simplicity we assume that the secondary
costs are conveniently quantized; e.g., $c_{ij} \in \mathbb{Z}$.
We will use $B>0$ to denote the maximal allowable budget and
$\B = \{0, \ldots, B\}$ to denote the set of allowable budget levels.
In the extended graph, a node $\x_i^b \in \widehat{X}$ corresponds to
a location $\x_i \in X$ with the resource budget $b$.
The use of resources occurs when $c_{ij}>0$,
and the renewal/recharge of resources corresponds to $c_{ij}<0$.
If we are starting from $\x_i$ with a secondary-budget $b$ and decide
to move to $\x_j \in N_i$, in the expanded graph this becomes a transition
from $\x_i^b$ to $\x_j^{b-c_{ij}}$.

We now define the value function
$W_i^b = W(\x_i^b)$ as the
minimum accumulated primary-cost from $\x_i$ to $\T$, but minimizing only among the paths
along which the budget always remains in $\B$:
$$
W_i^b = \min\limits_{\y \in Y^b(\x_i)} \J(\y),
$$
where $Y^b(\x_i)$ is the set of ``$b$-feasible paths''
(i.e., those that can be traversed if starting from $\x_i$ with resource budget $b$).

Feasible paths should clearly include only those, along which
the resource budget remains non-negative.
But since we are allowing for secondary costs $c_{ij}$ of arbitrary sign,
this requires
a choice between two different ``upper budget bound'' interpretations:

1)  Any attempt to achieve a budget higher than $B$ makes a path
infeasible; i.e., since we are starting from $\x_i$ with budget
$b \in \B$,
$$
Y^b(\x_i) = \left\{\y = (\y_0=\x_i, \ldots, \y_r = \T) \, \mid \;
0 \, \leq \,
\left( b \, - \, \sum_{k=0}^{s-1} c(\y_k,\y_{k+1}) \right) \, \leq \, B, \; \Forall s \leq r \right\}.
$$
In this case, the optimality principle would yield a system of equations on the expanded graph:
$$
W_i^b = \min\limits_{\x_j \in N_i} \left\{ C_{ij} + W_j^{b-c_{ij}} \right\};
\qquad \qquad \Forall b \in \B; \qquad \Forall i \in \I;
$$
with the following ``boundary conditions'':
$$
W_{\T}^b = 0,  \; \Forall b \in \B; \qquad
W_i^b = +\infty, \; \Forall b \not \in \B, \Forall \x_i \in X.
$$
In practice, the latter condition can be omitted if we minimize over
$N_i^b = \{\x_j \in N_i \mid (b-c_{ij}) \in \B \}$ instead of $N_i$.

2) An interpretation more suitable for our purposes is to assume that
any resources in excess of $B$ are simply not accumulated (or are immediately lost),
but the path remains allowable.  If we define
a left-associative operation
$\alpha \ominus \beta = \min(\alpha-\beta, B)$, then
$$
Y^b(\x_i) = \left\{\y = (\y_0=\x_i, \ldots, \y_r = \T) \, \mid \;
\left( b \, \ominus c(\y_0,\y_1) \ominus \ldots \ominus c(\y_{s-1},\y_{s}) \right)
\, \geq \, 0,
\; \Forall s \leq r \right\},
$$
and the optimality principle on the expanded graph can be expressed as follows:
\begin{equation}
\label{DP:extended}
W_i^b = \min\limits_{\x_j \in N_i^b} \left\{ C_{ij} + W_j^{b \ominus c_{ij}} \right\};
\qquad \qquad \Forall b \in \B; \qquad \Forall i \in \I,
\end{equation}
with $W_{\T}^b = 0,  \Forall b \in \B$ and
$N_i^b = \{\x_j \in N_i \mid c_{ij} \leq b \}.$
Unlike the previous interpretation, this definition ensures that
the value function is
non-increasing in $b$
(since the set of feasible paths is non-decreasing as
we increase the budget).  Thus, we will also similarly
interpret the ``upper budget bound'' in
\iffullversion
sections \ref{s:SSP} and
\fi
\ref{s:continuous_reset}.


\begin{remark} \label{rem:types_of_causality}
It is natural to view the expanded graph as $(B+1)$ ``$b$-slices''
(i.e., copies of the original graph) stacked vertically,
with the transitions between slices corresponding to $c_{ij}$'s.
(Though, since the total costs of feasible paths
do not depend on the budget remaining upon reaching the target,
it is in fact unnecessary to have multiple copies of $\T$ in
the expanded graph; see Figure \ref{ex2}).
We note that the signs of the secondary costs
strongly influence the type of causality present
in the system as well as the efficiency of the corresponding numerical methods.
We consider three cases:\\

1. Since the primary costs are non-negative, the system \eqref{DP:extended} can always be
solved by Dijkstra's method on an expanded graph
regardless of the signs of $c_{ij}$'s.
The computational cost of this approach is $O(\hat{M} \log \hat{M})$, where
$\hat{M} = M \times (B+1)$ is the number of nodes in the expanded graph
(not counting the target).
We will refer to this most general type of causality as {\em implicit}.

2. However, if $c_{ij}$'s are strictly positive, Dijkstra's method
becomes unnecessary since the budget will strictly decrease along every path.
In this case $W_i^0 = +\infty$ for all $\x_i \in X \backslash \{\T\}$
and each $W_i^b$ depends only on nodes in the lower slices.
We will refer to this situation as {\em explicitly causal} since we can
compute the value function in $O(\hat{M})$ operations, in a single sweep
from the bottom to the top slice.

3. On the other hand, if the secondary costs are only non-negative,
the budget is non-increasing along
every path and we can use Dijkstra's method on one $b$-slice at a time
(again, starting from $b=0$ and advancing to $b=B$) instead of using it on
the entire expanded graph at once.
This {\em semi-implicit causality} results in the computational cost
of $O(\hat{M} \log M).$
\end{remark}

\begin{remark} \label{rem:partial_solutions}
The sign of secondary costs also determines whether the value function
$W_i^b$ actually depends on the maximum allowable resource level.
E.g., suppose the solution
$W_i^b$ was already computed
for all $b \in \B= \{0, \ldots, B\}$
and it is now necessary to solve the problem again, but for a wider range
of resource budgets $\tilde{\B} = \{0,\ldots, \tilde{B}\}$,
where $\tilde{B} > B$.
Assuming that $\tilde{W}$ solves the latter problem,
a useful consequence of explicit and semi-implicit causality is the fact
that $\tilde{W}_i^b = W_i^b$ for all $b \in \B$; thus,
the computational costs of this extension are decreased
by concentrating on the set of budgets $\{B+1,\ldots, \tilde{B}\}$ only.
Unfortunately, this convenient property does not generally hold
for implicitly causal problems.  (E.g., when refueling is allowed,
two cars with different capacity of gas tanks might have
different optimal driving directions
even if they are starting out with the same amount of gas.)
\end{remark}

\subsection{Safe/Unsafe splitting without budget resets.}
\label{ss:DSP_no_reset}

The framework presented so far is suitable for fairly general resource
expenditure/accumulation models.
In this paper, we are primarily interested in a smaller class of problems,
where the state space is split into ``safe'' and ``unsafe'' components
(i.e., $X \backslash \{\T\} = \safe \cup \nsafe$) with the secondary cost not accrued
(i.e., the constrained resource not used at all)
while traveling through $\safe.$
A simple example of this is provided in Figure \ref{ex1}.

In the absence of ``budget resets'' this safe/unsafe splitting is modeled
by simply setting $c_{ij} = 0$ whenever $\x_i \in \safe,$
yielding semi-implicit causality; see Figures \ref{ex2} and \ref{ex3}.
To make this example intuitive, we have selected
the primary costs $C_{ij}$ to be such that the ``primary-optimal''
path (described by $U$; see Equation \eqref{eq:primary})
always proceeds from $\x_i$ to $\x_{i+1}$, etc.
However, the budget limitations can make such primary-optimal paths
infeasible. E.g., when $B=3$ and no resets are allowed,
the best feasible path from $\x_4$ is $(\x_4, \x_5, \x_8, \T)$,
resulting in $W_4^3 = 6 > 5 = U_4$; see Figure \ref{ex3}.
Still, all constrained-optimal paths travel either on the same
$b$ slice or downward,
resulting in a semi-implicit causality.


\tikzstyle{Upoint}=[draw,shape=circle,inner sep = 1mm,minimum size = .7cm]
\tikzstyle{Spoint}=[draw,shape=diamond,inner sep = 1mm,minimum size = .8cm]
\tikzstyle{Tpoint}=[draw,shape=rectangle,inner sep = 1mm,minimum size = .7cm]

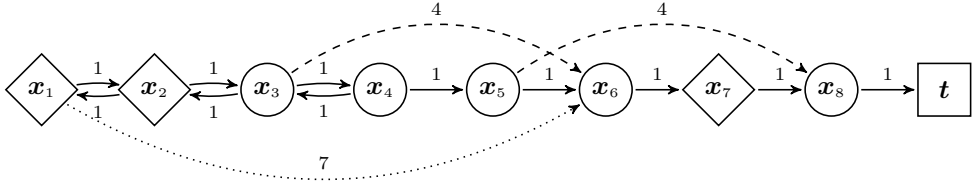
\begin{figure}[h!]
\centering
\begin{tikzpicture}[scale = .75,->,>=stealth',shorten >=1pt,shorten <=1pt,auto,node distance = 1.5cm,semithick]

    	\node[Spoint] (X1) at (0,0) {$\scriptstyle{\x_1}$};
    	\node[Spoint] (X2) at (2,0) {$\scriptstyle{\x_2}$};
    	\node[Upoint] (X3) at (4,0) {$\scriptstyle{\x_3}$};
    	\node[Upoint] (X4) at (6,0) {$\scriptstyle{\x_4}$};
    	\node[Upoint] (X5) at (8,0) {$\scriptstyle{\x_5}$};
    	\node[Upoint] (X6) at (10,0) {$\scriptstyle{\x_6}$};
    	\node[Spoint] (X7) at (12,0) {$\scriptstyle{\x_7}$};
    	\node[Upoint] (X8) at (14,0) {$\scriptstyle{\x_8}$};
    	\node[Tpoint] (T) at (16,0) {$\scriptstyle{\bm{t}}$};
    	
    	\path (X1) edge [bend left=8] node [above] {$\scriptstyle{1}$} (X2)
				   edge [bend right=30, dotted] node [above] {$\scriptstyle{7}$} (X6);
		\path (X2) edge [bend left=8] node [above] {$\scriptstyle{1}$} (X3)
				   edge [bend left=8] node [below] {$\scriptstyle{1}$} (X1);
		\path (X3) edge [bend left=8] node [above] {$\scriptstyle{1}$} (X4)
				   edge [bend left=8] node [below] {$\scriptstyle{1}$} (X2)
				   edge [bend left=35, dashed] node [above] {$\scriptstyle{4}$} (X6);
		\path (X4) edge node [above] {$\scriptstyle{1}$} (X5)
				   edge [bend left=8] node [below] {$\scriptstyle{1}$} (X3);
		\path (X5) edge node [above] {$\scriptstyle{1}$} (X6)
				   edge [bend left=35, dashed] node [above] {$\scriptstyle{4}$} (X8);
    	\path (X6) edge node [above] {$\scriptstyle{1}$} (X7);
    	\path (X7) edge node [above] {$\scriptstyle{1}$} (X8);
    	\path (X8) edge node [above] {$\scriptstyle{1}$} (T);
\end{tikzpicture}
\caption{Diamond-shaped nodes are in $\safe$ and
circle-shaped ones are in $\nsafe$.  The primary costs $C_{ij}$ are specified
for each link, while the secondary costs are $c_{ij} = 1$ if $\x_i \in \nsafe$
and $c_{ij} = 0$ if $\x_i \in \safe.$
The arrow types (solid, dashed, or dotted) also correspond
to different values of primary transition costs.
To build a concrete illustration, we will assume that $B=3$
is the maximum number of ``unsafe'' (circle-shaped) nodes we can go through.}
\label{ex1}
\end{figure}

\begin{figure}[hhhh]
\centering
\begin{tikzpicture}[scale = .75,->,>=stealth',shorten >=1pt,shorten <=1pt,auto,node distance = 1.5cm,semithick]

    	\node (X1) at (0,7.5) {$\x_1$};
    	\node (X2) at (2,7.5) {$\x_2$};
    	\node (X3) at (4,7.5) {$\x_3$};
    	\node (X4) at (6,7.5) {$\x_4$};
    	\node (X5) at (8,7.5) {$\x_5$};
    	\node (X6) at (10,7.5) {$\x_6$};
    	\node (X7) at (12,7.5) {$\x_7$};
    	\node (X8) at (14,7.5) {$\x_8$};

		\node (b3) at (-2,6) {$\bm{b=3}$};
    	\node[Spoint] (X13) at (0,6) {$\scriptstyle{\x_1^3}$};
    	\node[Spoint] (X23) at (2,6) {$\scriptstyle{\x_2^3}$};
    	\node[Upoint] (X33) at (4,6) {$\scriptstyle{\x_3^3}$};
    	\node[Upoint] (X43) at (6,6) {$\scriptstyle{\x_4^3}$};
    	\node[Upoint] (X53) at (8,6) {$\scriptstyle{\x_5^3}$};
    	\node[Upoint] (X63) at (10,6) {$\scriptstyle{\x_6^3}$};
    	\node[Spoint] (X73) at (12,6) {$\scriptstyle{\x_7^3}$};
    	\node[Upoint] (X83) at (14,6) {$\scriptstyle{\x_8^3}$};
	
		\node (b2) at (-2,4) {$\bm{b=2}$};
    	\node[Spoint] (X12) at (0,4) {$\scriptstyle{\x_1^2}$};
    	\node[Spoint] (X22) at (2,4) {$\scriptstyle{\x_2^2}$};
    	\node[Upoint] (X32) at (4,4) {$\scriptstyle{\x_3^2}$};
    	\node[Upoint] (X42) at (6,4) {$\scriptstyle{\x_4^2}$};
    	\node[Upoint] (X52) at (8,4) {$\scriptstyle{\x_5^2}$};
    	\node[Upoint] (X62) at (10,4) {$\scriptstyle{\x_6^2}$};
    	\node[Spoint] (X72) at (12,4) {$\scriptstyle{\x_7^2}$};
    	\node[Upoint] (X82) at (14,4) {$\scriptstyle{\x_8^2}$};
	
		\node (b1) at (-2,2) {$\bm{b=1}$};
    	\node[Spoint] (X11) at (0,2) {$\scriptstyle{\x_1^1}$};
    	\node[Spoint] (X21) at (2,2) {$\scriptstyle{\x_2^1}$};
    	\node[Upoint] (X31) at (4,2) {$\scriptstyle{\x_3^1}$};
    	\node[Upoint] (X41) at (6,2) {$\scriptstyle{\x_4^1}$};
    	\node[Upoint] (X51) at (8,2) {$\scriptstyle{\x_5^1}$};
    	\node[Upoint] (X61) at (10,2) {$\scriptstyle{\x_6^1}$};
    	\node[Spoint] (X71) at (12,2) {$\scriptstyle{\x_7^1}$};
    	\node[Upoint] (X81) at (14,2) {$\scriptstyle{\x_8^1}$};

		\node (b0) at (-2,0) {$\bm{b=0}$};
    	\node[Spoint] (X10) at (0,0) {$\scriptstyle{\x_1^0}$};
    	\node[Spoint] (X20) at (2,0) {$\scriptstyle{\x_2^0}$};
    	\node[Upoint] (X30) at (4,0) {$\scriptstyle{\x_3^0}$};
    	\node[Upoint] (X40) at (6,0) {$\scriptstyle{\x_4^0}$};
    	\node[Upoint] (X50) at (8,0) {$\scriptstyle{\x_5^0}$};
    	\node[Upoint] (X60) at (10,0) {$\scriptstyle{\x_6^0}$};
    	\node[Spoint] (X70) at (12,0) {$\scriptstyle{\x_7^0}$};
    	\node[Upoint] (X80) at (14,0) {$\scriptstyle{\x_8^0}$};

    	\node[Tpoint] (T) at (16, 0) {$\scriptstyle{\bm{t}}$};

    	\path (X13) edge [bend left=8] (X23)
				    edge [bend right=20, dotted] (X63);
    	\path (X12) edge [bend left=8] (X22)
				    edge [bend right=20, dotted] (X62);
    	\path (X11) edge [bend left=8] (X21)
				    edge [bend right=20, dotted] (X61);
    	\path (X10) edge [bend left=8] (X20)
				    edge [bend right=20, dotted] (X60);

    	\path (X23) edge [bend left=8] (X13);
    	\path (X22) edge [bend left=8] (X12);
    	\path (X21) edge [bend left=8] (X11);
    	\path (X20) edge [bend left=8] (X10);

    	\path (X23) edge (X33);
    	\path (X22) edge (X32);
    	\path (X21) edge (X31);
    	\path (X20) edge (X30);

    	\path (X33) edge (X22)
					edge (X42)
					edge [dashed] (X62);
    	\path (X32) edge (X21)
					edge (X41)
					edge [dashed] (X61);
    	\path (X31) edge (X20)
					edge (X40)
					edge [dashed] (X60);

    	\path (X43) edge (X32)
					edge (X52);
    	\path (X42) edge (X31)
					edge (X51);
    	\path (X41) edge (X30)
					edge (X50);

    	\path (X53) edge (X62)
					edge [dashed] (X82);
    	\path (X52) edge (X61)
					edge [dashed] (X81);
    	\path (X51) edge (X60)
					edge [dashed] (X80);

    	\path (X63) edge (X72);
    	\path (X62) edge (X71);
    	\path (X61) edge (X70);

    	\path (X73) edge (X83);
    	\path (X72) edge (X82);
    	\path (X71) edge (X81);
    	\path (X70) edge (X80);

    	\path (X83) edge (T);
    	\path (X82) edge (T);
    	\path (X81) edge (T);

\end{tikzpicture}
\caption{Budget-constrained shortest path (no ``resets'').
For every node except for the target, the superscript denotes
{\bf the remaining unsafe-node budget} at that node.
Transition on the same level from a safe node, down to the next level
if going from an unsafe one.  The primary transition costs are indicated
by the type of the arrow; see Figure \ref{ex1}.}
\label{ex2}
\end{figure}
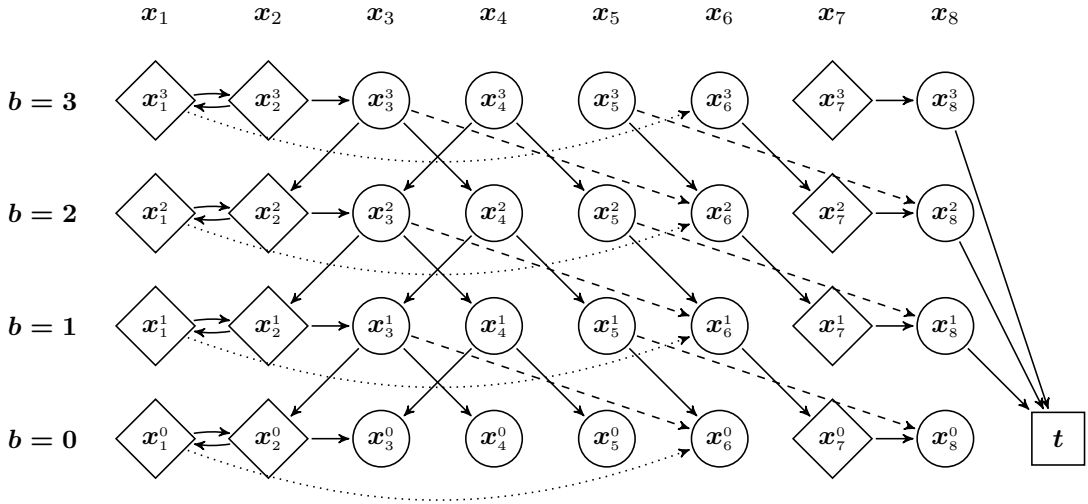

\begin{figure}[hhhh]
\centering
\begin{tikzpicture}[scale = .75,->,>=stealth',shorten >=1pt,shorten <=1pt,auto,node distance = 1.5cm,semithick]

    	\node (X1) at (0,7.5) {$\x_1$};
    	\node (X2) at (2,7.5) {$\x_2$};
    	\node (X3) at (4,7.5) {$\x_3$};
    	\node (X4) at (6,7.5) {$\x_4$};
    	\node (X5) at (8,7.5) {$\x_5$};
    	\node (X6) at (10,7.5) {$\x_6$};
    	\node (X7) at (12,7.5) {$\x_7$};
    	\node (X8) at (14,7.5) {$\x_8$};

		\node (b3) at (-2,6) {$\bm{b=3}$};
    	\node[Spoint] (X13) at (0,6) {$\scriptstyle{9}$};
    	\node[Spoint] (X23) at (2,6) {$\scriptstyle{8}$};
    	\node[Upoint] (X33) at (4,6) {$\scriptstyle{7}$};
    	\node[Upoint] (X43) at (6,6) {$\scriptstyle{6}$};
    	\node[Upoint] (X53) at (8,6) {$\scriptstyle{4}$};
    	\node[Upoint] (X63) at (10,6) {$\scriptstyle{3}$};
    	\node[Spoint] (X73) at (12,6) {$\scriptstyle{2}$};
    	\node[Upoint] (X83) at (14,6) {$\scriptstyle{1}$};
	
		\node (b2) at (-2,4) {$\bm{b=2}$};
    	\node[Spoint] (X12) at (0,4) {$\scriptstyle{10}$};
    	\node[Spoint] (X22) at (2,4) {$\scriptstyle{11}$};
    	\node[Upoint] (X32) at (4,4) {$\scriptstyle{\infty}$};
    	\node[Upoint] (X42) at (6,4) {$\scriptstyle{\infty}$};
    	\node[Upoint] (X52) at (8,4) {$\scriptstyle{5}$};
    	\node[Upoint] (X62) at (10,4) {$\scriptstyle{3}$};
    	\node[Spoint] (X72) at (12,4) {$\scriptstyle{2}$};
    	\node[Upoint] (X82) at (14,4) {$\scriptstyle{1}$};
	
		\node (b1) at (-2,2) {$\bm{b=1}$};
    	\node[Spoint] (X11) at (0,2) {$\scriptstyle{\infty}$};
    	\node[Spoint] (X21) at (2,2) {$\scriptstyle{\infty}$};
    	\node[Upoint] (X31) at (4,2) {$\scriptstyle{\infty}$};
    	\node[Upoint] (X41) at (6,2) {$\scriptstyle{\infty}$};
    	\node[Upoint] (X51) at (8,2) {$\scriptstyle{\infty}$};
    	\node[Upoint] (X61) at (10,2) {$\scriptstyle{\infty}$};
    	\node[Spoint] (X71) at (12,2) {$\scriptstyle{2}$};
    	\node[Upoint] (X81) at (14,2) {$\scriptstyle{1}$};

		\node (b0) at (-2,0) {$\bm{b=0}$};
    	\node[Spoint] (X10) at (0,0) {$\scriptstyle{\infty}$};
    	\node[Spoint] (X20) at (2,0) {$\scriptstyle{\infty}$};
    	\node[Upoint] (X30) at (4,0) {$\scriptstyle{\infty}$};
    	\node[Upoint] (X40) at (6,0) {$\scriptstyle{\infty}$};
    	\node[Upoint] (X50) at (8,0) {$\scriptstyle{\infty}$};
    	\node[Upoint] (X60) at (10,0) {$\scriptstyle{\infty}$};
    	\node[Spoint] (X70) at (12,0) {$\scriptstyle{\infty}$};
    	\node[Upoint] (X80) at (14,0) {$\scriptstyle{\infty}$};

    	\node[Tpoint] (T) at (16, 0) {$\scriptstyle{0}$};

    	\path (X13) edge [bend left=8] (X23)
				    edge [bend right=20, dotted] (X63);
    	\path (X12) edge [bend left=8] (X22)
				    edge [bend right=20, dotted] (X62);
    	\path (X11) edge [bend left=8] (X21)
				    edge [bend right=20, dotted] (X61);
    	\path (X10) edge [bend left=8] (X20)
				    edge [bend right=20, dotted] (X60);

    	\path (X23) edge [bend left=8] (X13);
    	\path (X22) edge [bend left=8] (X12);
    	\path (X21) edge [bend left=8] (X11);
    	\path (X20) edge [bend left=8] (X10);

    	\path (X23) edge (X33);
    	\path (X22) edge (X32);
    	\path (X21) edge (X31);
    	\path (X20) edge (X30);

    	\path (X33) edge (X22)
					edge (X42)
					edge [dashed] (X62);
    	\path (X32) edge (X21)
					edge (X41)
					edge [dashed] (X61);
    	\path (X31) edge (X20)
					edge (X40)
					edge [dashed] (X60);

    	\path (X43) edge (X32)
					edge (X52);
    	\path (X42) edge (X31)
					edge (X51);
    	\path (X41) edge (X30)
					edge (X50);

    	\path (X53) edge (X62)
					edge [dashed] (X82);
    	\path (X52) edge (X61)
					edge [dashed] (X81);
    	\path (X51) edge (X60)
					edge [dashed] (X80);

    	\path (X63) edge (X72);
    	\path (X62) edge (X71);
    	\path (X61) edge (X70);

    	\path (X73) edge (X83);
    	\path (X72) edge (X82);
    	\path (X71) edge (X81);
    	\path (X70) edge (X80);

    	\path (X83) edge (T);
    	\path (X82) edge (T);
    	\path (X81) edge (T);

\end{tikzpicture}
\caption{Semi-implicit causality: run Dijkstra's algorithm
{\bf slice-by-slice} on the expanded graph (from bottom to the top).
The number inside each node represents the minimum cost from  $\x_i^b$ to the target.
(If there are no ``safe cycles'', the causality becomes explicit
and there is  no need for Dijkstra's.)}
\label{ex3}
\end{figure}
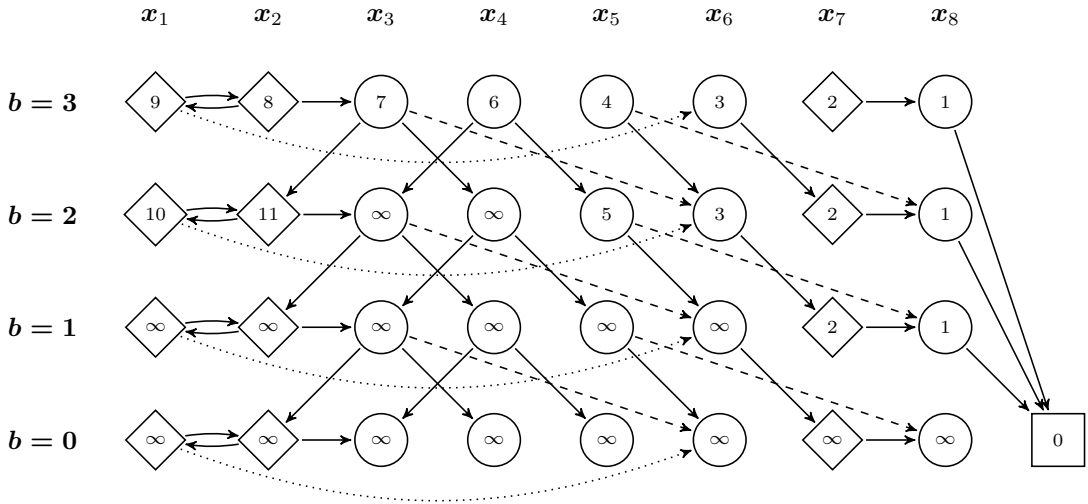

Since for large $B$ the expanded graph is significantly
bigger than the original one,
it is useful to consider ``fast'' techniques for pruning $\widehat{X}$.
We note that the $W_i^0 = \infty$ for all $\x_i \in \nsafe$ and so the
``0-budget'' copies of $\nsafe$ nodes can always be safely removed from
the extended graph.  A more careful argument can be used to reduce the
computational domain much further:

\iffullversion
\begin{remark}
\label{rem:DP_domain_restrict}
\fi
To begin, we describe the resource-optimal paths by defining another value function
$V_i = V(\x_i)$ to be the minimum resource budget $b$ needed to reach the target from
$\x_i$.
For the described model, $c_{ij} \geq 0$ and it is easy to show that
this new ``resource value function'' would have to satisfy
\begin{equation}
\label{eq:secondary}
V_i = \min\limits_{\x_j \in N_i} \left\{ c_{ij} + V_j \right\};
\qquad i \in \I = \{1, \ldots, M\}.
\end{equation}
We will further define $\tilde{V}_i$ as
the minimum resource budget sufficient to follow a primary-optimal
path from $\x_i$.  The equation for $\tilde{V}_i$ is similar to
\eqref{eq:secondary}, but with the minimum taken over the set
of optimal transitions in \eqref{eq:primary}
(instead of minimizing over the entire $N_i$).
Analogously, we can define $\tilde{U}_i$ to be
the primary cost along the resource-optimal trajectories.
Definitions of these four functions are summarized in the caption of
Table \ref{tab1}.  The left side of this table shows the values of
$U, \tilde{V}, V,$ and $\tilde{U}$ for the example presented in Figure \ref{ex1}.

We note that

1) $W_i^b = \tilde{U}_i$ for $b=V_i$.
We will refer to the set of such $\x_i^b$ nodes as the {\em Minimum Feasible Level}
since $W_i^{\beta} = \infty$ for any $\beta<V_i$.

2) $W_i^b = U_i$ for any $b \geq \tilde{V_i}$, since the primary-optimal
trajectories are feasible there.

These two observations can be used to strongly reduce the extended graph
on which we solve the system \eqref{DP:extended}.
E.g., in Figure \ref{ex3}, this pruning excludes all nodes except for
$\{\x_1^3, \x_2^3
\}$.
In examples where $V_i \ll B \ll \tilde{V_i}$ for most nodes,
this remaining subset of $\widehat{X}$ will generally be much larger,
but the pruning is still worthwhile since $U, \tilde{V}, V,$ and
$\tilde{U}$ can be computed on a much smaller (original) graph.
A similar procedure for the continuous case is described in section \ref{ss:mfl}.
\iffullversion
\end{remark}
\fi

\begin{table}[h]
$
\begin{array}{cc}
\text{Without budget resets.}
&
\text{With budget resets.}\\

\begin{tabular}{|c|c|c|c|c|c|c|c|c|}
\hline
\textbf{} & $\x_1$ & $\x_2$ & $\x_3$ & $\x_4$ & $\x_5$ & $\x_6$ & $\x_7$ & $\x_8$\\
\hline
$U$				&	8	&	7	&	6	&	5	&	4	&	3	&	2	&	1\\
$\tilde{V}$		&	5	&	5	&	5	&	4	&	3	&	2	&	1	&	1\\
$V$				&	2	&	2	&	3	&	3	&	2	&	2	&	1	&	1\\
$\tilde{U}$		&	10	&	11	&	7	&	6	&	5	&	3	&	2	&	1\\
\hline
\end{tabular}

&

\begin{tabular}{|c|c|c|c|c|c|c|c|c|}
\hline
\textbf{} & $\x_1$ & $\x_2$ & $\x_3$ & $\x_4$ & $\x_5$ & $\x_6$ & $\x_7$ & $\x_8$\\
\hline
$U$				&	8	&	7	&	6	&	5	&	4	&	3	&	2	&	1\\
$\tilde{V}$		&	-	&	-	&	4	&	3	&	2	&	1	&	-	&	1\\
$V$				&	-	&	-	&	1	&	2	&	2	&	1	&	-	&	1\\
$\tilde{U}$		&	-	&	-	&	9	&	10	&	4	&	3	&	-	&	1\\
\hline
\end{tabular}
\end{array}
$
\label{tab1}
\vspace*{2mm}
\caption{Various ``optimal'' costs for the example in Figure \ref{ex1} with $B=3$.
Primary-optimal cost $(U)$, resource-optimal cost $(V)$,
resource cost along primary-optimal path $(\tilde{V})$,
and primary cost along resource-optimal path $(\tilde{U})$ for
the no-reset and budget-reset problems.
Note that for the latter problem we can also define
$\tilde{V}(\x_3) = \infty$ since $B  = 3$.}

\end{table}

\FloatBarrier

\subsection{Safe/Unsafe splitting with budget resets.}
\label{ss:DSP_with_reset}

Our main focus, however, is on problems where the budget
is also fully reset upon entering $\safe.$
To model that, we can formally set $c_{ij} = -B < 0$,
ensuring the transition from $\x^b_i$ to
$\x_j^B$, whenever $\x_i \in \safe$.
Since we always have a maximum budget in $\safe$,
there is no need to maintain a copy of it in each $b$-slice.
This results in a smaller expanded graph with
$\left( |\safe| + B|\nsafe| + 1 \right)$ nodes;  see Figure \ref{ex5}.
Unfortunately, the negative secondary costs make this problem
implicitly causal, impacting the efficiency.
Moreover, unlike the no-resets case, the projections of
constrained-optimal trajectories onto the original graph
may contain loops.  E.g., starting from $\x_4$ with
the resource budget $b=2$, the constrained-optimal
trajectory is $(\x_4, \x_3, \x_2, \x_3, \x_6, \x_7, \x_8, \T)$
and the first two transitions are needed to restore the budget
to the level sufficient for the rest of the path. Note that,
if we were to re-solve the same problem with $B=4$, the
constrained-optimal path from the same starting location and budget
would be quite different:
$(\x_4, \x_3, \x_2, \x_3, \x_4, \x_5, \x_6, \x_7, \x_8, \T)$,
resulting in a smaller $W_4^2 = 9$; see Remark \ref{rem:partial_solutions}.

\begin{figure}[hhhh]
\centering
\begin{tikzpicture}[scale = .75,->,>=stealth',shorten >=1pt,shorten <=1pt,auto,node distance = 1.5cm,semithick]

    	\node (X1) at (0,7.5) {$\x_1$};
    	\node (X2) at (2,7.5) {$\x_2$};
    	\node (X3) at (4,7.5) {$\x_3$};
    	\node (X4) at (6,7.5) {$\x_4$};
    	\node (X5) at (8,7.5) {$\x_5$};
    	\node (X6) at (10,7.5) {$\x_6$};
    	\node (X7) at (12,7.5) {$\x_7$};
    	\node (X8) at (14,7.5) {$\x_8$};

		\node (b3) at (-2,6) {$\bm{b=3}$};
    	\node[Spoint] (X13) at (0,6) {$\scriptstyle{9}$};
    	\node[Spoint] (X23) at (2,6) {$\scriptstyle{8}$};
    	\node[Upoint] (X33) at (4,6) {$\scriptstyle{7}$};
    	\node[Upoint] (X43) at (6,6) {$\scriptstyle{\bm{5}}$};
    	\node[Upoint] (X53) at (8,6) {$\scriptstyle{4}$};
    	\node[Upoint] (X63) at (10,6) {$\scriptstyle{3}$};
    	\node[Spoint] (X73) at (12,6) {$\scriptstyle{2}$};
    	\node[Upoint] (X83) at (14,6) {$\scriptstyle{1}$};
	
		\node (b2) at (-2,4) {$\bm{b=2}$};
    	\node[Upoint] (X32) at (4,4) {$\scriptstyle{\bm{7}}$};
    	\node[Upoint] (X42) at (6,4) {$\scriptstyle{\bm{10}}$};
    	\node[Upoint] (X52) at (8,4) {$\scriptstyle{\bm{4}}$};
    	\node[Upoint] (X62) at (10,4) {$\scriptstyle{3}$};
    	\node[Upoint] (X82) at (14,4) {$\scriptstyle{1}$};
	
		\node (b1) at (-2,2) {$\bm{b=1}$};
    	\node[Upoint] (X31) at (4,2) {$\scriptstyle{\bm{9}}$};
    	\node[Upoint] (X41) at (6,2) {$\scriptstyle{\infty}$};
    	\node[Upoint] (X51) at (8,2) {$\scriptstyle{\infty}$};
    	\node[Upoint] (X61) at (10,2) {$\scriptstyle{\bm{3}}$};
    	\node[Upoint] (X81) at (14,2) {$\scriptstyle{1}$};

    	\node[Tpoint] (T) at (16, 2) {$\scriptstyle{0}$};

    	\path (X13) edge [bend left=8] (X23)
				    edge [bend right=20, dotted] (X63);

    	\path (X23) edge [bend left=8] (X13);

    	\path (X23) edge [bend left=8] (X33);

    	\path (X33) edge [bend left=8] (X23)
					edge (X42)
					edge [dashed] (X62);
    	\path (X32) edge (X23)
					edge (X41)
					edge [dashed] (X61);
    	\path (X31) edge (X23);

    	\path (X43) edge (X32)
					edge (X52);
    	\path (X42) edge (X31)
					edge (X51);

    	\path (X53) edge (X62)
					edge [dashed] (X82);
    	\path (X52) edge (X61)
					edge [dashed] (X81);

    	\path (X63) edge (X73);
    	\path (X62) edge (X73);
    	\path (X61) edge (X73);

    	\path (X73) edge (X83);

    	\path (X83) edge (T);
    	\path (X82) edge (T);
    	\path (X81) edge (T);

\end{tikzpicture}
\caption{The problem with budget resets for $B=3$.
(The $b=0$ slice is omitted simply because $W_i^0 = \infty$ for any $\x_i \in \nsafe$.)
Implicit/monotone causality: run Dijkstra's algorithm {\bf on the entire expanded graph}.
The number inside each node represents the minimum cost from $\x_i^b$ to the target.
We show in bold font the values different from the ``no resets'' case.}
\label{ex5}
\end{figure}
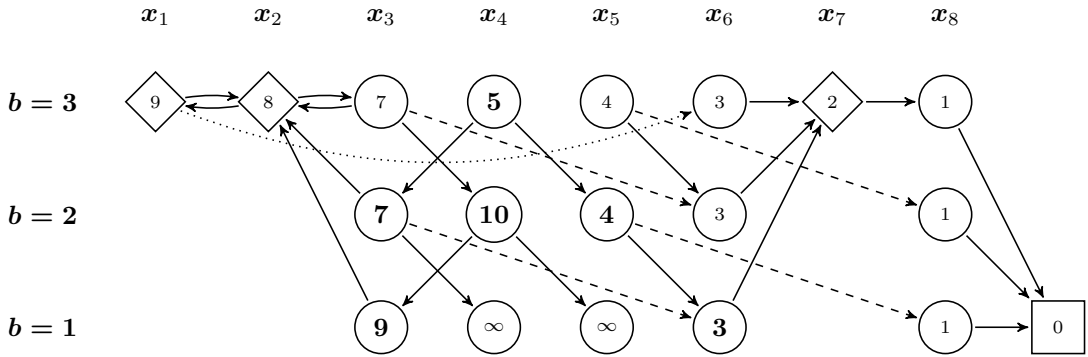

Unfortunately, the domain restriction/pruning techniques outlined in
\iffullversion
Remark \ref{rem:DP_domain_restrict}
\fi
are no longer directly applicable.
For the current example, the values of functions $U$, $\tilde{V}$, $V$, and $\tilde{U}$
are provided in the right half of Table \ref{tab1},
but their interpretations (and algorithms needed to compute them) are more subtle.
E.g., the function $V_i$ is naturally interpreted as the minimum starting
budget sufficient to reach $\T$ starting from $\x_i$.
Thus, $V_i$ is not defined on ``safe'' nodes (since on $\safe$ the budget is always equal to $B$),
and for $\x_i \in \nsafe$ the value of $V_i$ is $B$-dependent.
If $\bar{b} = V_i$, then $Y^{\bar{b}}(\x_i) \neq \emptyset$
(i.e., there exists a $\bar{b}$-feasible path from $\x_i$),
but
$Y^b(\x_i) = \emptyset$ for all $b < \bar{b}$.
If $\y = (\y_0=\x_i, \ldots, \y_r = \T) \in Y^{\bar{b}}(\x_i)$ and
$s$ is the smallest index such that $\y_s
\in \safe$, then
$\bar{b} = \sum_{k=0}^{s-1} c(\y_k,\y_{k+1})$.
As a result, we could also interpret $V_i$ as the minimum ``resource cost''
of traveling from $\x_i$ through $\nsafe$ to reach
$\hat{\safe} = \{ \x_j \in \safe \mid W_j^B < \infty \}.$
Unfortunately, this definition does not allow for efficient computations
on the original graph, since $\hat{\safe}$ is not known a priori.

If the values $W_j^B$ were already known for
all the ``Safe'' nodes $\x_j \in \safe,$ we could efficiently compute
$V_i$ for all $\x_i \in \nsafe$ (thus restricting the computational domain);
moreover, all $W_i^b$ could be then computed in a single upward sweep
(from the slice $b=1$ to the slice $b=B$).
This observation serves as a basis for
the iterative method summarized in Algorithm \ref{alg:DSP_iter}.
Intermediate stages of this algorithm applied to the above example
are depicted in Figure \ref{ex5_iterative}.

\begin{algorithm}
\textbf{Initialization:}\\
\BlankLine
$W_{\T} := 0;$\\
$W_j^B := \infty, \qquad \Forall \x_j \in \safe;$\\
$W_i^b := \infty, \qquad \Forall \x_i \in \nsafe, \, b \in \B.$\\
\BlankLine
Compute $U_i$ for all $\x_i \in X$.\\	

\vspace{0.4cm}

\textbf{Main Loop:}\\
\BlankLine
\Repeat{all $W_j^B$'s stop changing}{
\BlankLine
\textbf{Phase I:}\\
	Using the current $W_j^B$ values to specify $\hat{\safe}$,
	compute $V_i$ and $\tilde{U}_i$ for each $\x_i \in \nsafe$.\\

	\BlankLine
	\BlankLine

	\ForEach{$b = 1, \ldots, B$} {

		\ForEach{$\x_i \in \nsafe$} {
			
			\BlankLine
			\If{$b \geq V_i$} {
				\BlankLine
				\eIf{$b = V_i$} {
					\BlankLine
					$W_i^b := \tilde{U}_i$;
					\BlankLine
				}{
					\eIf{$W_i^{b-1} = U_i$} {
						\BlankLine
						$W_i^b := U_i$;
						\BlankLine
					}{
						\BlankLine
						Compute $W_i^b$ from equation \ref{DP:extended}.\\
						\BlankLine
					}
				}
			}
			\BlankLine

		}
	}

	\BlankLine
	\BlankLine
	\BlankLine
	\textbf{Phase II:}\\
	Run Dijkstra's method {\bf on $\safe$ only} to (re)-compute $W_j^B$'s
	for all $\x_j \in \safe$.\\
	\BlankLine

}
\BlankLine
\BlankLine
\caption{Iterative algorithm for the DSP problem with budget resets.}
\label{alg:DSP_iter}
\end{algorithm}

Since the primary costs $C_{ij}$'s are positive, the value function can still be
computed non-iteratively  -- by using Dijkstra's method on the full expanded graph
without attempting any domain restriction techniques.
It is possible to find examples, on which the advantages of restricting the domain
outweigh the non-iterative nature of Dijkstra's method.
But our main reason for describing the iterative Algorithm \ref{alg:DSP_iter}
is its applicability to the problems of
\iffullversion
sections \ref{s:SSP} and
\fi
\ref{s:continuous_reset}, where
Dijkstra-like methods are generally inapplicable.

\begin{figure}[hhhh]
\centering
\hrule
\vspace*{3mm}
\raggedright {\bf Initialize:}
$ \qquad W_1^3 = W_2^3 = W_7^3 = \infty.$
\vspace*{3mm}
\hrule
\vspace*{3mm}

\raggedright {\bf Fix ``Safe'' and solve on ``Unsafe'':}
\vspace*{3mm}

\begin{tikzpicture}[scale = .75,->,>=stealth',shorten >=1pt,shorten <=1pt,auto,node distance = 1.5cm,semithick]

    	\node (X1) at (0,7.5) {$\x_1$};
    	\node (X2) at (2,7.5) {$\x_2$};
    	\node (X3) at (4,7.5) {$\x_3$};
    	\node (X4) at (6,7.5) {$\x_4$};
    	\node (X5) at (8,7.5) {$\x_5$};
    	\node (X6) at (10,7.5) {$\x_6$};
    	\node (X7) at (12,7.5) {$\x_7$};
    	\node (X8) at (14,7.5) {$\x_8$};

		\node (b3) at (-2,6) {$\bm{b=3}$};
    	\node[Spoint] (X13) at (0,6) {$\scriptstyle{\infty}$};
    	\node[Spoint] (X23) at (2,6) {$\scriptstyle{\infty}$};
    	\node[Upoint] (X33) at (4,6) {$\scriptstyle{\infty}$};
    	\node[Upoint] (X43) at (6,6) {$\scriptstyle{6}$};
    	\node[Upoint] (X53) at (8,6) {$\scriptstyle{5}$};
    	\node[Upoint] (X63) at (10,6) {$\scriptstyle{\infty}$};
    	\node[Spoint] (X73) at (12,6) {$\scriptstyle{\infty}$};
    	\node[Upoint] (X83) at (14,6) {$\scriptstyle{1}$};
	
		\node (b2) at (-2,4) {$\bm{b=2}$};
    	\node[Upoint] (X32) at (4,4) {$\scriptstyle{\infty}$};
    	\node[Upoint] (X42) at (6,4) {$\scriptstyle{\infty}$};
    	\node[Upoint] (X52) at (8,4) {$\scriptstyle{5}$};
    	\node[Upoint] (X62) at (10,4) {$\scriptstyle{\infty}$};
    	\node[Upoint] (X82) at (14,4) {$\scriptstyle{1}$};
	
		\node (b1) at (-2,2) {$\bm{b=1}$};
    	\node[Upoint] (X31) at (4,2) {$\scriptstyle{\infty}$};
    	\node[Upoint] (X41) at (6,2) {$\scriptstyle{\infty}$};
    	\node[Upoint] (X51) at (8,2) {$\scriptstyle{\infty}$};
    	\node[Upoint] (X61) at (10,2) {$\scriptstyle{\infty}$};
    	\node[Upoint] (X81) at (14,2) {$\scriptstyle{1}$};

    	\node[Tpoint] (T) at (16, 2) {$\scriptstyle{0}$};

    	\path (X13) edge [bend left=8] (X23)
				    edge [bend right=20, dotted] (X63);

    	\path (X23) edge [bend left=8] (X13);

    	\path (X23) edge [bend left=8] (X33);

    	\path (X33) edge [bend left=8] (X23)
					edge (X42)
					edge [dashed] (X62);
    	\path (X32) edge (X23)
					edge (X41)
					edge [dashed] (X61);
    	\path (X31) edge (X23);

    	\path (X43) edge (X32)
					edge (X52);
    	\path (X42) edge (X31)
					edge (X51);

    	\path (X53) edge (X62)
					edge [dashed] (X82);
    	\path (X52) edge (X61)
					edge [dashed] (X81);

    	\path (X63) edge (X73);
    	\path (X62) edge (X73);
    	\path (X61) edge (X73);

    	\path (X73) edge (X83);

    	\path (X83) edge (T);
    	\path (X82) edge (T);
    	\path (X81) edge (T);

\end{tikzpicture}
\vspace*{3mm}
\hrule
\vspace*{3mm}
\raggedright {\bf Fix ``Unsafe'' and solve on ``Safe'':}  $\qquad  W_7^3 = 2.$
\vspace*{3mm}
\hrule
\vspace*{3mm}

\raggedright {\bf Fix ``Safe'' and solve on ``Unsafe'':}
\vspace*{3mm}

\begin{tikzpicture}[scale = .75,->,>=stealth',shorten >=1pt,shorten <=1pt,auto,node distance = 1.5cm,semithick]

    	\node (X1) at (0,7.5) {$\x_1$};
    	\node (X2) at (2,7.5) {$\x_2$};
    	\node (X3) at (4,7.5) {$\x_3$};
    	\node (X4) at (6,7.5) {$\x_4$};
    	\node (X5) at (8,7.5) {$\x_5$};
    	\node (X6) at (10,7.5) {$\x_6$};
    	\node (X7) at (12,7.5) {$\x_7$};
    	\node (X8) at (14,7.5) {$\x_8$};

		\node (b3) at (-2,6) {$\bm{b=3}$};
    	\node[Spoint] (X13) at (0,6) {$\scriptstyle{\infty}$};
    	\node[Spoint] (X23) at (2,6) {$\scriptstyle{\infty}$};
    	\node[Upoint] (X33) at (4,6) {$\scriptstyle{7}$};
    	\node[Upoint] (X43) at (6,6) {$\scriptstyle{5}$};
    	\node[Upoint] (X53) at (8,6) {$\scriptstyle{4}$};
    	\node[Upoint] (X63) at (10,6) {$\scriptstyle{3}$};
    	\node[Spoint] (X73) at (12,6) {$\scriptstyle{2}$};
    	\node[Upoint] (X83) at (14,6) {$\scriptstyle{1}$};
	
		\node (b2) at (-2,4) {$\bm{b=2}$};
    	\node[Upoint] (X32) at (4,4) {$\scriptstyle{7}$};
    	\node[Upoint] (X42) at (6,4) {$\scriptstyle{\infty}$};
    	\node[Upoint] (X52) at (8,4) {$\scriptstyle{4}$};
    	\node[Upoint] (X62) at (10,4) {$\scriptstyle{3}$};
    	\node[Upoint] (X82) at (14,4) {$\scriptstyle{1}$};
	
		\node (b1) at (-2,2) {$\bm{b=1}$};
    	\node[Upoint] (X31) at (4,2) {$\scriptstyle{\infty}$};
    	\node[Upoint] (X41) at (6,2) {$\scriptstyle{\infty}$};
    	\node[Upoint] (X51) at (8,2) {$\scriptstyle{\infty}$};
    	\node[Upoint] (X61) at (10,2) {$\scriptstyle{3}$};
    	\node[Upoint] (X81) at (14,2) {$\scriptstyle{1}$};

    	\node[Tpoint] (T) at (16, 2) {$\scriptstyle{0}$};

    	\path (X13) edge [bend left=8] (X23)
				    edge [bend right=20, dotted] (X63);

    	\path (X23) edge [bend left=8] (X13);

    	\path (X23) edge [bend left=8] (X33);

    	\path (X33) edge [bend left=8] (X23)
					edge (X42)
					edge [dashed] (X62);
    	\path (X32) edge (X23)
					edge (X41)
					edge [dashed] (X61);
    	\path (X31) edge (X23);

    	\path (X43) edge (X32)
					edge (X52);
    	\path (X42) edge (X31)
					edge (X51);

    	\path (X53) edge (X62)
					edge [dashed] (X82);
    	\path (X52) edge (X61)
					edge [dashed] (X81);

    	\path (X63) edge (X73);
    	\path (X62) edge (X73);
    	\path (X61) edge (X73);

    	\path (X73) edge (X83);

    	\path (X83) edge (T);
    	\path (X82) edge (T);
    	\path (X81) edge (T);

\end{tikzpicture}
\vspace*{3mm}
\hrule
\vspace*{3mm}
\raggedright {\bf Fix ``Unsafe'' and solve on ``Safe'':}  $\qquad W_1^3 = 9; \quad W_2^3 = 8.$
\vspace*{3mm}
\hrule
\vspace*{3mm}

\raggedright {\bf Fix ``Safe'' and solve on ``Unsafe'':}
\vspace*{3mm}

\begin{tikzpicture}[scale = .75,->,>=stealth',shorten >=1pt,shorten <=1pt,auto,node distance = 1.5cm,semithick]

    	\node (X1) at (0,7.5) {$\x_1$};
    	\node (X2) at (2,7.5) {$\x_2$};
    	\node (X3) at (4,7.5) {$\x_3$};
    	\node (X4) at (6,7.5) {$\x_4$};
    	\node (X5) at (8,7.5) {$\x_5$};
    	\node (X6) at (10,7.5) {$\x_6$};
    	\node (X7) at (12,7.5) {$\x_7$};
    	\node (X8) at (14,7.5) {$\x_8$};

		\node (b3) at (-2,6) {$\bm{b=3}$};
    	\node[Spoint] (X13) at (0,6) {$\scriptstyle{9}$};
    	\node[Spoint] (X23) at (2,6) {$\scriptstyle{8}$};
    	\node[Upoint] (X33) at (4,6) {$\scriptstyle{7}$};
    	\node[Upoint] (X43) at (6,6) {$\scriptstyle{5}$};
    	\node[Upoint] (X53) at (8,6) {$\scriptstyle{4}$};
    	\node[Upoint] (X63) at (10,6) {$\scriptstyle{3}$};
    	\node[Spoint] (X73) at (12,6) {$\scriptstyle{2}$};
    	\node[Upoint] (X83) at (14,6) {$\scriptstyle{1}$};
	
		\node (b2) at (-2,4) {$\bm{b=2}$};
    	\node[Upoint] (X32) at (4,4) {$\scriptstyle{7}$};
    	\node[Upoint] (X42) at (6,4) {$\scriptstyle{10}$};
    	\node[Upoint] (X52) at (8,4) {$\scriptstyle{4}$};
    	\node[Upoint] (X62) at (10,4) {$\scriptstyle{3}$};
    	\node[Upoint] (X82) at (14,4) {$\scriptstyle{1}$};
	
		\node (b1) at (-2,2) {$\bm{b=1}$};
    	\node[Upoint] (X31) at (4,2) {$\scriptstyle{9}$};
    	\node[Upoint] (X41) at (6,2) {$\scriptstyle{\infty}$};
    	\node[Upoint] (X51) at (8,2) {$\scriptstyle{\infty}$};
    	\node[Upoint] (X61) at (10,2) {$\scriptstyle{3}$};
    	\node[Upoint] (X81) at (14,2) {$\scriptstyle{1}$};

    	\node[Tpoint] (T) at (16, 2) {$\scriptstyle{0}$};

    	\path (X13) edge [bend left=8] (X23)
				    edge [bend right=20, dotted] (X63);

    	\path (X23) edge [bend left=8] (X13);

    	\path (X23) edge [bend left=8] (X33);

    	\path (X33) edge [bend left=8] (X23)
					edge (X42)
					edge [dashed] (X62);
    	\path (X32) edge (X23)
					edge (X41)
					edge [dashed] (X61);
    	\path (X31) edge (X23);

    	\path (X43) edge (X32)
					edge (X52);
    	\path (X42) edge (X31)
					edge (X51);

    	\path (X53) edge (X62)
					edge [dashed] (X82);
    	\path (X52) edge (X61)
					edge [dashed] (X81);

    	\path (X63) edge (X73);
    	\path (X62) edge (X73);
    	\path (X61) edge (X73);

    	\path (X73) edge (X83);

    	\path (X83) edge (T);
    	\path (X82) edge (T);
    	\path (X81) edge (T);

\end{tikzpicture}
\caption{The problem with budget resets for $B=3$ (specified in Figure \ref{ex5}) solved by an iterative Algorithm \ref{alg:DSP_iter}.
Note that the MFL for $\x_3$ and $\x_4$ is not correctly determined until the last iteration.}
\label{ex5_iterative}
\end{figure}
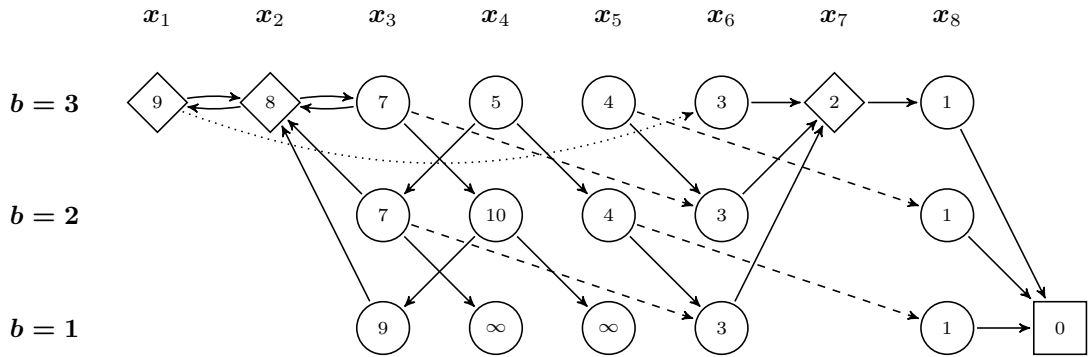


\iffullversion



\section{Stochastic SP on graphs with renewable resources.}
\label{s:SSP}

Stochastic Shortest Path (SSP) problems form an important subclass of
Markov Decision Processes (MDPs) with a range of applications including
manufacturing, resource allocation, and dynamic routing \cite{Bertsekas_DPbook}.
In addition, it is well known that SSPs can be used to obtain
semi-Lagrangian discretizations of continuous optimal control
problems \cite{KushnerDupuis, Tsitsiklis}.
The latter approach has been also used in \cite{VladMSSP} and
the connections will be highlighted in section \ref{ss:Numerics}.

We first briefly describe the usual (single-criterion) SSPs,
closely following the exposition in standard references; e.g., \cite{Bertsekas_DPbook}.
We consider a controlled stochastic process on a previously described
graph with nodes $X$.  A compact set of controlled values
$A_i = A(\x_i)$ is assumed known for every $\x_i \in X$.
A choice of the control $\ba \in \A_i$ determines our
next transition penalty $C_{i}(\ba) = C(\x_i, \ba)$ and
the probability distribution over the set of possible
successor nodes; i.e., $p_{ij}(\ba) = p(\x_i,\x_j,\ba)$ is the
probability of transition to $\x_j$ if the current node is $\x_i$
and we choose the control $\ba$.
Both $p$ and $C$ are usually assumed to be known and continuous in $\ba$.
We will further assume that $C(\x_i, \ba) \geq \Delta > 0.$
The target $\T$ is assumed to be {\em absorbing}; i.e.,
$p_{\T\T}(\ba)=1$ and $C_{\T}(\ba) = 0$ for all $\ba \in A_{\T}.$
The goal is to minimize the expected total cost from each $\x_i$ to $\T$.
More formally, a stationary policy can be defined as a mapping
$\mu: X \rightarrow \bigcup_i A_i$ satisfying $\mu(\x_i) \in A_i$ for all
$\x_i \in X$.  The expected cost of this policy starting from $\x_i$ is
$$
\J(\x_i, \mu) = E \left[
\sum\limits_{k=0}^{+\infty} C(\y_k, \mu(\y_k))
\right],
$$
where $\y = (\y_0 = \x_i, \y_1, \ldots)$ is a random path resulting
from applying $\mu$.  Under the above assumptions, it is easy to show
that there exits an optimal stationary policy $\mu^*$,
whose expected cost is used to define the value function:
$$
U_i = \min_{\mu} \J(\x_i, \mu) = \J(\x_i, \mu^*).
$$
\FloatBarrier
Moreover, the optimality principle yields a non-linear coupled
system of equations
\begin{equation}
\label{eq:SSP_general}
U_i = \min\limits_{\ba \in A_i}
\left\{
C(\x_i, \ba) + \sum\limits_{j=1}^{M+1} p_{ij}(\ba) U_j
\right\}, \qquad
\Forall i \in \I;
\end{equation}
with $U_{\T} = 0$.
The vector $U = (U_1, \ldots, U_M)$ can be viewed as a fixed point
of an operator $T:R^M \rightarrow R^M$ defined componentwise by the right hand side
of equation \eqref{eq:SSP_general}.  This operator is not a contraction,
but under the above assumptions one can show that $T$ will have a unique
fixed point, which can be recovered as a limit of
{\em value iterations} $U^{n+1} = T(U^n)$ regardless of the
initial guess $U^0 \in R^M$; see \cite{BertsTsi}.
However, for a general SSP, an infinite number of iterations may be needed
for convergence while non-iterative (label-setting type) methods generally
do not produce the correct solution to the system \eqref{eq:SSP_general}.
An important practical question is to determine a {\em causal} subclass of SSPs;
i.e., the SSPs on which faster label-setting algorithms can be used legitimately.
The first step in this direction was made by Bertsekas \cite{Bertsekas_DPbook},
who showed that Dijkstra-like methods apply provided there exists
a {\em consistently improving optimal policy}; i.e. an optimal stationary
policy $\mu^*$ such that $p_{ij}(\mu^*(\x_i)) > 0 \; \Rightarrow \; U_i > U_j.$
Unfortunately, this condition is implicit and generally cannot be verified until
the value function is already computed.  Explicit conditions on the cost function
$C$ that guarantee the existence of such a $\mu^*$ have been derived
in \cite{VladMSSP} for a subclass of Multimode SSP problems (MSSPs).
The latter include some of
the SSPs resulting from discretizations of Eikonal PDEs, a perspective first
studied by Tsitsiklis in \cite{Tsitsiklis}
and later extended by Vladimirsky in \cite{VladMSSP}.

While the methods for multi-objective control of
{\em discounted infinite-horizon} MDPs
have been developed at least since 1980s \cite{White_MDP},
to the best of our knowledge, no such methods have
been introduced for multi-objective SSPs so far.  Similarly to the deterministic case,
we will assume that the secondary cost $c_i(\ba)=c(\x_i, \ba)$
is specified for all controls,
that the secondary costs are conveniently quantized (e.g., integers), and that
the goal is to define the value function by minimizing only over those policies,
whose cumulative secondary costs always stay between $0$ and
some positive integer $B$.
More formally, we consider stationary policies on the extended graph
(i.e., $\mu: X \times \mathbb{Z} \rightarrow \bigcup_i A_i$) and
define the restricted set of allowable policies
$$
\M(\x_i, b) = \left\{\mu \mid
P\left( 0 \leq \, b \ominus c(\y_0, \mu(\y_0, b_0)) \ominus \ldots \ominus c(\y_s, \mu(\y_s, b_s))  \, \leq B \right)
\; = \; 1, \quad \Forall s>0 \right\},
$$
where $\y = (\y_0 = \x_i, \y_1, \ldots)$ is a random path and $(b_0 = b, b_1, \ldots)$
is the corresponding sequence of secondary budgets resulting
from applying $\mu$ starting at $(\x_i,b)$.  The value function can be then defined
as $W_i^b = \min_{\mu \in \M(\x_i, b)} \J(\x_i, b, \mu)$ and the optimality principle
yields the system of equations
\begin{equation}
\label{eq:SSP_general_multi}
W_i^b \; = \;  \min\limits_{\ba \in A_i^b}
\left\{
C(\x_i, \ba) + \sum\limits_{j=1}^{M+1} p_{ij}(\ba) W_j^{b \ominus c_i(\ba)}
\right\}; \qquad
b \in [0, B]; \qquad \Forall i \in \I,
\end{equation}
where $A_i^b = \{\ba \in A_i \mid  c_i(\ba) \leq b \},$
and $W_t^b = 0.$

This system can always be solved by the value iterations on the expanded graph
regardless of signs of $c_i(\ba)$'s. But for general SSPs,
this iterative algorithm does not have to converge in a finite number of steps,
if the problem is not causal.
For the subclass of MSSP problems, the criteria developed in \cite{VladMSSP}
can be used on the {\em primary} running cost function $C_i(\ba)$
to determine if a more efficient Dijkstra-like method is applicable;
in the latter case the system is {\em implicitly causal}.

If we assume that all $c_i(\ba) \geq c > 0$, this obviously rules out policies
which do not guarantee terminating in  $B / c$ steps.
Similarly to the deterministic case,
this assumption yields an {\em explicit causality}, allowing to solve the system
\eqref{eq:SSP_general_multi} in a single sweep (from $b=0$ to $b=B$).

Finally, if $c_i(\ba)$'s are known to be non-negative,
$W_i^{b_1}$ may depend on $W_j^{b_2}$ only if $b_1 \geq b_2$.
Assuming the MSSP structure and the absolute causality of each $C_i(\ba)$,
the causality of the full problem is {\em semi-implicit}
and a Dijkstra-like method can be used on each individual $b$-slice of
the extended graph separately.

Returning to modeling the use of resources
in an ``unsafe'' set only, we can achieve this by setting $c_i(\ba) = 0$,
whenever $\x_i \in \safe$ and $c_i(\ba) > 0$ for
$\x_i \in \nsafe$, yielding semi-explicit causality.
On the other hand, to model the budget resets in ``safe'' nodes,
we can instead define $c_i(\ba) = -B < 0$, for all $\x_i \in \safe.$
As in the previous section, the resets make it
unnecessary to store more than one copy
of safe nodes (since $b=B$ on $\safe$).
The resets make it harder to perform domain restriction techniques
similar to those described in Remark \ref{rem:DP_domain_restrict}.
This difficulty can be circumvented by a minor modification of
Algorithm \ref{alg:DSP_iter}.
In section \ref{ss:Numerics} we explain that the SSP interpretation of
the PDE discretization is actually causal on the $\safe$ (and therefore
a Dijkstra-like method is applicable to recomputing the $W_j^B$ values
for $ \x_j \in \safe$); see also Remark \ref{rem:valueiter}.
For the general budget-reset SSPs, Algorithm \ref{alg:DSP_iter}
would have to be changed to use the value iterations on $\safe$ instead.

\fi

\section{Continuous problems with budget reset.}
\label{s:continuous_reset}

\subsection{Continuous optimal control formulation}\label{ss:eikonal}
We begin with a brief review of the basic exit time optimal control
problem
in the absence of resource constraints.
Given an equation for controlled motion in $\domain$, the objective is to characterize
optimal paths (and their associated optimal costs) from every point
in a domain $\domain$ to the boundary $\boundary$.
The static Hamilton-Jacobi-Bellman (HJB)
PDE
discussed in this subsection is fairly
standard and we focus only on the details important for the subsequent discussion.
For a comprehensive treatment of more general
resource-unconstrained problems, we refer interested readers to the
monograph \cite{BardiDolcetta} and the references within.


Consider an open bounded set $\domain\subset\R^n$. Let $A\in\R^n$ be a compact set of
control values and $\A$ the set of admissible controls,
\bq
\A = \{\text{measurable functions }\control\colon\ [0,+\infty) \to A\}.
\eq
Let $\y\colon [0,\infty)\rightarrow \domain$ represent the time-evolution of a particle state
governed by the dynamical system
\bq\label{pathx}
\begin{aligned}
\frac{d \y}{d t}(t)=\dot{\y}(t) &= \fB(\y(t), \control(t)),\\
\y(0) &= \x,
\end{aligned}
\eq
where $\fB\colon\domain\times A\rightarrow \R^n$ represents the velocity.
We refer to all $\y(\cdot)$ that satisfy \eqref{pathx} with admissible controls $\A$ as \emph
{admissible paths}.
For notational simplicity, we have suppressed explicitly writing the dependence of $\y
(\cdot)$ on the control $\control\in\A$ and the initial state $\x\in\domain$.




Define $\runcost\colon\domain\rightarrow\R$ as the running cost and $\term \colon
\boundary\rightarrow \R$ to be the terminal cost when the state reaches a point on the
boundary\footnote{In many applications, the objective is to optimally steer towards a
prescribed closed target set $\mathcal{T}\subset\domain$. This fits into our formulation
by setting the domain to be $\domain\backslash\mathcal{T}$
with the terminal cost
$\term = \infty$ on $\boundary$ and $\term$ finite on $\partial\mathcal{T}$.}.
Thus, the total cost associated with an initial state $\x\in\domain$ and an admissible
control $\control(\cdot)\in\A$ is
\bq\label{costfunc}
\J(\x,\control(\cdot)) = \int^{T}_{0}\runcost(\y(s),\control(s))\, ds+ \term(\y(T)),
\eq
where $T = \min\{t \mid \y(t)\in\boundary\}$.
Then the \emph{value function} $u\colon\domain\rightarrow\R$ is the minimal total cost to
the boundary $\boundary$ starting at $\x$:
\bq\label{valuefunc}
\val(\x)  = \inf_{\control(\cdot)\in\A}\J(\x,\control(\cdot)).
\eq
By Bellman's optimality principle \cite{Bellman_PNAS}, for every sufficiently small $
\tau>0$,
\bq\label{dpp_1}
\val(\x) = \inf_{\control(\cdot) \in\A}\left\{ \int^{\tau}_{0}\runcost(\y(s),\control(s))\, ds +  \val
(\y(\tau))\right\}.
\eq
A Taylor series expansion about $\x$ yields the \emph{Hamilton-Jacobi-Bellman
equation}:
\bq\label{hjbeqn}
\min_{\control\in A}\{\runcost(\x,\control)+\grad \val (\x) \cdot \fB(\x,\control)\} = 0.
\eq
Note that, since the minimization is taken over a compact set $A$, the infimum becomes a
minimum.
From the definition, the boundary condition on the value function is
\bq\label{bc}
\val(\x) = \term(\x), \quad \x\in\partial\domain.
\eq

For the functions $\fB, \runcost, \term$ introduced above, we make the following
assumptions:
\begin{enumerate}[label= (A\arabic{*}), ref=(A\arabic{*})]
\item\label{assumptionLip}
$\fB$ and $\runcost$ are Lipschitz continuous.

\item\label{assumptionFK}
There exist constants $\runcost_1, \runcost_2$, such that $0<\runcost_1\le \runcost(\x,
\control)\le \runcost_2$ for all $\x\in\cdomain$ and $\control\in A$.

\item \label{stc}
The velocity $\fB: \cdomain \times A \mapsto \R^n$ is bounded
(i.e., $\|\fB\| \leq F_2$) and
the motion in every direction is always possible, i.e.,\\
$\hspace*{1cm}
\forall \x \in \cdomain,$ and all unit vectors $\bv \in \R^n$,
$\quad \exists \ba \in A \text{ s.t. }
\bv \cdot \fB(\x, \ba) = \|\fB(\x, \ba)\| \geq F_1> 0.$\\
Moreover, we assume that the
\emph{scaled vectogram} $\{\fB(\x,\control)/\runcost(\x,\control) \mid
\control\in A\}$ is strictly convex at each $\x\in\domain$.

\item\label{assumptionLSC}
$\term$ is lower semi-continuous and $\min_{\boundary}\term < \infty$.

\end{enumerate}
The assumption \ref{stc} yields the \emph{small-time controllability} property \cite
{BardiFalcone} which guarantees the continuity of the value function $\val$ on $\domain$.
In general, even under the aforementioned assumptions, classical solutions to the PDE
\eqref{hjbeqn} usually do not exist, while weak (Lipschitz-continuous) solutions are not
unique.
Additional selection criteria were introduced by Crandall and Lions \cite{CranLion}
to recover the unique weak solution (the so called \emph{viscosity solution}),
coinciding with the value function of the control problem.




Several important classes of examples will repeatedly arise in the
following sections.
We will refer to problems as having \emph{geometric dynamics} if
 \bq\label{geometric_dynamics}
A=\{\control \mid \|\control\|=1\}, \qquad \fB(\x,\control)=\control f(\x,\control).
\eq
For control problems with geometric dynamics and unit running cost ($\runcost(\x,\control)=1$),
the PDE \eqref{hjbeqn} reduces to
\bq\label{eikonal}
\max_{\|\control\|=1}
\left\{ f(\x,\control)\ \left(-\control \cdot\!\grad \val(\x) \right) \right\} = 1.
\eq
This models a particle which travels through the point $\x$ in the
direction $\control\in A$ at the speed $f(\x,\control)$.
For zero terminal cost $\term=0$, the value $\val(\x)$ equals the minimum travel time of
such a particle from $\x$ to $\boundary$.
If we further assume isotropic speed $f(\x,\control) = f(\x)$, \eqref{eikonal} simplifies to the
\emph{Eikonal equation}:
\bq\label{eikonal2}
f(\x)\|\grad \val(\x)\| = 1.
\eq
Finally, for the unit speed 
with $\term=0$, the viscosity solution
to \eqref{eikonal2} is the distance function to $\boundary$.

Under the assumptions \ref{assumptionLip}-\ref{assumptionLSC}, the value function $\val$
can be used to extract the optimal feedback control $\control^*=\control^*(\x)$, provided $
\val$ is smooth at $\x\in\domain$.
For the Eikonal case \eqref{eikonal2}, for instance, it can be shown that $\control^*(\x) = -
\grad \val(\x) / |\grad \val(\x) |$ (note that $\grad \val\ne 0$ under the assumptions \ref
{assumptionFK});
the points where $\grad \val$ does not exist are precisely those where the optimal control is not unique.
Subsequently the optimal trajectory $\y^*(\cdot)$ from $\x\in\domain$ can be computed as
solution to the initial value problem \eqref{pathx} with $\control(t) = \control^*(\y(t))$.
It is easy to show that $\y^*(\cdot)$
coincides with the characteristic curve to $\x$ from the boundary $\boundary$,
but traveling in the opposite direction.
Furthermore,
the optimal control will be unique at $\y^*(t)$
for all $t>0$ until this trajectory reaches $\boundary$.

\subsection{Augmented PDE for the budget constrained problem}\label{ss:bcp}
The problem of budget constrained optimal paths was proposed in \cite{KumarVlad},
where general multi-criterion optimal control problems were reduced to solving an \emph{augmented
PDE} on an expanded state space.
For the purpose of this article, we shall briefly describe this method for the single budget
constraint.

Define the extended space $\edomain = \domain\times(0,\maxb]$,
where $\maxb>0$ is a prescribed maximum resource budget.
We shall call a point in $(\x,\bud)\in\edomain$ an \emph{extended state}.
Here, $\bud$ represents the current resource budget.
We represent a path parametrized by time $t\ge 0$ in the extended domain as
$\z(t) = (\y(t),\budp(t))\in\edomain$.

Next, define the secondary cost $\runcostb \colon \domain\times A \rightarrow \R$, strictly
positive, which is the decrease rate of the budget.
We shall assume that $\runcostb$ is Lipschitz continuous and there exist constants
$\runcostb_1, \runcostb_2$ such that
\bq\label{Khat_assump}
0 <\runcostb_1\le \runcostb(\x,\control) \le \runcostb_2,\quad\forall \, \x\in\domain,
\control\in A.
\eq
Then the equations of motion in $\edomain$ are
\bq\label{pathxbud}
\begin{aligned}
\dot{\y}(t) &= \fB(\y(t), \control(t)), &\y(0) = \x, \\
\dot{\budp}(t) &= -\runcostb(\y(t),\control(t)),&\budp(0) = \bud.
\end{aligned}
\eq
For an arbitrary control $\control\in\A$, define the terminal time starting at the extended
state $(\x,\bud)$ as
\[
\hat T(\x, \bud, \control)  =
\min \{t \mid \y(t)\in\partial\domain,\,\budp(t) \ge 0\}.
\]
Define the set of all \emph{feasible} controls for paths starting at the extended state $(\x,
\bud)$ as
\[
\hat\A(\x,\bud) = \{\control\in\A \mid \hat T(\x, \bud, \control)<\infty\}.
\]
We shall call paths corresponding to feasible controls as feasible paths.
Then the value function is
\[
\vale(\x,\bud) = \inf_{\control(\cdot)\in\hat\A(\x,\bud)}\J(\x,\control(\cdot)).
\]
The boundary condition is
\bq\label{bc_bc}
\vale(\x,\bud) =
\begin{cases}
\term(\x) &  \x\in\boundary, \bud\in[0,\maxb]\\
\infty & \x\in\domain, \, \bud= 0.
\end{cases}
\eq
Note that $\vale(\x,\bud)\ge \val(\x)$, since in the unconstrained case the set of controls is larger.
Moreover, if the starting budget $\bud$ is insufficient to reach the boundary,
the set $\hat\A(\x,\bud)$ is empty and $\vale(\x,\bud) = \infty$.
Wherever the value function is finite on $\edomain$,
Bellman's optimality principle similarly yields the HJB equation
\bq\label{hjb_aug}
\min_{\control\in A}\left\{\grad_{\x} \vale \cdot \fB(\x,\control) - \runcostb(\x,\control)\frac
{\partial \vale}{\partial \bud} + \runcost(\x,\control) \right\} = 0.
\eq
Here $\grad_{\x}$ denotes the vector-valued operator of partial derivatives with respect to
$\x$.
For the case of isotropic velocity and unit running cost $\runcost=1$, and $\runcostb(\x,
\control) = \runcostb(\x)$, the latter equation reduces to
\bq\label{isotropic_bc}
f(\x)|\grad_{\x} \vale| + \runcostb(\x)\frac{\partial \vale}{\partial \bud} = 1.
\eq

The (augmented) velocity function $(\fB(\x,\control),-\runcostb(\x,\control))$ in
\eqref{pathxbud} no longer satisfies the assumption \ref{stc}.
This implies the lack of small time controllability in $\edomain$ and
possibly discontinuities in the corresponding value function $\vale$.
While the original definition of viscosity solutions required continuity,
a less restrictive notion of \emph{discontinuous viscosity solutions} applies in
this case; see \cite[Chapter 5]{BardiDolcetta}, \cite{Soravia_IntegralConstraint},
\cite{MottaRampazzo} and the references therein.

The assumption \eqref{Khat_assump} on $\runcostb$ implies that the characteristics
move in strictly monotonically increasing $\bud$-direction (see also Property \ref
{prop_mono}).
Thus, the value function $\vale$ is \emph{explicitly causal} in $\bud$:
the value function at a fixed point $\x'\in\domain$, $\bud'\in(0,\maxb]$ depends only on the
value function in $\{(\x,\bud)\mid \bud < \bud'\}$.
Moreover, since $\runcostb > 0$, by rearranging the PDE \eqref{isotropic_bc}
\[
\frac{\partial \vale}{\partial \bud} = \frac{1}{\runcostb(\x)}\left( 1- f(\x)|\grad_{\x} \vale|\right),
\]
the problem can be viewed as a ``time-dependent'' Hamilton-Jacobi equation, where $
\bud$ represents the ``time.''
In \cite{KumarVlad}, this explicit causality property of \eqref{hjb_aug} yielded an efficient
numerical scheme involving a single ``sweep'' in the increasing $\bud$-direction (see
section \ref{sec:oopath}).
We note the parallelism with the discrete problem formulation,
described in case 2 of Remark \ref{rem:types_of_causality}.

\begin{remark}\label{rem:semi-implicit}
Before moving on to the budget reset problem, we briefly discuss a slight generalization to
the budget constrained problem in \cite{KumarVlad}.
Suppose we
relax the lower bound in \eqref{Khat_assump} to allow $\runcostb=0$ on some closed subset
$\safe\subset\cdomain$.
The set $\safe$ can be interpreted as a ``safe'' set, on which the resources are not used.
(In this setting, the problem previously considered in
\cite{KumarVlad} corresponds to $\nsafe = \domain$.)
Clearly, explicit causality no longer holds when $\safe \cap \domain \neq \emptyset$;
rather, $\vale$ has a \emph{semi-implicit causality property}:
for any fixed point $\x'\in\domain$, $\bud'\in(0,\maxb]$, the value function
$\vale(\x', \bud')$ can only depend on the values
$\{\vale(\x,\bud) \mid \x\in\cdomain, \, \bud \in (0, \bud'] \}$.
This possible interdependence between values on the same $\bud$ level
makes it impossible to transform the PDE into a time-dependent Hamilton-Jacobi equation
(as with the explicit causality case).
Instead, the value function satisfies a separate Eikonal equation on the $\text{int}(\safe)$ part of each $\bud$-slice:
\bq\label{decoupleEikonal}
\min_{\control\in A}\{\grad_{\x} \vale(\x,\bud)\cdot \fB(\x,\control) + \runcost(\x,\control) \} =
0,\qquad
\forall \bud\in[0,\maxb], \x\in\text{int}(\safe).
\eq
We again note the parallelism with the discrete formulation, described in case 3 of
Remark \ref{rem:types_of_causality}; see also Figures \ref{ex2} and \ref{ex3}.
\end{remark}

\subsection{A PDE formulation for the budget reset problem}\label{ss:br_pde}
We now consider a continuous analog of the problem
in section \ref{ss:DSP_with_reset} and
Figure \ref{ex5}.
Partition the closure of the domain $\bar{\domain} = \nsafe\cup\safe$, where $\nsafe$ is
open (thus, $\boundary\subset\safe$).
Here, $\nsafe$ represents the ``unsafe'' set, where the budget decreases at some
Lipschitz continuous rate $\runcostb$,
\bq\label{Khat_assump2}
0 <\runcostb_1\le \runcostb(\x,\control) \le \runcostb_2,\quad\forall \, \x\in\nsafe,
\control\in A.
\eq
and $\safe$ represents the ``safe'' set, where the budget resets to its maximum.
We denote the interface between the two sets in the interior of the domain as
$\interface = \partial \nsafe \cap \domain.$
We will further assume that the set $\safe \, \backslash \,
\{\x \in \boundary \mid q(\x) < \infty \}$
consists of finitely many connected components.

To model the budget reset property, the budget $\budp(t)$ is set to $\maxb$ in $\safe$.
Thus, equations \eqref{pathxbud} are modified accordingly in $\safe$, and the resulting
dynamics is described by a {\em hybrid system}:

\begin{align}
&\begin{cases}
\dot{\y}(t) &= \fB(\y(t), \control(t))\\
\dot{\budp}(t) &= -\runcostb(\y(t),\control(t))
\end{cases}
& \text{ if }\y(t)\in\nsafe, \label{pathreset_U}
\\ &\begin{cases}
\dot{\y}(t) &= \fB(\y(t), \control(t))\\
\budp(t) &= B
\end{cases}
& \text{ if }\y(t)\in\safe,\label{pathreset_S}
\end{align}
with initial states:
\bq\label{pathresetinit_S}
\begin{aligned}
\budp(0) &=
\begin{cases}
\bud & \text{if } \y(0)\in\nsafe\\
\maxb & \text{if } \y(0)\in\safe.
\end{cases}
\end{aligned}
\eq


Similarly to
section \ref{ss:bcp}, we define the terminal time for a control $
\control\in\A$ as
\[
\tilde T (\x, \bud, \control) = \min\{t \mid \y(t)\in\boundary, \budp(\tau)\ge 0 \text{ for all }\tau\in[0,t]\}
\]
where the paths $\y(\cdot)$ satisfy the equations \eqref{pathreset_U}-\eqref
{pathresetinit_S}.
Let $\tilde\A(\x,\bud)=\{\control\in\A \mid \tilde T(\x, \bud, \control)<\infty\}$ to be the set of all
feasible controls for the budget reset problem.
Then the value function is defined as
\[
\valr(\x,b) = \inf_{\control(\cdot)\in\tilde\A(\x,\bud)}\J(\x,\control(\cdot)).
\]

\iffullversion
\begin{remark}
A simple, illustrative example is a budget reset problem, in which the objective is to reach
a target in minimum time ($\runcost=1$, $\term = 0$) while avoiding a prolonged
continuous tour in $\nsafe$ of more than time $\maxb>0$ (hence, $\runcostb = 1$).
\end{remark}
\fi

For the budget reset problem, a path $\z(\cdot) =(\y(\cdot),\budp(\cdot))$ in the extended
domain never visits points $\{(\x,\bud)\mid \x\in\safe, b<B\}$.
Thus, we consider the \emph{reduced extended domain}
\[
\rdomain = \udomain \cup \sdomain \; \subset \edomain,
\]
where,
\[
\begin{aligned}
\udomain &= \nsafe\times (0,B]\\
\sdomain &= \safe\times\{B\}.
\end{aligned}
\]
This reduction of the extended state space is analogous to the discrete example described
in Figure~\ref{ex5}.

The corresponding HJB equation
for this {\em hybrid control problem}
in $\rdomain$ can be described separately in $\udomain
$ and $\sdomain$. To distinguish between the value functions in these two domains we
write:
\[
\valr(\x,\bud) \equiv
\begin{cases}
\valu(\x,\bud) & (\x,\bud) \in \domain_{\nsafe}\\
\vals(\x) & \x \in \safe
.
\end{cases}
\]
As in the no-resets case, the value function $w$ is usually infinite on parts of $\rdomain$
(if the starting budget is insufficient to reach the boundary even with resets);
see also section \ref{ss:mfl}.
Bellman's optimality principle can be used to show that, wherever $w$ is finite,
it should satisfy the following PDEs:
\begin{align} \label{augPDE1}
\min_{\control\in A}\left\{\grad \valu(\x,\bud) \cdot \fB(\x,\control)- \runcostb(\x,\control)\frac
{\del \valu}{\del\bud}(\x,\bud)+ \runcost(\x,\control)\right\}  &= 0, &(\x,\bud)\in\udomain,\\
\label{augPDE2}
\min_{\control\in A}\left\{\grad \vals(\x) \cdot \fB(\x,\control)+ \runcost(\x,\control)\right\} &=
0, & \x\in\text{int}(\safe).
\end{align}
We note the similarity between the previously defined PDE \eqref{decoupleEikonal}
and the last equation \eqref{augPDE2}.
But while the latter is solved on one $\bud$-slice only (for $\bud=\maxb$),
the former is actually a $\bud$-parametrized family of static PDEs,
which are coupled to each other only through the boundary with $\udomain$.
This difference is a direct consequence of the budget reset property.

The boundary conditions on $\valr$ are:
\bq\label{br_bc}
\valr(\x,\bud) =
\begin{cases}
\term(\x) & \x\in\boundary\subset\safe\\
\infty & \x\in\domain, \bud = 0.
\end{cases}
\eq
In addition,
the following compatibility condition
(motivated below by Proposition \ref{prop_cont})
holds between $\valu$ and $\vals$, wherever $\valr$ is finite on $\interface = \partial \nsafe \cap \domain$ :
\bq\label{comp_1}
\lim_{\x'\rightarrow\x\atop\x'\in\nsafe}\valu(\x',\bud) = \vals(\x),\quad\text{for all}\,
\x \in\interface, \bud \in (0,\maxb].
\eq

We now list several properties of the value function $\valr$.
The proofs of the first two of them are omitted for brevity.

\begin{property}\label{prop_charac}
In $\domain_\nsafe$, the $\bud$-coordinate is monotonically decreasing along every trajectory when traversing forward in time.
\end{property}

\begin{property}\label{prop_mono}
If $\bud_1 \le \bud_2\le\maxb$ then $\valu(\x,\bud_1) \ge \valu(\x,\bud_2)$ for all $\x\in
\nsafe$.
\end{property}

From here on we will use
$\safe_\reach = \{\x\in\safe \mid \vals(\x)<\infty\}$
to denote the ``reachable'' part of the safe set.
We note that every connected component of
$\safe \, \backslash \, \{\x \in \boundary \mid q(\x) < \infty \}$
is either fully in $\safe_\reach$ or does not intersect it at all.
As a result, the number of connected components in the latter set is also finite.

\begin{lemma}\label{cont_lemma}
$\vals$ is locally Lipschitz continuous on $\safe_\reach\backslash\boundary$.
\end{lemma}
\begin{proof}
We will use the notation $B_\epsilon(\x) = \{ \x' \in \R^n \mid \|\x-\x'\| < \epsilon \}.$\\
Choose a point $\x\in\safe_\reach\backslash\boundary$,
and an $\epsilon>0$ sufficiently small so that $B_\epsilon(\x)\subset\domain$ and
\bq\label{feas_cond1}
\frac{\epsilon}{F_1} < \frac{\maxb}{\runcostb_2}
.
\eq
Take any $\x'\in B_\epsilon(\x)\cap(\safe_\reach\backslash\boundary)$.
Then, a straight line path from $\x$ to $\x'$ will take at most $\|\x-\x'\|/F_1$ time to
traverse.
Furthermore, by \eqref{feas_cond1}, such a path must be feasible.
Thus, Bellman's optimality principle gives
$
\vals(\x') \le  \runcost_2 / F_1\|\x-\x'\| + \vals(\x).
$
By swapping the roles of $\x$ and $\x'$ we have $|\vals(\x')-\vals(\x)| \le  \runcost_2 / F_1\|
\x-\x'\|$, as desired.
\end{proof}



A consequence of Lemma \ref{cont_lemma} is that $\vals$ is bounded on each
connected component of $\safe_\reach$, excluding $\boundary$.
Since $\safe_\reach$ consists of finitely many connected components,
$\vals^{\max} = \sup\limits_{\x\in\safereach\backslash\boundary}\vals(\x)$
is also finite.

\begin{lemma}\label{bound_lemma}
For a
point $\x\in\nsafe$ such that $\valu(\x,\maxb)<\infty$,
let $\y^*$ be an optimal feasible path from $\x$ to $\boundary$.
Let $T(\x)$ be the first arrival time of $\y^*$ to $\boundary$.
Then,
\bq\label{bound_T}
T(\x) \le  \frac{\maxb}{\runcostb_1}+\frac{\vals^{\max}}{\runcost_{1}}.
\eq

\end{lemma}
\begin{proof}
Let $t^*> 0$ be the first instance in time such that $\y^*(t^*)\in\safe$.
We subdivide $\y^*$ into two segments, the first from $\x$ to $\y^*(t^*)$ and the second
from $\y^*(t^*)$ to the final point $\y^*(T(\x))\in\boundary$.
The time taken to traverse the first segment is bounded by $\maxb/\runcostb_1$ to satisfy
the budget feasibility constraint, and the time on the second segment is bounded by $\vals^{\max}/
\runcost_1$.
\end{proof}

\begin{prop}\label{prop_cont}
The compatibility condition \eqref{comp_1}, holds on $\partial\safe_{\reach}\backslash
\boundary$ for $\bud\in(0,\maxb)$.
\end{prop}
\begin{proof}

Fix a point $\x\in\partial\safe_\reach\backslash\boundary$. 
Choose an $\epsilon_1>0$ sufficiently small so that $B_{\epsilon_1}(\x)\subset\domain$
and
$
\epsilon_1 < {\bud F_1}/{\runcostb_2}.
$
If we choose $\x'\in\nsafe$ such that $\|\x'-\x\|<\epsilon_1$, by arguments similar to the
proof of Lemma~\ref{cont_lemma}, the straight line path from $\x'$ to $\x$ is feasible.
Thus, by Bellman's optimality principle,
\bq\label{goal_1}
\valu(\x',\bud) \le \vals(\x) + \epsilon_1 \runcost_2/F_1,\quad \text{for $\bud\in(0,\maxb]$.}
\eq

Suppose now that $\bud \in (0,B)$.
Choose $\epsilon_2>0$ so that $B_{\epsilon_2}(\x)\subset\domain$ and
$
\epsilon_2 < \frac{(\maxb-\bud) F_1}{\runcostb_2}.
$
Assume $\x'\in\nsafe$ and $\|\x-\x'\|<\epsilon_2$, and consider the straight line path from $
\x$ to $\x'$.
Suppose the resource cost and time required to traverse this path is $\bud'$ and $\tau$,
respectively.
Since $\tau \le \epsilon_2/F_1$, we have by the choice of $\epsilon_2$,
$
\quad \bud' \le \tau \runcostb_2 \le\left(\epsilon_2/F_1\right)\runcostb_2 < \maxb - \bud,
$
or equivalently, $\quad \bud < \maxb - \bud'$.\\
Thus, by Bellman's optimality principle,
\bq\label{goal_2}
\begin{aligned}
\vals(\x)&\le \valu(\x',\maxb - \bud')  + \epsilon_2 \runcost_2/F_1\\
&\le \valu(\x',\bud)  + \epsilon_2 \runcost_2/F_1
\quad \text{for $\bud\in(0,\maxb)$,}
\end{aligned}
\eq
where the second inequality follows from Proposition~\ref{prop_mono}.
Therefore, \eqref{goal_1} and \eqref{goal_2} imply that if $\|\x'-\x\|<\epsilon=\min
\{\epsilon_1,\epsilon_2\}$ then $|\valu(\x',\bud) - \vals(\x)|\le\epsilon\runcost_2/F_1$ for $
\x'\in\nsafe, \bud\in(0,\maxb)$; this proves the compatibility condition \eqref{comp_1} for $
\bud\in(0,\maxb)$.
\end{proof}

\begin{prop}\label{prop_cont2}
Assume the running costs and the dynamics are isotropic; i.e.,
$\runcost(\x,\control)=\runcost(\x)$, $\runcostb(\x,\control) =\runcostb(\x),$
and $\fB(\x,\control) = \control f(\x), \|\control\|=1$.
Then the compatibility condition \eqref{comp_1} holds on $\partial\safe_{\reach}
\backslash \boundary$ for $\bud =\maxb$.
\end{prop}

\begin{proof}
Fix $\x\in\partial\safe_{\reach}\backslash \boundary$.
We prove that $\vals(\x) - \lim\limits_{\x'\rightarrow\x}\valu(\x',\maxb) \le 0$.
Note that \eqref{comp_1}
follows, since the proof of \eqref{goal_1}
also covered the current case ($\bud=\maxb$).

Take a sequence $\{ \x_j \} \subset \nsafe$ converging to $\x$, such that $\valu(\x_j,\maxb)$ is finite
for each $j$.
Suppose for each $j$ that $\control^*_j\in\tilde\A(\x_j,\maxb)$ is an optimal feasible control
and $\y_j^*$ is the corresponding feasible path such that $T_j = \min\{t \mid \y_j^*(t)\in
\boundary\}$.
Let $T_{\max} = \sup_j\{T_j\}$ and for each $j$, set $\y^*_j(t) = \y^*_j(T_j)$
for $t\in[T_j,T_{\max}]$.
Note that by Lemma~\ref{bound_lemma}, $T_{\max}$ is finite.
Also, condition \ref{stc} yields uniform boundedness and equi-continuity of
the paths $\y_j^*$.
Therefore, the Arzela-Ascoli theorem ensures that (upon reindexing in $j$) a subsequence
$\y^*_{j}$ converges uniformly to a path $\y^*$ in $\cdomain$ corresponding to some
admissible control $\control^*\in\A$.

Next we show that $\control^*\in\tilde\A(\x',\maxb)$, i.e. $\y^*$ is a feasible path.
Define $T_j^*$ (and $T^*$) to be the first
instance
in time
when $\y^*_j$
(and, respectively, $ \y^*$) reaches $\partial\safe_\reach$.
Let $\tilde T$ be the infimum limit of the sequence $\{T^*_j\}$ and consider its subsequence
such that, upon reindexing in $j$, $\tilde T = \lim\limits_{j\to \infty} T^*_j$.
Since $\partial\safe_\reach$ is closed,
$
\y^*\left(\tilde T\right) = \lim\limits_{j\rightarrow \infty}\y^*_j(T^*_j)
\in\partial\safe_\reach,
$
and thus $\tilde T \ge T^*$.
This implies that
$
\lim\limits_{j\to \infty}\int_{T^*}^{T^*_j} \runcostb(\y_j^*(s)) ds \ge 0.
$
Using this observation and the Lipschitz continuity of $\runcostb$, we have,
\[
\begin{aligned}
\int_0^{T^*} \runcostb(\y^*(s)) ds - \int_0^{T_j^*} \runcostb(\y_j^*(s)) ds
&= \int_0^{T^*} \runcostb(\y^*(s)) - \runcostb(\y_j^*(s)) ds - \int_{T^*}^{T_j^*} \runcostb
(\y_j^*(s)) ds \\
&\le \int_0^{T^*} \left|\runcostb(\y^*(s)) - \runcostb(\y_j^*(s))\right| ds - \int_{T^*}^{T_j^*}
\runcostb(\y_j^*(s)) ds\\
&\le L  \int_0^{T^*} \left\|\y^*(s) - \y_j^*(s)\right\| ds - \int_{T^*}^{T_j^*} \runcostb(\y_j^*(s))
ds,
\end{aligned}
\]
where $L$ is the Lipschitz constant for $\runcostb(\cdot)$.
Since the last expression has a non-positive limit as $j\to\infty$
and each $\y_j^*$ is feasible, this shows that on the interval $[0,T^*]$
the trajectory $\y^*$ is feasible as well:
\[
\int_0^{T^*} \runcostb(\y^*(s)) ds \le \lim_{j\rightarrow\infty} \int_0^{T_j^*} \runcostb(\y_j^*
(s)) ds \le \maxb.
\]
If $\y^*(t)$ does not remain in $\safe_\reach$ after $t=T^*$, a similar argument can be used
to prove the feasibility of all other ``unsafe segments'' of the trajectory.
(Alternatively, appending the optimal trajectory corresponding to $\y^*(T^*) \in \safe_\reach$
yields another feasible trajectory, which is at least as cheap with regard to
the primary running cost $\runcost$.)

Applying the same argument to the total running cost \eqref{costfunc} and recalling that
$q$ is lower semi-continuous, we obtain
$
\J(\x,\control^*)\le \lim\limits_{j\rightarrow\infty}\valu(\x_{j},\maxb).
$
This completes the proof since, by definition of the value function,
$\vals(\x) \le \J(\x,\control^*).$
\end{proof}

\section{Numerical methods for continuous budget reset problems}\label{ss:Numerics}

Throughout this section, we assume the following setup: let $\grid$ be a set of
gridpoints on the domain $\bar{\domain}$.
While a similar formulation can be constructed for arbitrary unstructured meshes,
we restrict our description to a uniform Cartesian grid $\grid$ with grid spacing $\h$.
We denote $\grid_\nsafe = \grid\cap\nsafe$ and $\grid_\safe = \grid\cap\safe$.
To simplify the treatment of boundary conditions, we assume that
$\boundary$ and $\interface$ are well-discretized by the gridpoints
in $\del\grid\subset\grid$ and in $\partial \grid_\safe$ respectively.
We also assume that the set of allowable budgets $[0,\maxb]$,
is discretized into equispaced intervals partitioned by gridpoints
$\Bd =  \{ \bud_j = j \Delta\bud \mid j=0,1,\dots,\Nb\}$,
where $\Delta\bud >0$ is fixed ahead of time.

\subsection{Discretization of the unconstrained case.}\label{ss:hjb_semilag}

To begin, we briefly discuss the usual semi-Lagrangian discretization techniques
for static HJB equations of the form~\eqref{hjbeqn}.
Denote $\Valn(\x)$ to be the numerical approximation to $\val(\x)$
at the gridpoint $\x\in\grid$.
Suppose the current system state is $\x \in \grid \cap \domain$
and the constant control value $\control$ is to be used for a short time $\tau>0$.
Assuming that $K$ and $\fB$ are locally constant, the new position is approximated by
$\x_{\control}(\tau) = \x + \tau\fB(\x,\control)$ and the approximate accumulated
transition cost is $\runcost(\x,\control) \tau$.
For small $\tau$, this yields a first-order 
{\em semi-Lagrangian} discretization
\bq\label{semilag1}
\begin{aligned}
\Valn(\x) &= \min_{\control\in A}\{\tau \runcost(\x,\control) + \Valn(\x_{\control}(\tau)) \},
&\text{for all }\x\in\grid \backslash \del\grid\\
\Valn(\x) &= \term(\x),&\text{for all } \x\in\del\grid,
\end{aligned}
\eq
of Bellman's optimality principle \eqref{dpp_1}.
Since $\x_{\control}(\tau)$ is usually not a gridpoint, $\Valn(\x_{\control}(\tau))$ needs
to be interpolated using adjacent grid values.
Many different variants of the above scheme result from different choices of $\tau$.
Falcone and co-authors have extensively studied the discretized systems\footnote{Some
of the papers cited in this subsection have considered related but slightly different PDEs,
including those for finite-horizon and infinite-horizon optimal control,
but the fundamental idea remains the same.}
which use the same $\tau>0$ for all $\x$ and $\control$;
see \cite{BardiFalcone, Falc_Dial, Falcone_InfHor} and
higher-order accurate versions in \cite{FalconeFerretti}.
Alternatively, $\tau$ can be chosen for each $\control$
to ensure that $\x_{\control}(\tau)$ lies on some
pre-specified set near $\x$.
For example, in a version considered by Gonzales and Rofman \cite{GonzalezRofman},
the motion continues in the chosen direction until reaching
the boundary of an adjacent simplex.  E.g., on a Cartesian grid in $\R^2$,
if $\x_s$ and  $\x_w$ are two gridpoints adjacent to $\x$
and $\fB(\x,\control)$ is some
southwest-ward direction of motion, then $\tau$ is chosen to ensure that
$\x_{\control}(\tau)$ lies on a segment $\x_s \x_w$, and
$\Valn(\x_{\control}(\tau))$ is approximated by a linear interpolation between
$\Valn(\x_s)$ and $\Valn(\x_w)$.
Interestingly, in the case of geometric dynamics \eqref{geometric_dynamics},
this type of semi-Lagrangian scheme is also equivalent to
the usual Eulerian (upwind finite difference) discretization.
A detailed discussion of this for isotropic problems on grids
can be found in \cite{Tsitsiklis} and for general anisotropic
problems on grids and meshes in \cite[Appendix]{SethVlad3}.
In addition, both types of semi-Lagrangian schemes
can be viewed as controlled Markov processes on $\grid$
\iffullversion
(indeed, as SSP problems, discussed in section \ref{s:SSP})
\fi
;
this earlier approach was pioneered by Kushner and Dupuis in
\cite{KushnerDupuis}; see also a more recent discussion in \cite{VladMSSP}
on applicability of label-setting algorithms.

We note that \eqref{semilag1} is a large coupled system of non-linear equations.
If $\Psi$ is an upper bound on $\Valn$,
this system can be in principle solved by fixed point iterations
starting from
an initial guess $\Valn = \Psi$ on $\grid \backslash \del\grid.$
However, this approach can be computationally expensive, and
an attractive alternative is to develop a Dijkstra-like non-iterative algorithm.
For the fully isotropic case, two such methods are
Tsitsiklis' Algorithm \cite{Tsitsiklis} and Sethian' Fast Marching Method
\cite{SethFastMarcLeveSet}.
An overview of many (isotropic) extensions of this approach can be found
in \cite{SethSIAM}.  Ordered Upwind Methods \cite{SethVlad2,SethVlad3}
have similarly handled the anisotropic case; a recent efficient
modification was introduced in \cite{AltonMitchell_MAOUM}.
Similar fast methods were introduced for several related PDEs
by Falcone and collaborators \cite{CristianiFalcone_FMM_Games, CCF, CFF, CFFM}.

An alternative approach is to speed up the convergence of iterative
solver for \eqref{semilag1} by using Gauss-Seidel iterations with
an alternating ordering of gridpoints.  Such ``Fast Sweeping'' algorithms
\cite{BoueDupuis,TsaiChengOsherZhao, Zhao, KaoOsherQian} 
are particularly efficient when the direction of characteristics does not change
too often.  A comparison of various non-iterative and fast-iterative
approaches (as well as a discussion of more recent hybrid algorithms)
can be found in \cite{ChaconVlad_FMSM}.

We emphasize that, for the purposes of this paper,
any one of the above approaches can be used modularly,
whenever we need to solve equation \eqref{augPDE2}.
The discretization of \eqref{augPDE1} is explained 
in subsection \ref{sec:oopath}.
But before dealing with these technical details, subsection \ref{ss:iterative_brp} addresses the main computational challenge of budget reset problems: the a priori unknown boundary condition on $\interface$, which results in an implicit interdependence of gridpoints in $\domain_{\nsafe}$ and $\domain_{\safe}$.



\subsection{Iterative treatment of the budget reset problem}\label{ss:iterative_brp}

First, we note that the simpler case of $\domain=\nsafe$ (i.e.,
the constrained optimal control problem presented in section~\ref{ss:bcp})
can be solved by a single upward sweep in the $\bud$-direction, as described in \cite{KumarVlad}.
This is a direct consequence of the explicit
causality property of the value function when $\runcostb$ is strictly positive on $\domain$.
Moreover, without budget resets, relaxing $\runcostb$ to be non-negative (i.e.
introducing a safe subset $\safe$ where $\runcostb=0$)
still yields semi-implicit causality; see remark \ref{rem:semi-implicit}.

In contrast, the introduction of budget resets on a safe subset $\safe\subset
\domain$, breaks this causal structure. If the values on $\interface$ were a priori known,
we could efficiently solve \eqref{augPDE2} on $\safe$ by either
Marching or Sweeping techniques at least in the case of geometric dynamics.
But since the values on $\interface$ are not provided, there are no known
non-iterative algorithms to numerically solve
this problem on $\rdomain$.
Therefore, we propose to solve the PDEs \eqref{augPDE1} and \eqref{augPDE2} (with
boundary conditions and the compatibility condition \eqref{comp_1}) by an alternating
iterative process.
We construct a sequence of functions $\valu^k$ and $\vals^k$ for $k=0,1,2,\dots$, which
converge to $\valu$ and $\vals$ respectively, as $k \to \infty$.

We begin with a recursive definition for these new functions on
$\domain_{\nsafe}$ and $\domain_{\safe}$.
Of course, the actual implementation in section \ref{ss:mfl} relies on their numerical approximations;
the resulting method is summarized in Algorithm \ref{alg:budgetReset}.
Initially, set
\[
\begin{aligned}
&\valu^0(\x,\bud) = \infty &(\x,\bud) \in \udomain,\\
&\vals^0(\x) = \infty  &\x\in\safe.\\
\end{aligned}
\]
Then for $k = 1,2,\dots$,
\begin{description}

\item[Phase I]
Find $\valu^{k}$ as the viscosity solution of equation \eqref{augPDE1}
with boundary conditions
\bq\label{w1_bc}
\valu^{k}(\x,\bud) =
\begin{cases}
\term(\x) & \x\in\boundary, \bud\in(0,\maxb]\\
\liminf\limits_{\x'\rightarrow\x\atop\x'\in\safe}
\vals^{k-1}(\x') & \x\in\interface,  \bud\in(0,\maxb]\\
\infty & \x \in \nsafe, \bud = 0.
\end{cases}
\eq

\item[Phase II]
Find $\vals^{k}$ as the viscosity solution of equation \eqref{augPDE2}
with boundary conditions
\bq\label{w2_bc}
\vals^{k}(\x) =
\begin{cases}
\term(\x) & \x \in \boundary\\
\liminf\limits_{\x'\rightarrow\x\atop\x'\in\nsafe}
\valu^{k}(\x',\maxb)&  \x \in\interface.
\end{cases}
\eq

\end{description}

\iffullversion
\begin{remark}\label{rem:gamma_interp}
\fi
We note that the $\liminf$'s in the above definition are primarily for the sake
of notational consistency (since solving \eqref{augPDE2} on $\text{int}(\safe)$
does not really specify $w_2^k$'s values on $\interface \subset \del \safe$.
Alternatively,
we can solve the PDE on $\safe$, treating boundary conditions on $\interface$
``in the viscosity sense'' \cite{BardiDolcetta}.
This is essentially the approach used in our numerical implementation.
\iffullversion
\end{remark}
\fi


Intuitively, $\valu^k$ and $\vals^k$ can be interpreted
as answering the same question as $\valu$ and $\vals$,
but with an additional constraint that no trajectory is allowed to
reset the budget (by crossing from $\nsafe$ to $\safe$) more than $(k-1)$ times.
As a result, for many problems convergence is attained (i.e., $\valu^k = \valu$
and $\vals^k = \vals$) after a finite number of recursive steps.
E.g., in the simplest case where $\runcost$ and $f$ are constant on $\domain$,
$\runcostb>0$ is constant on $\nsafe$, and all connected components of
$\safe \backslash \boundary$ are convex, then
any optimal trajectory might enter each connected component at most once.
See table \ref{tab:convergence} for the experimental confirmation of this phenomenon.

\subsection{Discretization of $\vals^k$, $\valu^k$}\label{sec:oopath}

Let $\Valu^k(\x,\bud_j)$ be an approximation of
$\valu^k(\x,\bud_j)$ for all $(\x,\bud_j)\in \grid_{\nsafe} \times\Bd$;
similarly, let $\Vals^k(\x)$ be an approximation of
$\vals^k(\x)$ for all $\x \in \grid_{\safe}$.
The ``natural'' boundary conditions are implemented as follows:
\begin{align}\label{num_bc1}
\Valu^k(\x,\bud_j) &= \term(\x), & \x\in\del\grid\cap\bar{\nsafe}, \bud_j\in\Bd, \\ \label
{num_bc2}
\Vals^k(\x) &= \term(\x), & \x\in\del\grid\cap\safe.
\end{align}
Additional boundary conditions on $\interface$ stem from the recursive definition of
$\vals^k$ and $\valu^k$ (yielding the compatibility condition \eqref{comp_1} in the limit).
In Phase I, we use
\bq
\label{num_c1}
\Valu^k(\x,\bud_j) = \Vals^{k-1}(\x), \quad \x\in\grid_\safe, \bud_j\in\Bd,
\eq
and then solve the discretized system \eqref{semiLag_b} on the relevant
subset of $\grid_\nsafe \times \Bd$; see the discussion below and in section
\ref{ss:mfl}.

In Phase II, the numerical compatibility condition is enforced
on the set of gridpoints\\
$\grid^{\interface}_{\nsafe} =
\{ \x \in \grid_{\nsafe} \mid \x \text{ is adjacent to some } \x' \in \safe \}.$
We set
\bq
\label{num_c2}
\Vals^k(\x) = \Valu^{k}(\x,\maxb), \quad \x\in \grid^{\interface}_{\nsafe},
\eq
and then recover $\Vals^k$
by solving the system of equations equivalent to \eqref{semilag1}
on the entire $\grid_\safe$ (including on $\grid \cap \interface$).
As explained in subsection \ref{ss:hjb_semilag}, this can
be accomplished by many different efficient numerical algorithms.



To derive the system of equations defining $\Valu^k$ on $\grid_\nsafe \times \Bd$,
we adapt the approach introduced in \cite{KumarVlad}.
Property \ref{prop_charac} is fundamental for the method's efficiency:
the characteristic curves emanating from $\del \safereach$ all move in
increasing direction in $b$.
Thus, we only need a single `upward' sweep in the $b$ direction to capture the value
function along the characteristics.
We exploit this result in the semi-Lagrangian framework as follows.
For $\x\in\grid\cap\nsafe$, $\bud_j\in\Bd$, $\tau>0$, write
$\x_{\control}(\tau) = \x + \tau\fB(\x,\control)$ and
$\bud_{\control}(\tau) = \bud_j - \tau\runcostb(\x,\control)$.
If we choose
$
\tau=\tau_{\control,\x} =
(\Delta\bud) / \runcostb(\x,\control),
$
this ensures that $\bud_{\control}(\tau) = \bud_{j-1}$, and
the semi-Lagrangian scheme at $(\x, \bud_j)$ becomes
\bq\label{semiLag_b}
\Valu^k(\x,\bud_j) = \min_{\control\in A}\left\{ \tau \runcost(\x,\control) + \Valu^k(\x_
{\control}(\tau), \bud_{j-1})\right\}.
\eq
For each $j=1,2,\dots$, we solve \eqref{semiLag_b} for
$\Valu^k(\x,\bud_j)$, at each $\x\in\grid_\nsafe$
based on the (already computed) $\Valu^k$ values
in the $b_{j-1}$ resource-level.

For an arbitrary control value $\control\in A$, the point $\x_{\control}(\tau)$ usually
is not a gridpoint; so,
$\Valu^k(\x_{\control}(\tau), \bud_{j-1})$ has to be approximated using values at the nearby
gridpoints (some of which may be in $\grid_{\safe}$).
Since our approximation of $\x_{\control}(\tau)$ is first-order accurate,
it is also natural to use a first-order approximation for its $\Valu^k$ value.
Our implementation relies on a bilinear interpolation. E.g., for $n=2$,
suppose the gridpoints $\x_1,\x_2,\x_3,\x_4\in\grid$ are the four corners of the grid cell
containing $\x_{\control}(\tau)$,
ordered counter-clockwise with $\x_1$ on the bottom left corner.
If $(\gamma_1,\gamma_2) = (\x_{\control}(\tau) - \x_1)/h$, then the bilinear
interpolation yields
\[
\begin{aligned}
\Valu^k(\x_{\control}(\tau),\bud_{j-1}) &=
\gamma_1\left( \gamma_2 \mathcal{\Val}^k(\x_3,\bud_{j-1}) +
(1-\gamma_2) \mathcal{\Val}^k(\x_2, \bud_{j-1}) \right) \\
&+ (1-\gamma_1)\left(\gamma_2 \mathcal{\Val}^k(\x_4,\bud_{j-1}) +
(1-\gamma_2) \mathcal{\Val}^k(\x_1,\bud_{j-1})\right),
\end{aligned}
\]
where $\mathcal{\Val}^k(\x,\bud) = \Valu^k(\x,\bud)$ if $\x \in \grid_{\nsafe}$
and $\mathcal{\Val}^k(\x,\bud) = \Vals^{k-1}(\x)$ if $\x \in \grid_{\safe}$.
The resulting approximation is inherently continuous, while $\valu$ usually is not.
The issue of convergence of semi-Lagrangian schemes to discontinuous viscosity solutions is
discussed in Remark \ref{rem:convergence}.

\iffullversion
\begin{remark}\label{rem:valueiter}
The direct discretization of equations \eqref{augPDE1} and \eqref{augPDE2}
could also be interpreted as defining the value function of the corresponding
SSP problem on $(\grid_{\nsafe} \times\Bd) \cup \grid_{\safe}.$
In this framework, the iterative algorithm presented in this section can be viewed
as a Gauss-Seidel relaxation of value iterations on $\grid_{\nsafe} \times\Bd$
alternating with a Dijkstra-like method used on $\grid_{\safe}.$
\end{remark}
\fi

\subsection{Domain restriction techniques}\label{ss:mfl}

The value function $\valr(\x,\bud)$ is infinite at points that are not reachable within budget
$\bud$.
Since the dimension ($n+1$) of the domain where $\valu$ is solved is typically large,
a reduction of the computational domain to its reachable subset usually yields
substantial saving in computational time.
In \cite{KumarVlad} such a reduction was achieved by an efficient preprocessing step, which
involved computing the \emph{minimum feasible level}, i.e. the interface that
separates the reachable and unreachable parts of $\edomain$.
Here we use a similar technique
to find the ``lower'' boundary of the reachable subset of $\domain_\nsafe$.

Note that, by Property \ref{prop_mono}, for any $\x\in\nsafe$,
$\valu(\x,\bud_1) < \infty$ implies $\valu(\x,\bud) < \infty$ for all $\bud\in[\bud_1,\maxb]$;
and $\valu(\x,\bud_2) = \infty$ implies $\valu(\x,\bud) = \infty$ for all $\bud\in[0, \bud_2]$.
We formally define the minimum feasible level (MFL)
in the unsafe set as the graph of
\bq\label{mfl_defn}
\mfl(\x) = \mfl[\valu](\x) =\min\{\bud \mid \valu(\x,\bud) < \infty\},\quad \x\in\nsafe,
\eq
and in the safe set $\safe$, as a graph of
\bq\label{mfl_bc2}
\mfl(\x) =
\begin{cases}
0 & \x\in\safereach\\
\infty & \x\in\safe\backslash\safereach.
\end{cases}
\eq

The goal is to recover the MFL from some cheaper (lower-dimensional) computation.
We note that on $\nsafe$, $v(\x)$ can be interpreted as the value function
of another resource-optimal control problem, and as such it can be recovered
as the viscosity solution of a similar HJB equation
\bq\label{mfl_hjb}
\min_{\control\in A}\left\{ \runcostb(\x,\control) + \grad \mfl (\x) \cdot \fB(\x,\control) \right\} =
0,\quad \x \in\nsafe,
\eq
coupled with the boundary conditions \eqref{mfl_bc2}.
We note that
\iffullversion
the algorithm described in Appendix can be used to identify $\safe_\reach$
by an iterative process on $\cdomain$ (rather than on the higher dimensional $\udomain \cup \sdomain$).
\fi
So, in principle,
$\mfl (\x)$ can be fully computed without increasing the dimensionality of
$\cdomain$.
However, to use the MFL as the ``lowest budget'' boundary condition for \eqref{augPDE1},
we also need to know the values of $\valu$ on the MFL.
This corresponds to the ``constrained optimal'' cost $\mflv$ along resource-optimal
trajectories defined by $\mfl$; see Table \ref{tab1} for a similar example in the discrete setting.
The function $\mflv$ is formally defined below; here we note that it can also be
computed in the process of solving \eqref{mfl_hjb} provided $\vals$ is
a priori known on $\partial\safe_\reach$.
This is, indeed, the case when no resets are possible; i.e.,
$\nsafe=\domain$ and $\safe_\reach \subset \safe = \boundary$,
precisely the setting previously considered in \cite{KumarVlad}.
Unfortunately, for the general case ($\nsafe\ne\domain$),
we do not know of any (fast, lower-dimensional) algorithm to compute
$\vals$ on $\partial\safe_\reach$.
(Note that in a similar discrete example depicted in Figure \ref{ex5_iterative},
the values of the safe nodes continue changing until the last iteration.)
Instead, we proceed to recover the values on the MFL iteratively, using values of $\vals^k$
on $\partial\safe_\reach$.

To describe this iterative process, we first define the $k$-th approximation of the reachable
subset of $\safe$:
$ \qquad
\safe_\reach^k = \{\x\in\safe \mid \vals^k(\x) < \infty\}.
$
Then, $\text{MFL}^k$, the $k$-th approximation of the MFL, is a graph of
$
\mfl^k(\x) = \mfl[\valu^k](\x),
$
which
can be computed by solving \eqref{mfl_hjb} with boundary conditions $\eqref{mfl_bc2}$
where $\safe_\reach$ is replaced by $\safe_\reach^{k-1}$.
The numerical approximation $\Mfl^k(\x)$ can be efficiently computed
using the methods discussed in section \ref{ss:hjb_semilag}.

Once $\mfl^k$ is known, we can define the subset of $\nsafe$ that is reachable at the $k$-th
iteration as\\
$
\nsafe_\reach^k = \left\{\x \in\nsafe \mid \mfl^k(\x) \le \maxb\right\},
$
and
a function $\mflv^k\colon\nsafe_\reach^k\rightarrow \R$ as
\bq\label{defn_uk}
\mflv^k(\x) =
\begin{cases}
\valu^k(\x,\mfl^k(\x)),& \x\in\nsafe_\reach^k, \\
\vals^{k-1}(\x), & \x\in\partial\nsafe.
\end{cases}
\eq
Since we intend to use $\mflv^k$ as a ``lower boundary'' condition for $\valu^k$
, we must compute $\mflv^k$ using only the information derived
from $\vals^{k-1}$, already computed at that stage of the algorithm.
For this purpose, it is possible to represent $\mflv^k$ as a value function of another
related control problem on $\nsafe_\reach^k$.

Let $T= T(\x, \bud, \control) = \min\{t \mid \y(t)\in\safereach^{k-1}\}$.
Define $\tilde\A^k(\x)$ to be the set of all ``$\mfl^k$-optimal'' controls; i.e., the
controls which lead from $\x$ to $\safereach^{k-1}$ through $\nsafe_\reach^k$
using exactly $\mfl^k(\x)$ in resources.
For most starting positions $\x \in \nsafe_\reach^k$, this set will be a singleton,
but if multiple
feasible controls are available, their corresponding primary costs can be quite different.
Then $\mflv^k$ can be characterized as
\bq\label{dpp_prop}
\mflv^k(\x) =
\inf_{\control(\cdot)\in\tilde\A^k(\x) } \int_0^{T}\runcost(\y(t),\control(t))\ dt +
\vals^{k-1}(\y(T)).
\eq
By Bellman's optimality principle, \eqref{dpp_prop} yields
\bq\label{oo_bellman}
\mflv^k(\x) = \lim_{\tau\rightarrow 0^+}\min_{\control\in A^k(\x)}\{\tau \runcost(\x,\control) +
\mflv^k(\x + \tau\fB(\x,\control))\},
\eq
where $A^k(\x)\subset A$ is the set of minimizing control values in \eqref{mfl_hjb}.
If $\mflV^k(\x)$ is the approximation to $\mflv^k(\x)$ at a gridpoint $\x\in\grid_\nsafe$,
a natural semi-Lagrangian scheme based on \eqref{oo_bellman} is
\bq\label{semiLagV}
\mflV^k(\x) = \min_{\control\in A^k(\x)}\{\tau \runcost(\x,\control)
+ \mflV^k(\x + \tau\fB(\x,\control))\},
\qquad \x \in \nsafe_\reach^k,
\eq
where $\mflV^k(\x + \tau\fB(\x,\control))$ is interpolated, and the corresponding boundary
condition is
\bq\label{semiLagV_bc}
\mflV^k(\x) = \Vals^{k-1}(\x)\quad \x\in\grid \cap \safe_\reach^{k-1}.
\eq
Since the set $A^k(\x)$ has to be found when solving \eqref{mfl_hjb},
it is also natural to solve \eqref{semiLagV} at each gridpoint as soon as its
$\Mfl^k$ becomes available.


As discussed above, $\mflV^k$ acts as a numerical boundary condition on the surface $\bud=
\Mfl^k(\x)$ for the update scheme \eqref{semiLag_b}.
However, in general, $\Mfl^k(\x)\not\in\Bd$.
In our implementation, we set $\Valu^k(\x,\bud_j) = \mflV^k(\x)$, where $j$ is the
smallest integer such that $\Mfl^k(\x)\le j\Delta\bud$. This introduces additional
$O(\Delta b)$ initialization errors at the MFL.  An alternative (more accurate)
approach would require using cut-cells when interpolating near the MFL.

\noindent
The resulting iterative method is summarized in Algorithm~\ref{alg:budgetReset}.
\begin{algorithm}

\textbf{Initialization:}\\
$\Valu^0(\x,\bud) := \infty,\qquad \Forall\x\in\grid_\nsafe, \bud\in\Bd; $\\
$\Vals^0(\x) := \infty, \qquad \Forall\x\in\grid_\safe \backslash \boundary;$\\
$\Vals^0(\x) := q(\x), \qquad \Forall\x\in\boundary;$\\
\vspace{0.2cm}
Compute $\Valn(\x)$ for all $\x\in\grid$;
\vspace{0.4cm}

\textbf{Main Loop:}\\
\BlankLine
\ForEach{$k=1,2,\dots$ until $\Valu^k$ and $\Vals^k$ stop changing}{
\vspace{0.2cm}

Using $\Vals^{k-1}$ to specify $\grid\cap\safe_{\reach}^{k-1}$,
	compute $\Mfl^k$ and $\mflV^k$ for each $\x \in \grid_\nsafe$.\\

\vspace{0.2cm}

\textbf{Phase I:}\\
\ForEach{$\bud = \Delta\bud,\ 2\Delta\bud,\dots,\maxb$}{
	\BlankLine
	\ForEach{$\x \in \grid_\nsafe$}{
		\BlankLine
		\If{$b \geq \Mfl^k(\x)$} {
			\BlankLine
			\eIf{$b < (\Mfl^k(\x) + \Delta\bud)$} {
				\BlankLine
				$\Valu^k(\x,\bud) := \mflV^k(\x)$;
				\BlankLine
			}{
				\eIf{$\Valu^k(\x,\bud-\Delta\bud) = \Valn(\x)$} {
					\BlankLine
					$\Valu^k(\x,\bud) := \Valn(\x)$;
			    	\BlankLine
				}{
					\BlankLine
					Compute $\Valu^k(\x,\bud)$ from equation \ref{semiLag_b};
					\BlankLine
				}
			}	
		
		}
	}

}

\vspace{0.4cm}

\textbf{Phase II:}\\
Compute $\Vals^k$ on $\grid_\safe$;\\

}

\caption{The budget reset problem algorithm.}
\label{alg:budgetReset}
\end{algorithm}

\noindent
We note that the following properties
are easy to prove inductively using the comparison
principle \cite{BardiDolcetta} on the PDEs \eqref{augPDE1}, \eqref{augPDE2},
and \eqref{mfl_hjb}.


\begin{prop}\label{prop:monotone}
The iterative method
is monotone in the following sense:\\
\begin{enumerate}
\item
$\valr^{k+1}(\x,\bud) \le \valr^{k}(\x,\bud)$ for each $(\x,\bud)\in\rdomain$ and
$k=0,1,2,\dots$.
\item
$\safereach^{k}\subseteq\safereach^{k+1},$ for each $k=1,2,\dots$
\item
$\mfl^{k+1}(\x) \le \mfl^{k}(\x)$, for each $\x\in\domain$ and each $k=1,2,\dots$
\end{enumerate}
\end{prop}

\begin{remark}\label{rem:convergence}
We briefly discuss the convergence of $\Vals$ and $\Valu$ to $\vals$ and $\valu$,
respectively, as $\h, \Delta\bud\rightarrow 0$.
Under the listed assumptions,
the value function $\vals$ is Lipschitz continuous on $\text{int}(\safe)$
and
can be approximated by the methods described in section \ref{ss:hjb_semilag};
these approximations will induce errors depending on $\h$ only,
since $\Vals$ is approximated only in the top slice $\bud=\maxb$.
For example, standard Eulerian type (first-order upwind finite difference)
discretizations are $O(\h)$ accurate,
provided the solution has no rarefaction-fan-type singularities on the
influx part of the boundary.
(The latter issue is illustrated by the convergence test in section
\ref{ss:convergencetest}.)

On the other hand, the value function $\valu$ can be discontinuous on
$\udomain$ and a weaker type of convergence is to be expected as a result.
In the absence of resets (i.e., with
$\nsafe = \domain$),
if we
focus on any compact $\mathcal{K} \subset \udomain$ on which $\valu$ is
continuous, then the semi-Lagrangian scheme \eqref{semiLag_b} has been proven
to converge to $\valu$ on $\mathcal{K}$ uniformly, provided $\h = o(\Delta\bud)$
as $\h, \Delta\bud \rightarrow 0$; see \cite{BardiFalconeSoravia1,BardiFalconeSoravia2}.
To the best of our knowledge, there are no corresponding theoretical results for
convergence to discontinuous viscosity solutions of hybrid control problems
(e.g., equations \eqref{augPDE1}).  Nevertheless, the numerical evidence strongly
supports the convergence of described schemes (section \ref{ss:convergencetest}).
In \cite{KumarVlad} it was empirically demonstrated that without resets
the $L_1$-norm convergence (or $L_\infty$-norm convergence away from discontinuities)
can often be attained even with a less restrictive choice of $\h = O(\Delta\bud)$.
In section \ref{ss:convergencetest}, we show that this also holds true even if
budget resets are allowed.

Finally, we note two additional sources of ``lower boundary'' errors:
due to an approximation of the MFL and due to an approximation of $\valu=\tilde{u}$ on it;
the first of these is $O(\Delta \bud)$ while the latter is $O(\h)$.
\end{remark}

The optimal paths in $\rdomain$ can be extracted from the value function
in a manner similar to the description in section
\ref{ss:eikonal}.
The only additional computation is in the $\bud$-direction 
for the parts of trajectory in $\domain_{\nsafe}$:
for example, in the isotropic case, the budget $\budp$ along the optimal path
$\y^*$ decreases by $\runcostb(\y^*(t))$.
For each $(\x, b)$, the optimal control value $\control^*$ can be found either from an
approximation of $\nabla_{\x} \Valn$ or by solving the local
optimization problem similar to \eqref{semiLag_b}.
Once $\control^*$ is known,
the system \eqref{pathreset_U} can be integrated forward in time
by any ODE solver (our implementation uses Euler's method).








\section{Numerical Results}
\label{sec:numresults}


For the sake of simplicity, we will assume that
the dynamics is
isotropic
(i.e., $\fB(\x, \control) = \control f(\x)$), the primary running cost is $\runcost \equiv 1$
with the zero exit cost on the target
(making the value function equal to the total time along the optimal trajectory),
and $\runcostb \equiv 1$
(constraining the maximum contiguous time spent in $\nsafe$)
in all of the examples.
In addition to a numerical convergence test of Algorithm \ref{alg:budgetReset} (section \ref{ss:convergencetest}),
we will illustrate the effects of geometries of $\nsafe$, $\safe$,
spatial inhomogeneity of the speed $f$ and different maximum budgets $\maxb$.
For each example we show the level curves of the value
function at the top $\bud$-slice (i.e., $w(\x, \maxb)$).
In subsections \ref{ss:inhomog_tests} and \ref{ss:visibility_tests}, we also
show constrained-optimal paths, starting at some representative point $\x \in \nsafe$
with the maximum starting budget $b = \maxb$.
We emphasize that, for other starting budgets $b < \maxb$, the constrained-optimal paths
can be quite different, but all the data needed to recover them
is also a by-product of Algorithm~\ref{alg:budgetReset}.

The numerical solutions are computed on $\domain=[-1,1]^2$ discretized by a uniform $N\times N$ cartesian grid.
In all examples except for the convergence tests in section \ref{ss:convergencetest},
we use $N=300, \, h = 2/(N-1)$, and the budget direction is discretized with
$
\Delta \bud = \maxb / \text{round}\!\left(\frac{\maxb}{0.8 h}\right);
$
resulting in $N_\bud = |\B| = (\maxb / \Delta \bud) + 1 = O(1/h)$.
The main loop in Algorithm \ref{alg:budgetReset} was terminated when
both $\|\Valu^k-\Valu^{k-1}\|_\infty$ and $\|\Vals^k - \Vals^{k-1}\|_\infty$ fell below the
tolerance threshold of $10^{-8}$.

All tests were performed in Matlab (version R2010b)
with most of the implementation compiled into MEX files.
The tests were conducted on a 2.6 GHz MacBook computer under Mac OS X with 4 GB of RAM.
On average, the computations took approximately one minute of processing time,
but we emphasize that our implementation was not optimized for maximum efficiency.



\subsection{Convergence test}\label{ss:convergencetest}

We test the convergence of Algorithm \ref{alg:budgetReset}
with
$
\safe = \{(x,y)\in\domain\mid x \le 1/3\} \cup \boundary.
$
Assume isotropic, unit speed and costs $\fb(\x,\ba) = \ba$, $\runcost = \runcostb = 1$,
$A = S^1$, with maximum budget $\maxb = 1$.
We consider the case of a point ``target" $\target = (1,0)$ by choosing the boundary conditions
\bq
\valr(x,y,\bud) =
\begin{cases}
0 & (x,y)=\target\\
+\infty & (x,y)\in\boundary\backslash \{\target\}.
\end{cases}
\eq
Note that the problem is symmetric with respect to the $x$ axis; moreover,
an explicit formula for the exact solution can be derived from simple geometric considerations.
To simplify the notation, we define a few useful geometric entities on the domain
(see Figure \ref{fig:conv_toplevel} for a diagram):
\bq
\begin{aligned}
\Line &= \text{the vertical line segment at $x=\frac{1}{3}$ for $-\frac{\sqrt{5}}{3} \le y \le \frac{\sqrt{5}}{3}$}.\\
P_1, P_2 &= \text{the upper and lower end points of $\Line$, respectively.}\\
P(x,y) &= \begin{cases}P_1, &\text{ if $y\geq 0$;}\\  P_2, &\text{ otherwise.} \end{cases}
\end{aligned}
\eq
We shall only describe an optimal (feasible) path from an arbitrary point $\x=(x,y)$ to $\target$,
since $\valr(x,y,\bud)$ is simply the length of that path.
For convenience we will use the notation
``$\x\rightharpoonup\y$'' as a shorthand for ``a (directed) straight line segment from $\x$ to $\y$''.

We begin by describing the optimal path from $\x\in\safe$.
If $\x\rightharpoonup\target$ passes through $\Line$, this line segment is precisely the optimal path.
If the line does not pass through $\Line$, the optimal path is
$\x\rightharpoonup P(\x) \rightharpoonup \target$.
Next, we describe the optimal path from $\x \in \nsafe$ with initial budget $\bud$.
Clearly, if $\x\in\nsafe$ is more than $\bud$ distance from both $x=\frac{1}{3}$ and $\target$,
then $\valr(\x,\bud) = +\infty$.
Also, if $\x$ is within $\bud$ distance from $\target$, the optimal path is $\x\rightharpoonup\target$.
Otherwise, the optimal path will have to first visit $\safe$ (and reset the budget to $\maxb$),
before moving to $\target$.
This situation can be further divided into two cases:
\begin{enumerate}
\item[Case 1.]
if $\x\in\nsafe$ is within $\bud$ distance from $\Line$, the optimal path is to move from
$\x\rightharpoonup\y\rightharpoonup \target$, where $\y$ is a point on $\Line$ such that
$\dist{\x}{\y}\le b$ minimizing $\dist{\x}{\y} + \dist{\y}{\target}$.

\item[Case 2.]
The optimal path is $\x\rightharpoonup\z\rightharpoonup P(\x) \rightharpoonup\target$,
where $\z$ is the closest point on the line $x=\frac{1}{3}$ to $P(\x)$ such that $\dist{\x}{\z}\le b$.
\end{enumerate}
From the above, it should be clear that for each $b>0$ the value function
will have a discontinuous jump on
$\D(b) = \{(\x,\bud)\in\edomain \mid \x \in \nsafe, \, \|\x-\target\| = b \}.$


We compare the numerically computed solution to the exact solution in the $L_1$ and $L_\infty$ norms.
For the $L_\infty$ norm, we compare the solutions on a subset $\edomain^\epsilon\subset\edomain$
where the $\valr$ is known to be continuous:
\bq
\edomain^\epsilon = \{(\x,\bud)\in\edomain \mid \dist{\x}{\y} > \epsilon \text{ for all $(\y,\bud) \in \D(b)$}\}.
\eq
In particular we investigate the $L_\infty$ norm errors for $\epsilon = \epsilon(h) = 3h$ and $\epsilon = 0.1$ (independent of $h$).
The $L_1$ errors are computed over the whole computational domain.
The errors are reported in Table \ref{tab:convergence}.
A contour plot of the numerical solution on the top $\bud$-slice is shown in Figure \ref{fig:conv_toplevel}.

The convergence observed in Table \ref{tab:convergence} is actually stronger
than predicted by theory.  First, in this numerical test
we always chose $\Delta b = h$, whereas the theory (even for the no resets case!)
guarantees convergence only for $h = o(\Delta \bud)$; see Remark \ref{rem:convergence}.
Secondly, the $L_\infty$-norm convergence is guaranteed on any fixed compact set
away from discontinuity, but the choice of $\epsilon = 3h$ goes beyond that.

At the same time,
the observed rate of convergence (in all norms) is less than one
despite our use of the first-order accurate discretization.
This is not related to any discontinuities in the solution, but is rather due
to the ``rarefaction fans'' (characteristics spreading from a single point)
present in this problem.
Indeed, this phenomenon is well known even for computations of distance function from
a single target point: a cone-type singularity in the solution results in much larger
local truncation errors near the target, thus lowering the rate of convergence.
A ``singularity factoring'' approach recently introduced in \cite{FomelLuoZhao}
allows to circumvent this issue at the target, but we note that there are two more
rarefaction fans spreading from points $P_1$ and $P_2$; see Figure \ref{fig:conv_toplevel}.
(Intuitively, this is due to the fact that optimal trajectories from infinitely many
starting points pass through $P_1$ or $P_2$ on their way to $\target$.)
Since in general examples the locations and types of such rarefaction fans are a priori unknown,
``factoring methods'' similar to those in \cite{FomelLuoZhao} are not applicable.

\begin{figure}[ht]
\begin{minipage}[b]{0.45\linewidth}
\centering
 \includegraphics[height=2.5in]{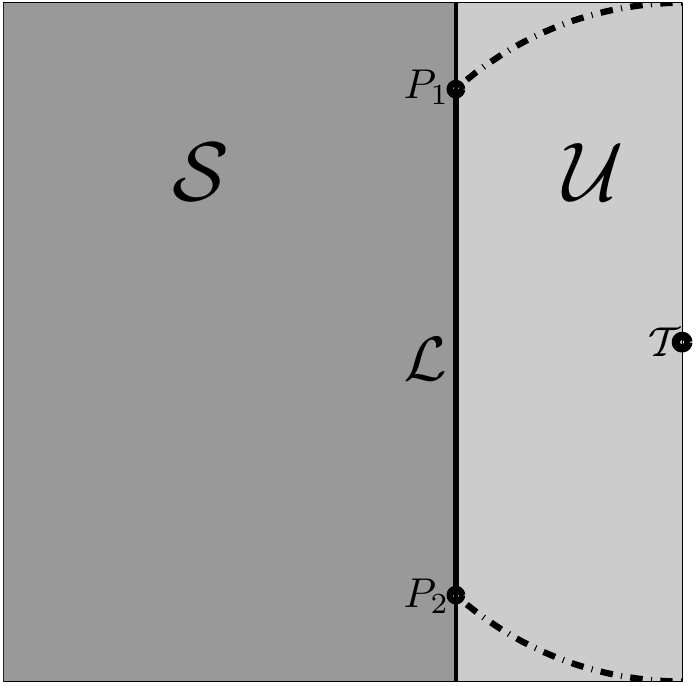}
\end{minipage}
\hspace{0.3cm}
\begin{minipage}[b]{0.4\linewidth}\centering
\includegraphics[height=2.5in]{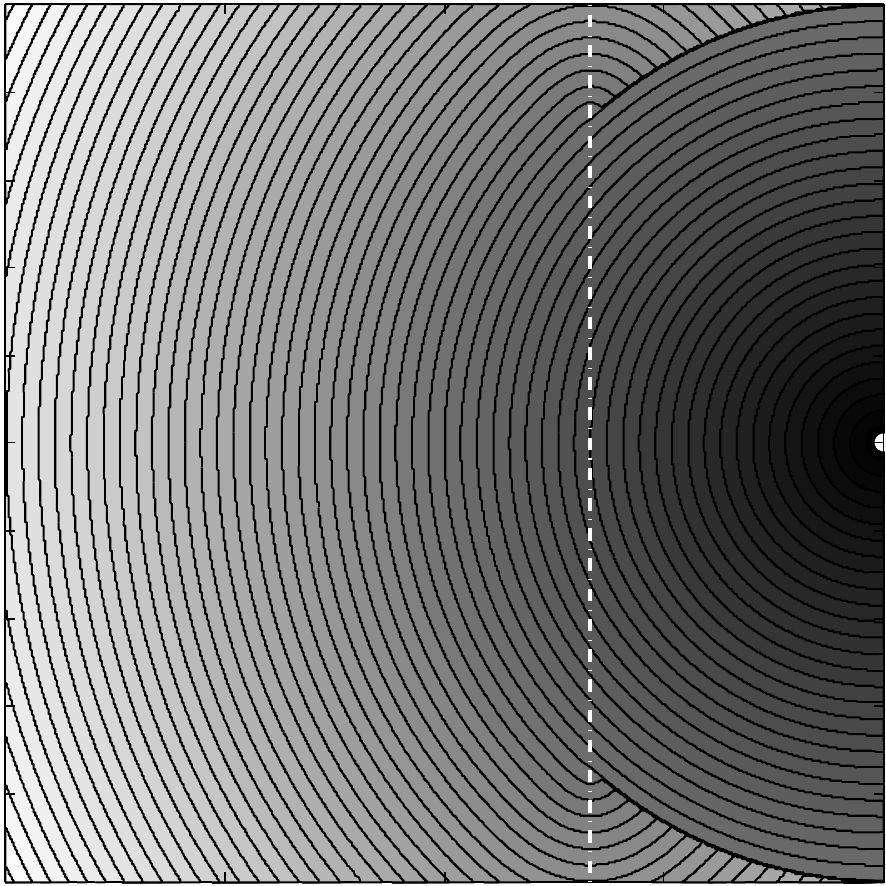}
\end{minipage}
\caption{Left: an illustration of the domain. The darker region is $\mathcal{S}$ and the lighter region
is $\mathcal{U}$. The target is $\target = (1,0)$. The black dotted line is the circle of radius $\maxb=1$
centered about $\target$ in $\mathcal{U}$. Right: a contour plot of the numerical solution of $\valr(x,y,1)$
for grid size $N=961$. The vertical white dotted line is the interface between $\safe$ and $\nsafe$ at $x=1/3$.
}
   \label{fig:conv_toplevel}
\end{figure}

\begin{table}
\centering
\begin{tabular}{|c|c||c|c||c|c||c|c|}
\hline
$N$ & h & $L_1(\edomain)$ & rate & $L_\infty (\edomain^\epsilon)$, $\epsilon = 3h$ & rate &
$L_\infty(\edomain^\epsilon)$, $\epsilon = 0.1$ & rate \\
\hline
\hline
61 & 1/30 & 0.0612 & - & 0.0681 & - & 0.0681 & - \\
\hline
121 & 1/60 & 0.0309 & 0.99 & 0.0513 & 0.41 & 0.0416 & 0.71 \\
\hline
241 & 1/120 & 0.0151 & 1.03 & 0.0401 & 0.36 & 0.0255 & 0.71 \\
\hline
481 & 1/240 & 0.0076 & 0.99 & 0.0265 & 0.60 & 0.0151 & 0.76 \\
\hline
961 & 1/480 & 0.0039 & 0.96 & 0.0162 & 0.71 &  0.0090 & 0.75 \\
\hline
\end{tabular}
\caption{Errors measured against the exact solution for various grid sizes $N$.
}
\label{tab:convergence}
\end{table}




\subsection{Geometry of $\safe$ and the number of iterations}



We illustrate how the information propagates within the main loop of Algorithm \ref{alg:budgetReset}.
Since the reachable part of the safe set ($\safereach$) is obtained iteratively,
it might seem natural to expect that the iterative process stops
once all the reachable components of $\safe$ are already found
(i.e., once $\safereach^k=\safereach$, also ensuring that $MFL^k = MFL$).
Here we show that this generally need not be the case and
the value function $\Val^k$ might require more iterations to converge after that point.
Roughly speaking, this occurs when the order of traversal of some optimal path
through a sequence of connected
components of $\safereach$ differs from the order in which those components were discovered.
Mathematically, the extra iterations are needed because the correct values of
$\Vals^k$ are not yet known on $\interface$, and the values of $\Valu^{k+1}$ on the MFL
are still incorrect as a result.

\begin{figure}[htbp]
   \centering
	\includegraphics[height=2in]{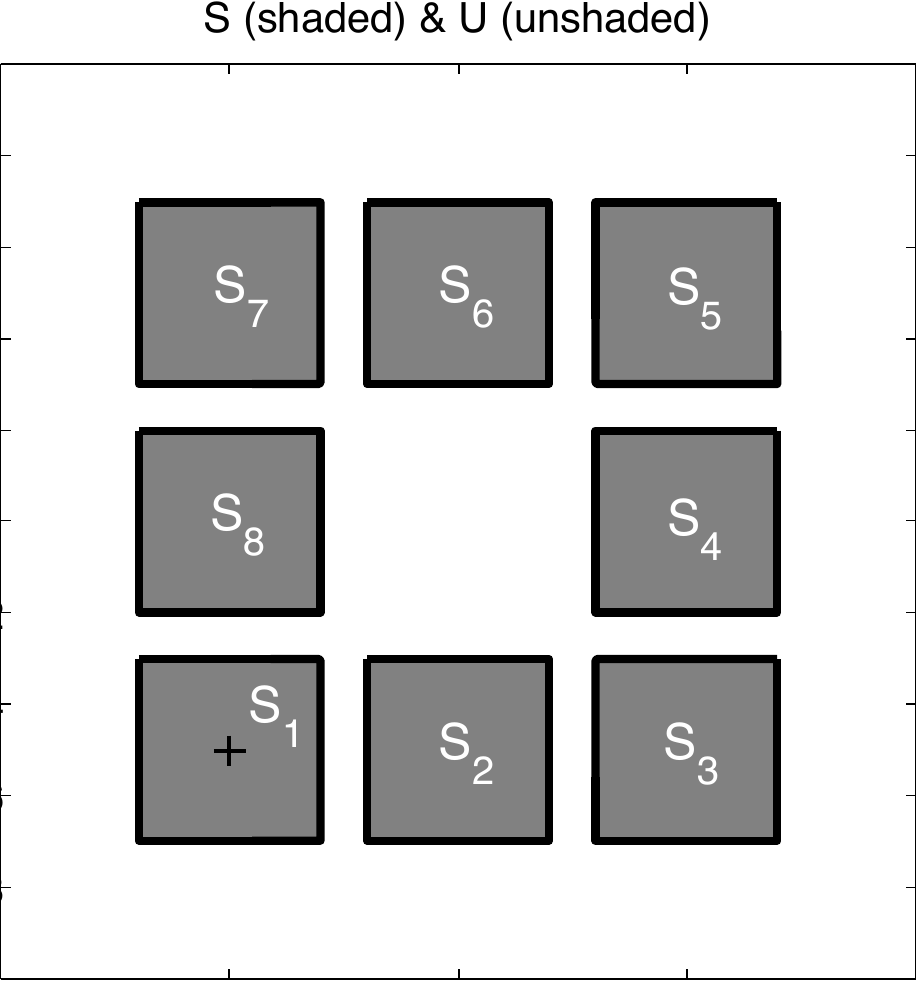}
	\hspace{1cm}
	\includegraphics[height=2in]{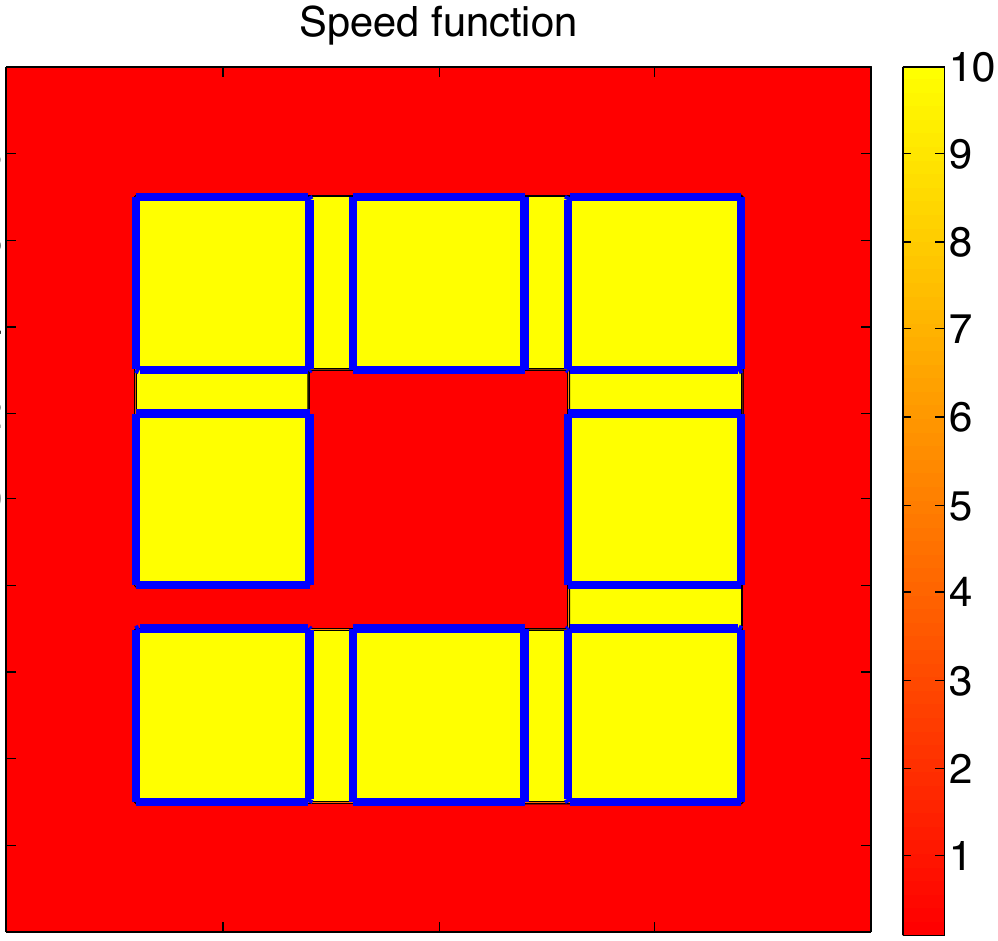}\\
	\caption{
	Left: The sets $\safe$, $\nsafe$ and the target $\target = (-0.5, -0.5)$ shown by a black 	cross.
	The eight connected components $S_1,S_2,\dots,S_8$ of $\safe$ are labeled.
	Right: the speed function.
	Note the slow speed in the corridor $C_8$ between the $S_1$ and $S_8$. }
   \label{fig:reset_iters1}
\end{figure}

Consider the following ``pathological'' example
(shown in
Figure \ref{fig:reset_iters1}): $\safe$ consists
of eight
square blocks $S_1,S_2,\dots, S_8$ with side lengths 0.4,
enumerated counter-clockwise,
with $\target\in S_1$.
To simplify the discussion,
we assume that $S_9:=S_1$ and introduce the notation
for ``unsafe corridors'' between the safe squares:
$$
C_i \; = \; \text{convex\_hull}(S_i \cup S_{i+1}) \, \cap \, \nsafe;
\qquad i=1,\ldots, 8.
$$
We set
\[
f(\x) =
\begin{cases}
10 & \x\in\safe\cup C_1 \cup C_2 \cup\cdots\cup C_7,\\
0.1 & \text{otherwise.}
\end{cases}
\]
Note that $f=0.1$ in $C_8$.
The target is $\target = (-0.5,-0.5)$ and the maximum budget is $\maxb = 1.5$.
Note that for all the starting positions in $S_2, \ldots, S_8$,
the corresponding constrained optimal trajectories will run toward $S_1$ clockwise
(to avoid crossing through the slow region in $C_8$).  The same is true for
all starting positions in $C_1,\ldots,C_7$ and even for points in $C_8$
that are significantly closer to $S_8$.

This problem was specifically designed to illustrate the difference between
the evolution of the reachable set $\domainreach^k = \nsafereach^k\cup\safereach^k$
and the evolution of the ``fully converged'' set
$\mathcal{F}^k=\{\x\mid \valr^k(\x)=\valr(\x)\}$, on which the value function
is correct after the $k$-th iteration.
Both sets eventually contain all $S_i$'s and $C_i$'s,
but a careful examination of Figure \ref{fig:reset_iters2}
reveals the difference.
For the reachable set, $\domainreach^1 = S_1$ initially, and the algorithm iteratively
discovers the reachable $S_i$'s (and $C_i$'s) simultaneously in both
the clockwise and counter-clockwise directions.
More than one $S_i$ can  be ``discovered'' per iteration in each direction:
e.g., at iteration $k=3$, Phase I of the algorithm discovers
that $C_2\cup C_3\subset\nsafereach^3$ owing to a feasible path that
``passes through the corner'' shared by $C_2$ and $C_3$; this subsequently leads to
Phase II discovering that $S_3\cup S_4\subset\safereach^3$.
The same argument implies that $C_6\cup C_7\subset\nsafereach^3$
and $S_6\cup S_7\subset\safereach^3$.  After one more iteration
we already see that $S_5\subset\safereach^4$ and another iteration
is sufficient to recover the entire reachable set $\domainreach = \domainreach^5$.

However, on a large part of $\domainreach$ the value function is still quite far from correct
at that point; e.g., a comparison of level curves in $C_5$ and $S_5$ shows that
$C_5 \not \in \mathcal{F}^5$.  Since $\target\in S_1\subset\safe$ and $S_1$ is convex,
we have $S_1 \subset \mathcal{F}^1$.  It is similarly easy to see that
$(\mathcal{F}^{k-1} \cup C_{k-1} \cup S_k) \subset \mathcal{F}^k$ for $k=2,...,8$.
Thus, it takes eight iterations to correctly compute $\valr$ on the entire safe set
and one more iteration to cover those unsafe points (including a part of $C_8$),
whose optimal trajectories take them first to $S_8$ and then clockwise
(through the ``fast belt'') toward $S_1$.
We note that, as iterations progress, the value function need not be converged
on recently discovered components of $\safereach$, even if the level sets already show
the generally correct (clockwise) direction of optimal trajectories.  For example,
it might seem that $S_8 \in \mathcal{F}^6$, but a careful comparison of level curves
shows that the value function is still incorrect even on $C_6$ and $S_7$.
(This is due to the fact that the reachability of $S_8$ is discovered by feasible
trajectories passing through a common corner of $C_6$ and $C_7$.)
Figure \ref{fig:convergence_pattern} confirms this by showing the $\infty$-norm of
the value changes after each iteration. We observe two key events:\\
$\bullet \;$ the initial drop in value changes,
when $\domainreach$ is fully discovered after iteration five\\
(at which point the errors are still independent of $h$ and $\Delta b$);\\
$\bullet \;$ and the convergence (i.e., the drop of value changes to the machine precision) after iteration nine.


\begin{figure}[htbp]
   \centering	
	\includegraphics[width=1.9in]{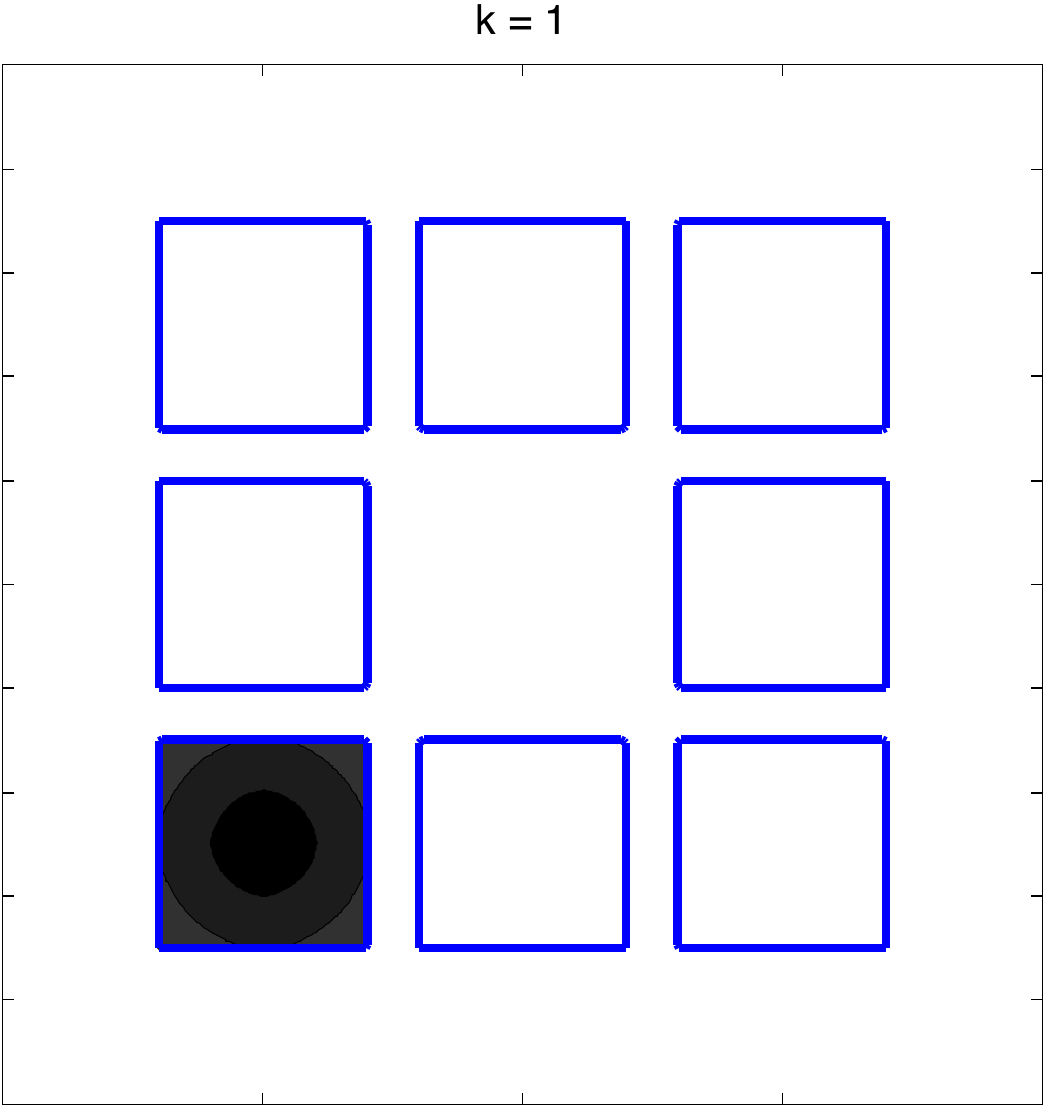}
	\includegraphics[width=1.9in]{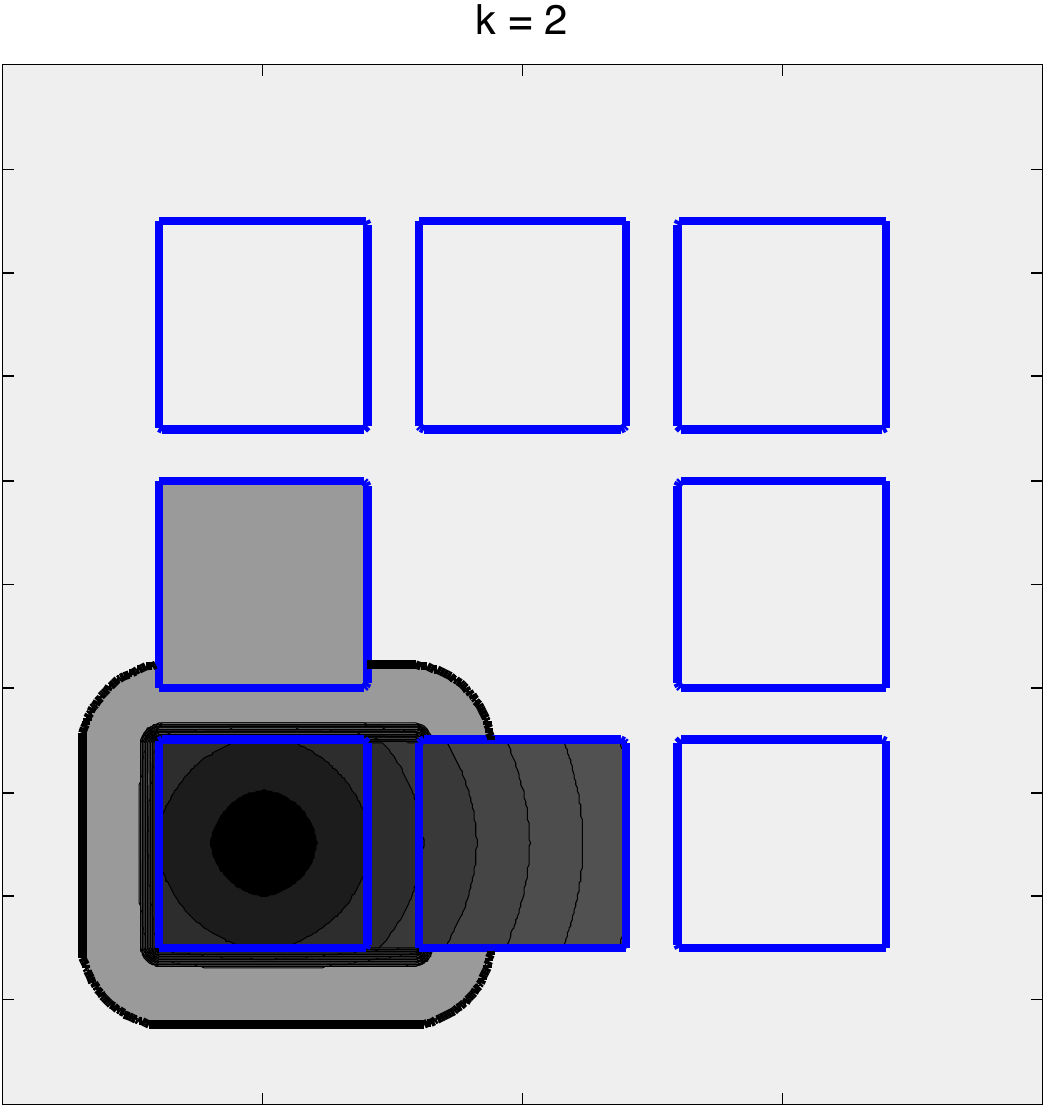}
	\includegraphics[width=1.9in]{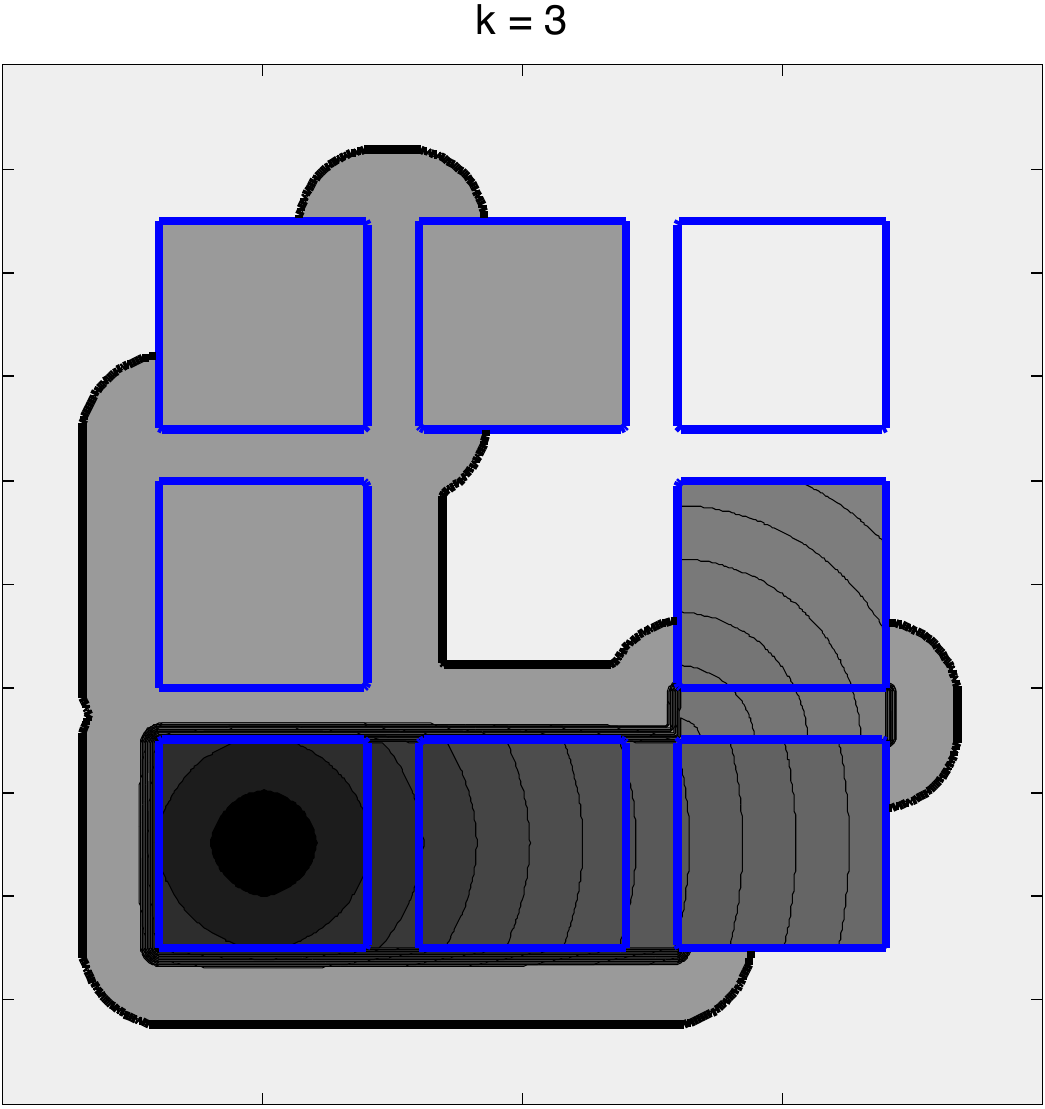}\\
	\includegraphics[width=1.9in]{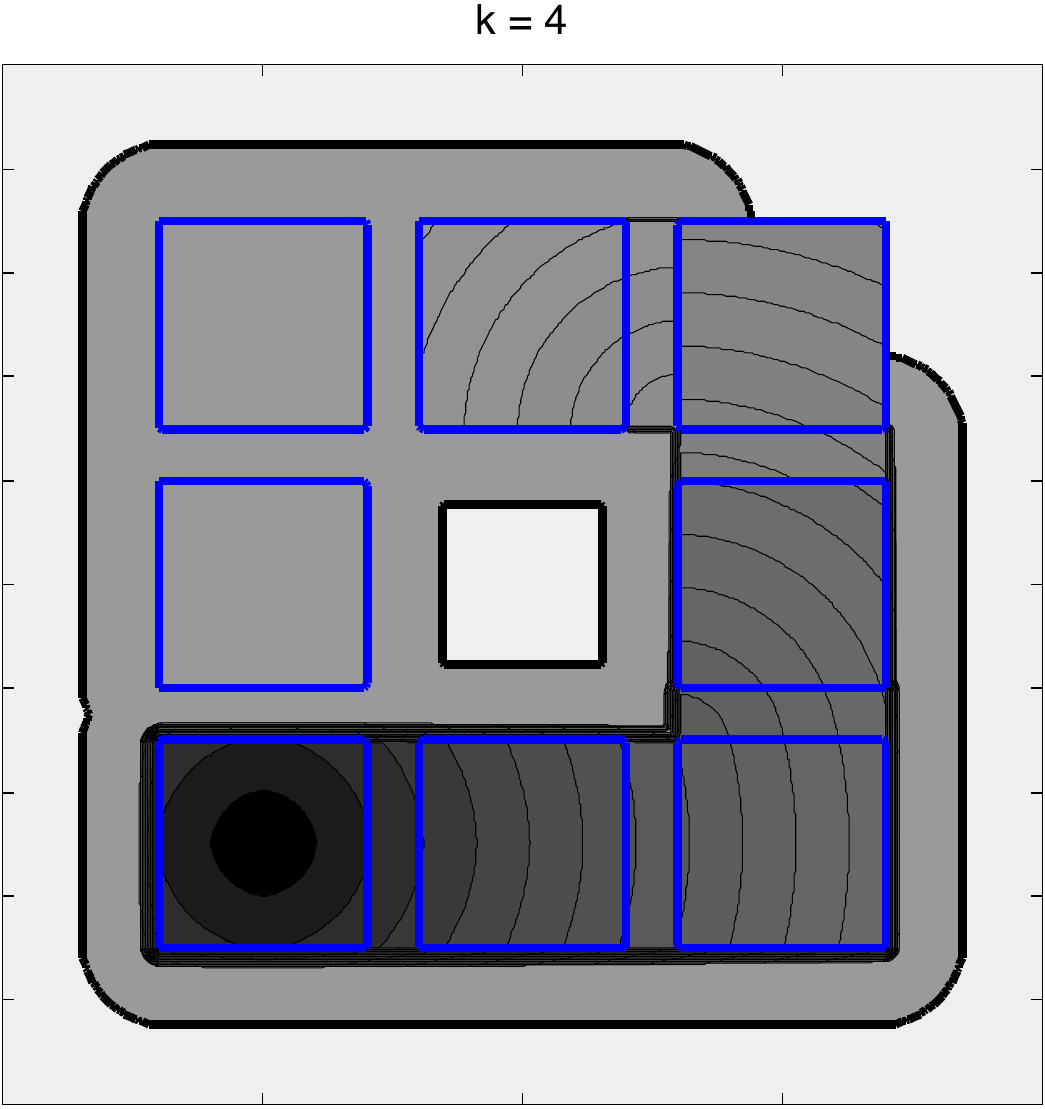}
	\includegraphics[width=1.9in]{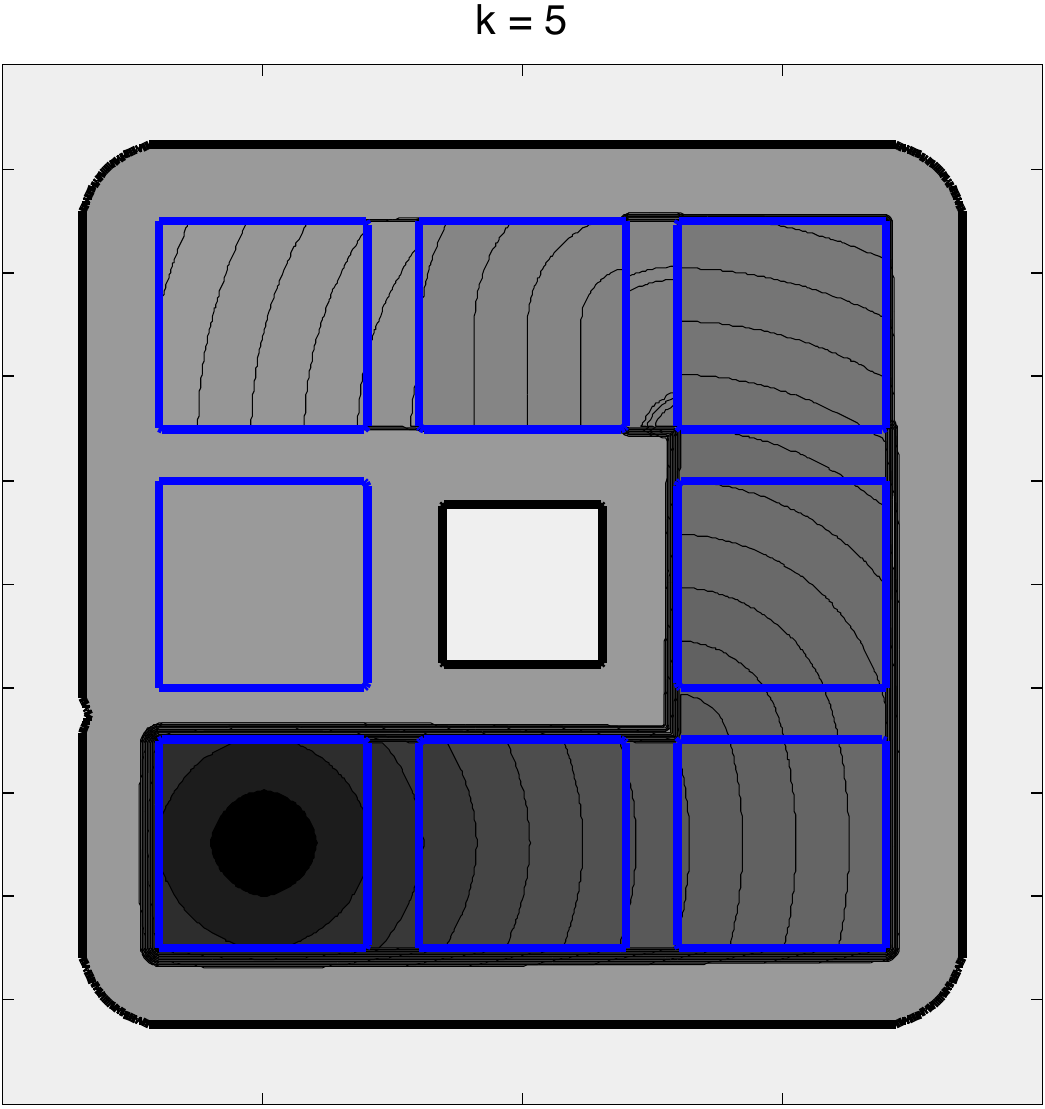}
	\includegraphics[width=1.9in]{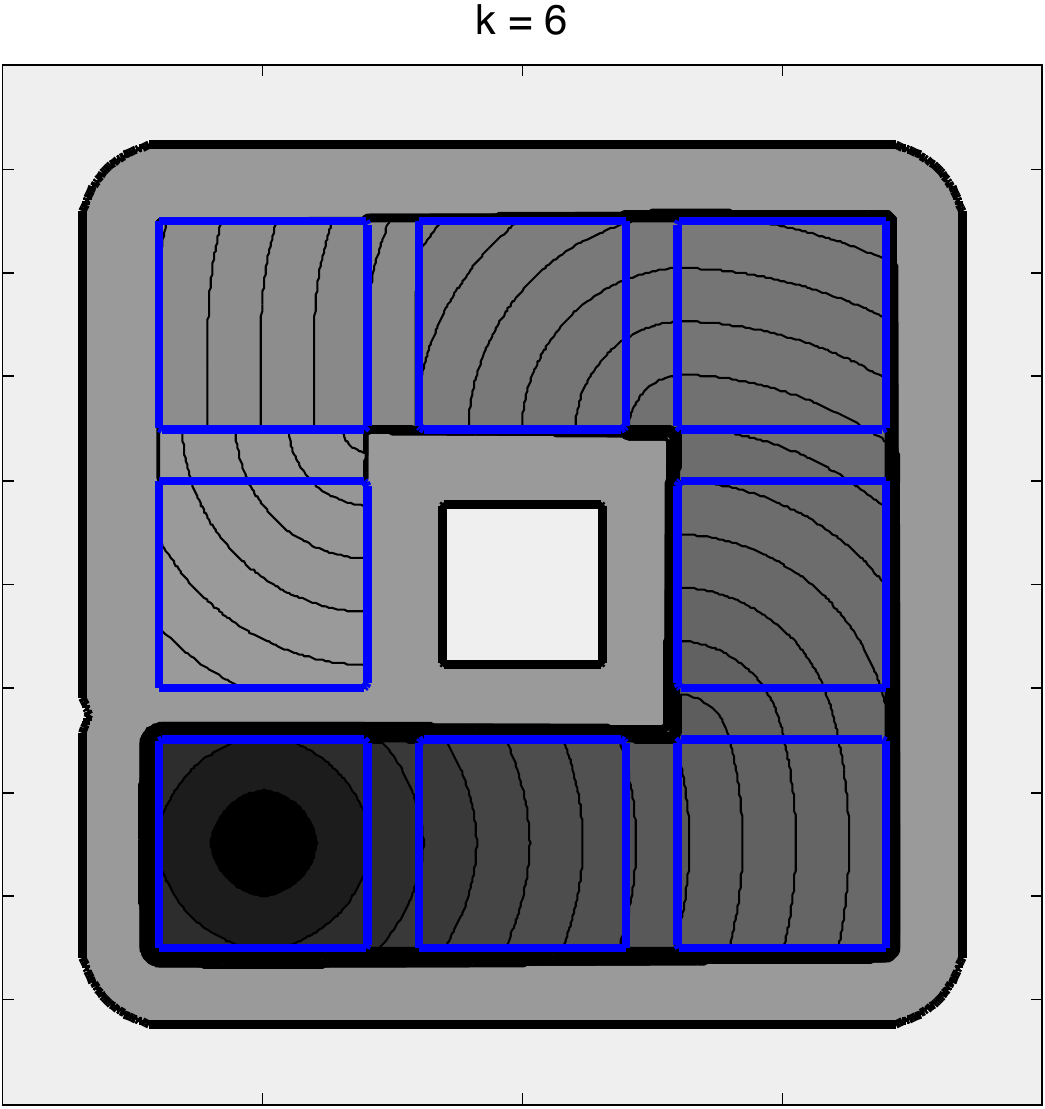}
   \caption{
   The first six iterations (each after Phase II) of $\Val^k$ on $\bud=\maxb$.
   The contour lines are scaled logarithmically to avoid bunching in the slow regions of $\nsafe$.}
   \label{fig:reset_iters2}
\end{figure}

\begin{figure}[htbp]
   \centering
	\includegraphics[width=3in]{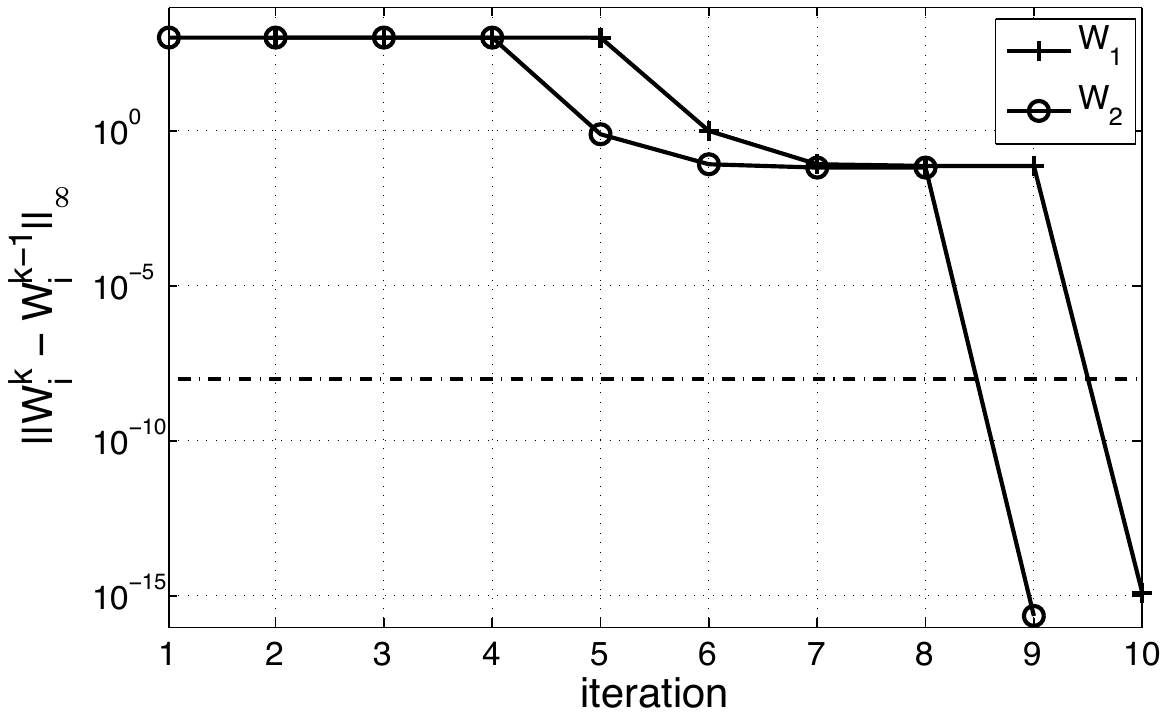}
   \caption{The max-norm difference between consecutive $\Valu^k$ and $\Vals^k$.
   The numerical convergence tolerance is shown by the thick horizontal dotted line.
  Both $\Valu^0$ and $\Vals^0$ were set to $10^3$.}
   \label{fig:convergence_pattern}
\end{figure}



\subsection{Optimal paths and the effect of varying $\maxb$}



We now consider two examples with inhomogeneous speed functions.
The first scenario involves a discontinuous speed function that is slow in the safe set:
\[
f(\x) =
\begin{cases}
1 & \x\in\nsafe\\
0.3 & \x\in\safe.
\end{cases}
\]
This example presents a curious dilemma: any good path will try to avoid $\safe$ to travel faster to $\target$,
but it must visit $\safe$ at least every $\maxb$ distance to keep the budget from depleting.
The numerical test for $\maxb = 0.4$ is shown in the center plot of Figure \ref{fig:inhomo}.
The computed ``optimal path'' tends to travel
along the interface $\interface$ on the $\nsafe$ side while occasionally making short visits into
$\safe$ to reset the budget\footnote{
Since $\interface \subset \safe$, it is not really possible to quickly travel on the $\interface$ itself. As a result, an optimal control does not exist, though the value function $\valu$ is still well-defined.
This lack of optimal control does not contradict the compatibility condition \eqref{comp_1} (i.e., $\valu = \vals$ on $\interface$).
The numerically recovered ``optimal'' trajectory shown in this Figure is 
actually
``$\varepsilon$-suboptimal'', where $\epsilon \to 0$ under the grid refinement.
}.
We have added small white circles to the plot in Figure \ref{fig:inhomo} (center) to identify the locations of these ``reentry'' points.

The second example illustrates the robustness of the numerical scheme to non-trivial speed functions in $\nsafe$:
we set $\maxb = 0.4$ and
\[
f(\x) = 1 - 0.5\sin(5\pi x)\sin(5\pi y).
\]
The computed value function and a sample path are shown in the right plot of Figure \ref{fig:inhomo}.






\begin{figure}[htbp]
\begin{center}
\includegraphics[width=1.9in]{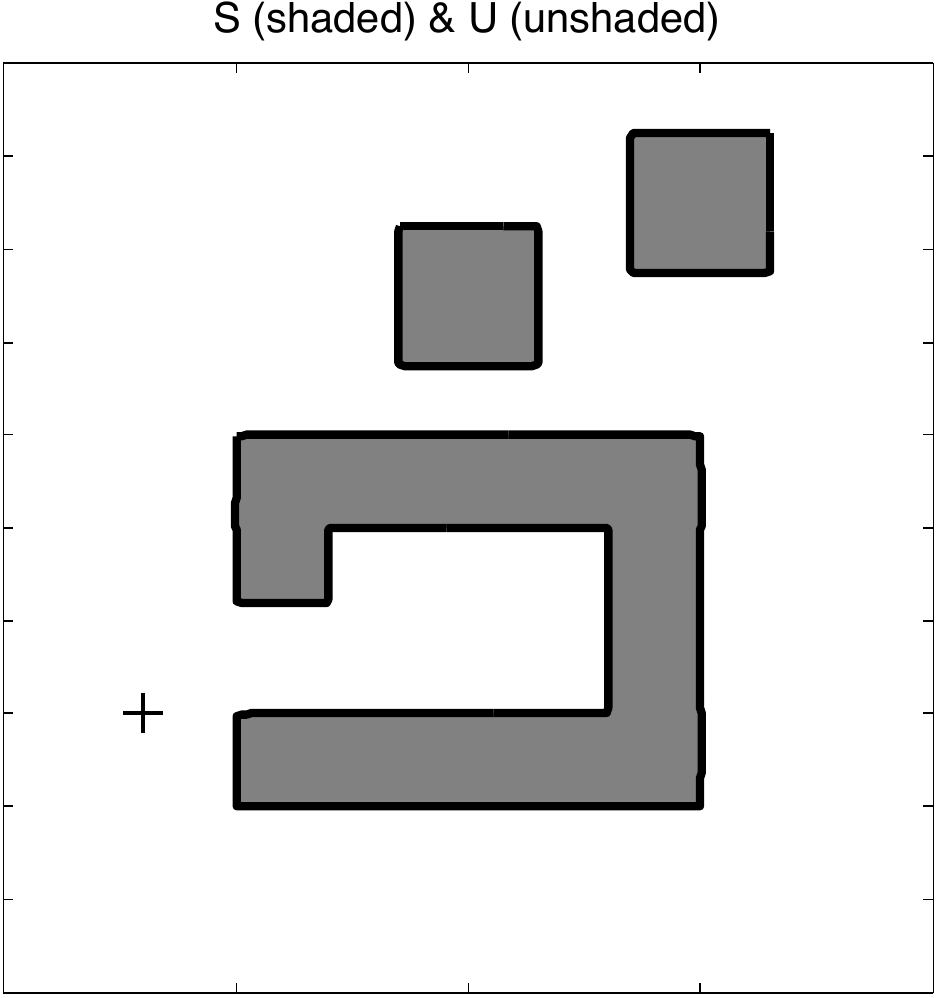}
\includegraphics[width=1.9in]{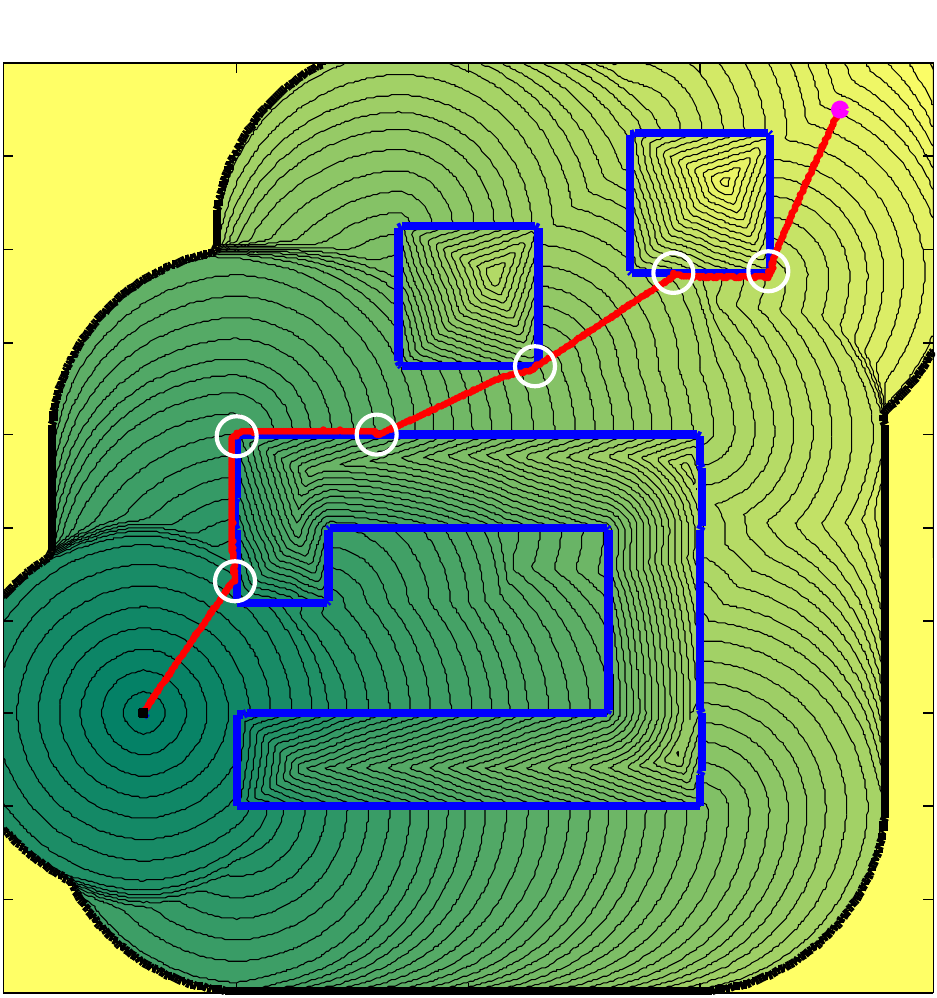}
\includegraphics[width=1.9in]{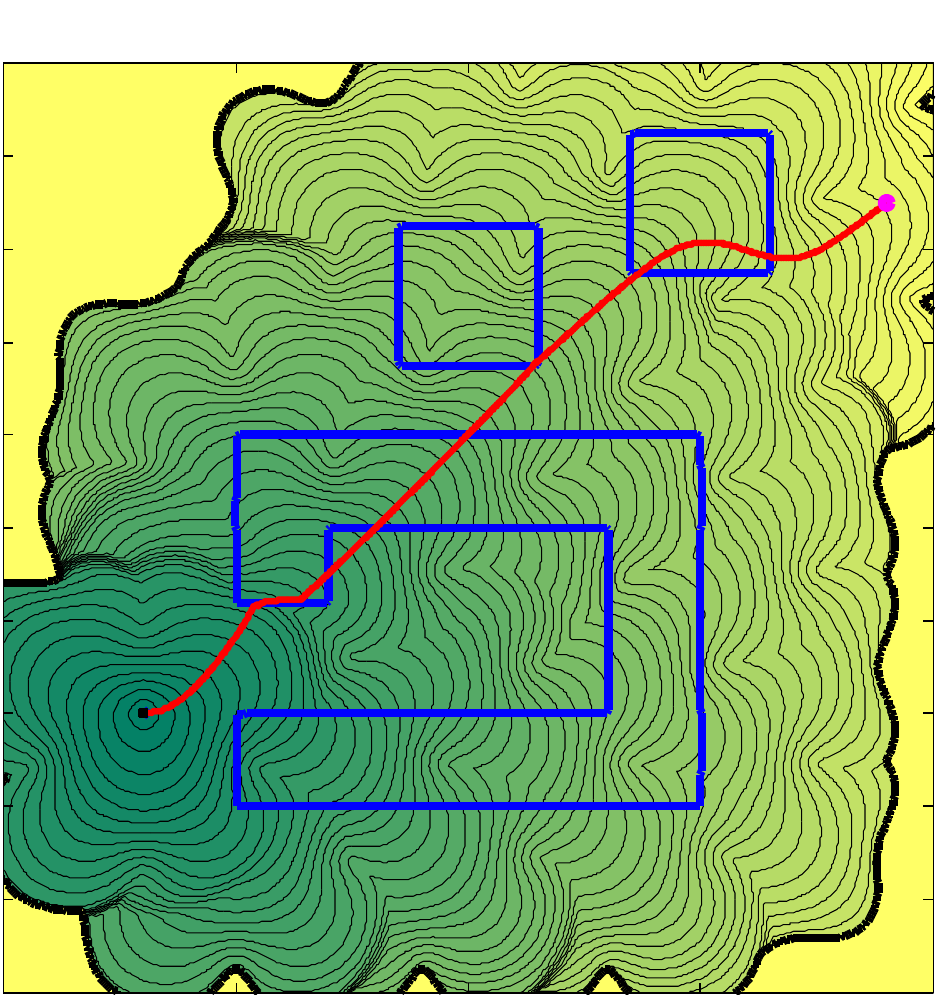}
\caption{Sample optimal paths on the ``islands'' example with inhomogeneous speed functions. }
\label{fig:inhomo}
\end{center}
\end{figure}






Next, we give numerical examples showing the effects of varying $\maxb$.
Figure \ref{fig:varyB} illustrates these effects on the ``islands'' examples (as in figure \ref{fig:inhomo}).
The speeds were set to $f=1$ on all of the domain.
While the optimal path is computed from the same initial point $(0.8,0.5)$,
note the large changes in its global behavior.

\begin{figure}[htbp]
\begin{center}
\includegraphics[width=1.9in]{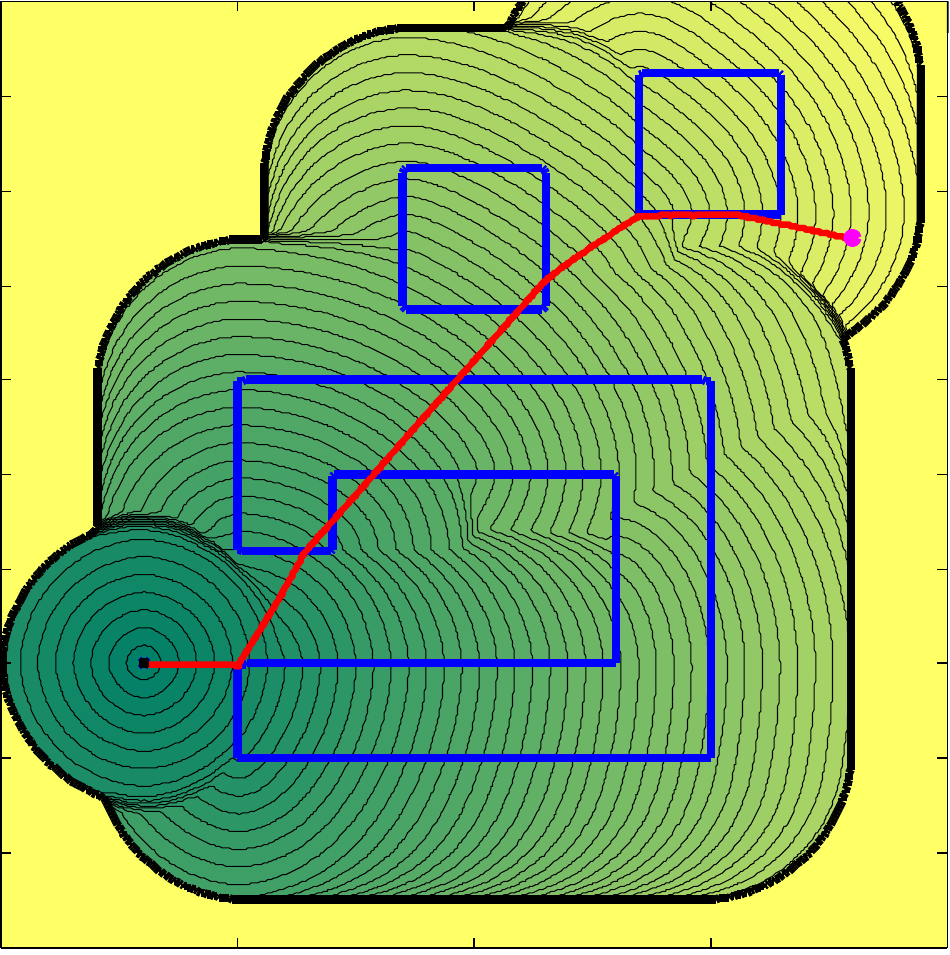}
\includegraphics[width=1.9in]{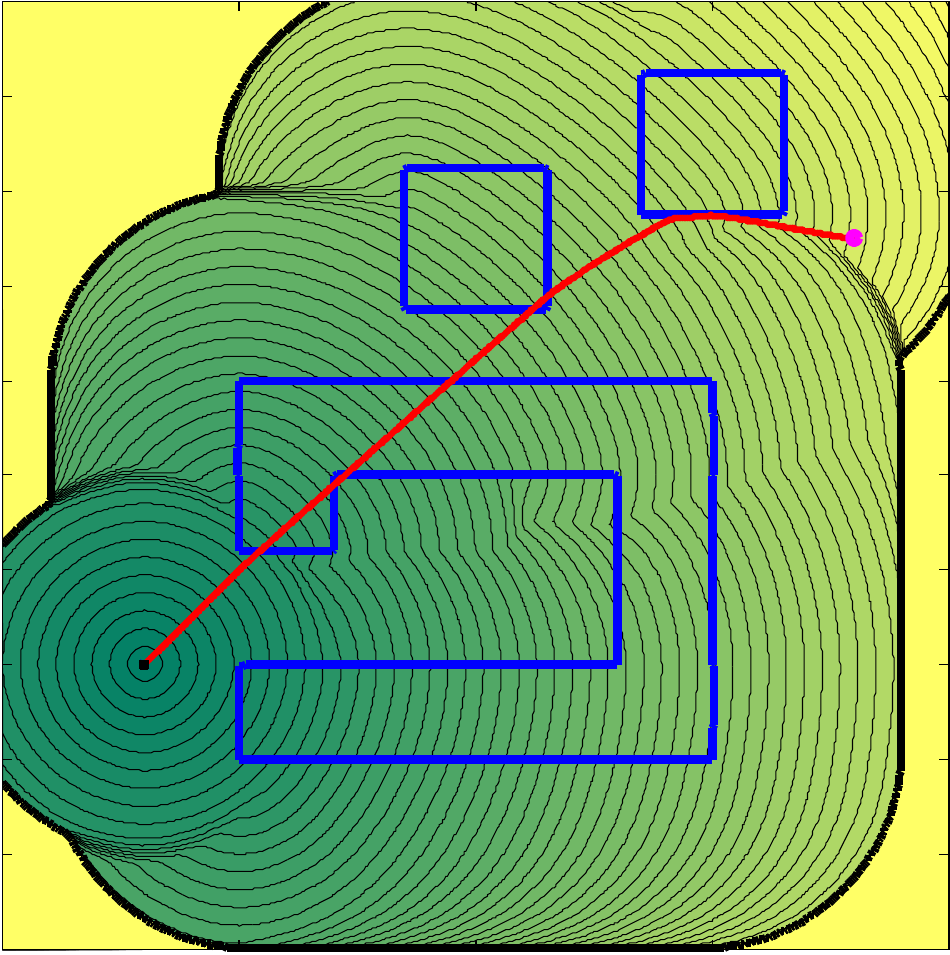}
\includegraphics[width=1.9in]{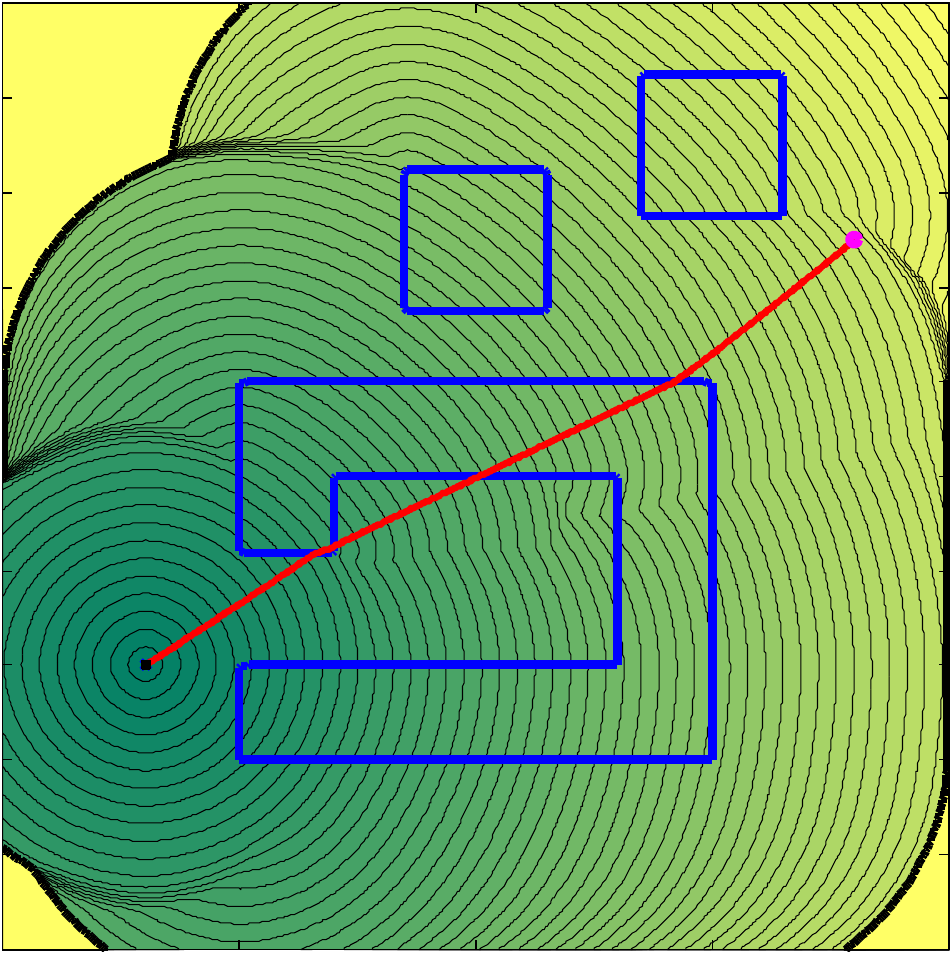}
\caption{Sample optimal paths on the ``islands'' example for varying $\maxb$. Left to right: $\maxb = 0.3$, 0.4 and 0.5, respectively.}
\label{fig:varyB}
\end{center}
\end{figure}

\label{ss:inhomog_tests}




\subsection{Constrained contiguous visibility time.}\label{visibility}
We apply Algorithm \ref{alg:budgetReset} to a problem involving visibility exposure:
suppose the objective is to move a robot towards a target in the shortest time,
among opaque and impenetrable obstacles, while avoiding  ``prolonged exposure''
to a static enemy observer.
In our budget reset problem setting, we impose the prolonged exposure constraint by
letting the enemy-visible region be $\nsafe$ and the non-visible region as $\safe$.
This is similar to the problem considered in \cite{KumarVlad}, except that once the
robot enters the non-visible region, it is again allowed to travel through the visible region
up to the time $\maxb$.

The domain consists of four obstacles which act both as state constraints and
as occluders.
The static observer is placed at $(0.8,0.8)$, and the corresponding visible set is
computed by solving an auxiliary (static and linear) PDE on $\domain$ \cite{visPaper}.
We compute the value function and optimal paths for the same starting location
but two different exposure budgets: $\maxb = 0.15$ and $\maxb = 0.3$, see
figure \ref{fig:visibility}.
Note that, for small $\maxb$, the budget is insufficient to directly travel across
the $\nsafe$ `corridor' between the ``shadows'' of the larger foreground obstacles;
for the larger $\maxb$ this shortcut is feasible, thus reducing the path length.

\begin{figure}[htbp]
   \centering
	\includegraphics[width=1.9in]{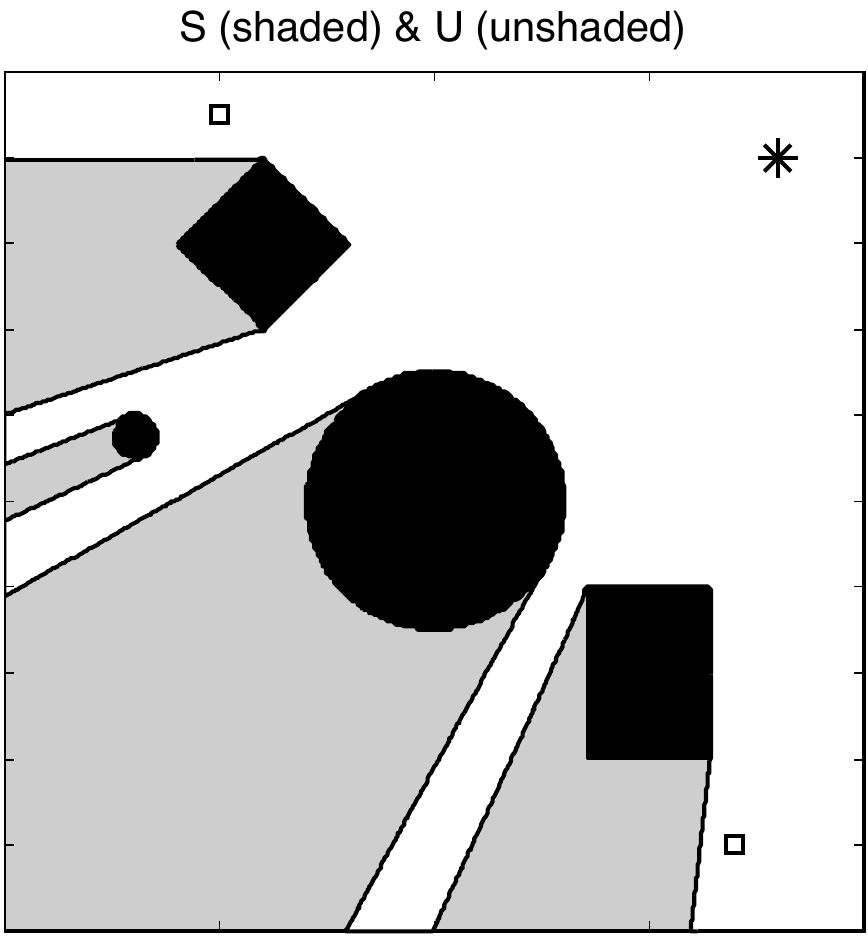}
	\includegraphics[width=1.9in]{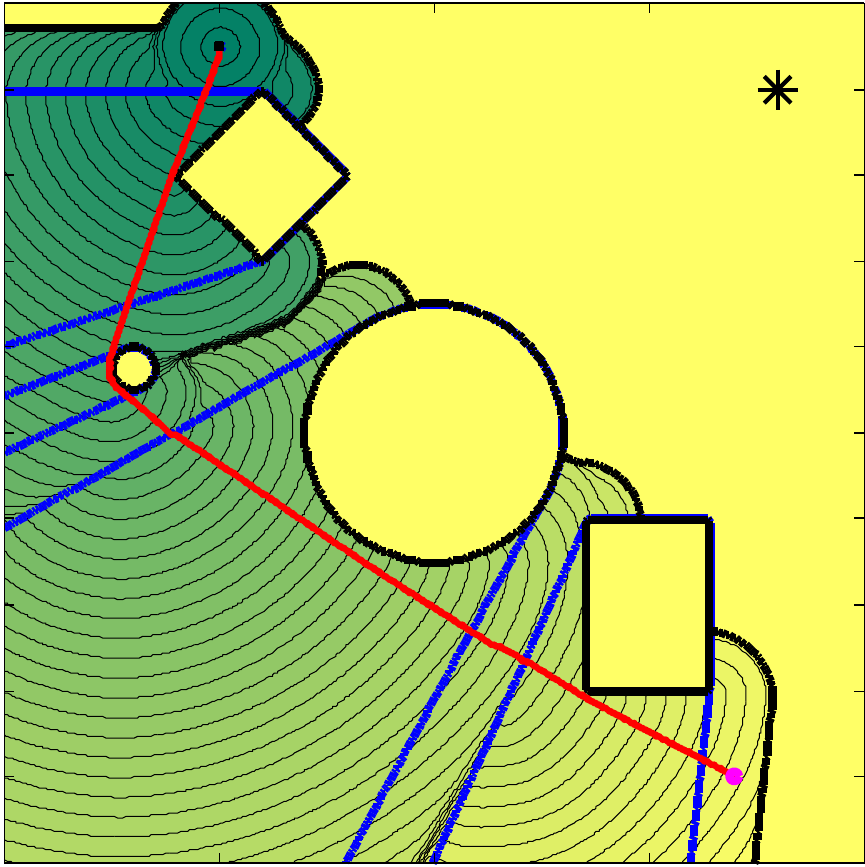}
	\includegraphics[width=1.9in]{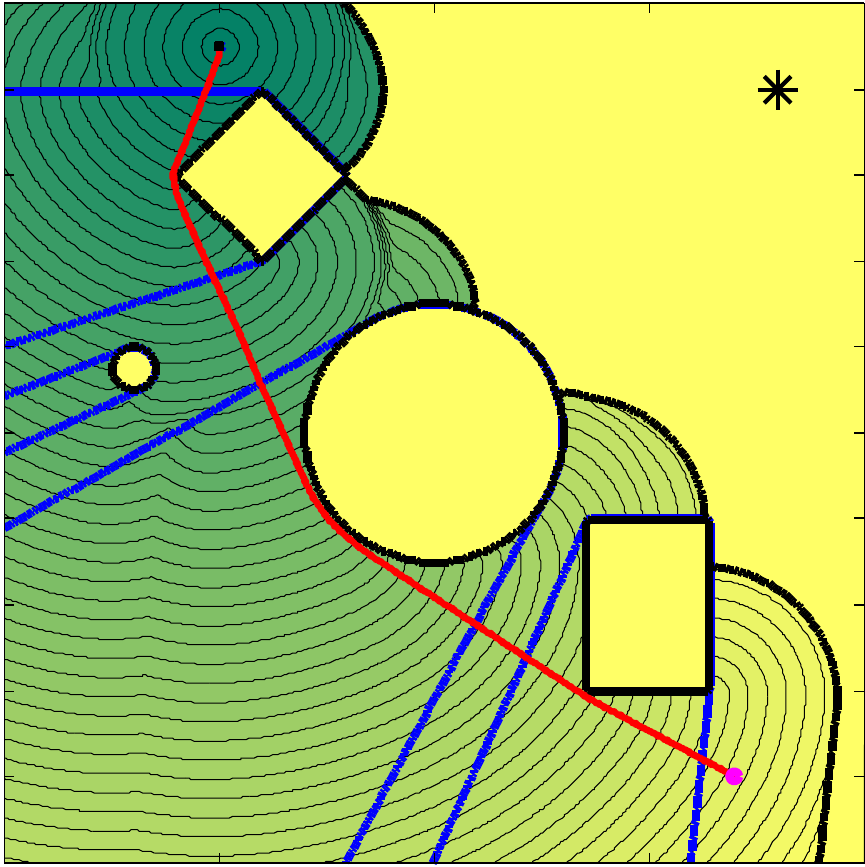}
   \caption{The problem of constrained contiguous visibility time.
   The static observer location is shown by an asterisk.
   Left: the observer-viewable region is white; the occluders/obstacles are black, and their ``shadows'' are gray.
   The objective of is to find the quickest path connecting the two small squares
   while avoiding prolonged enemy-exposure.
   Center and right: contour plots of $\Val$ at $\bud=\maxb$ and the constrained optimal
   paths with $\maxb = 0.15$ and 0.3, respectively.}
   \label{fig:visibility}
\end{figure}
\label{ss:visibility_tests}

\section{Conclusions.}
\label{s:conclusions}
In this paper we focused on computational methods for optimal control of budget-constrained problems
with resets.  In the deterministic case on graphs, we explained how such problems
can be solved by noniterative (label-setting) methods. We then introduced new fast iterative methods,
suitable for
\iffullversion
both deterministic and stochastic problems on graphs,
\fi
as well as for continuous deterministic budget reset problems.
Throughout, we utilized the causal properties of the value function
to make sure that dynamic programming equations are solved efficiently
on the new extended domain.
\iffullversion
In the appendix we also describe an iterative algorithm on the original (non-extended) domain
for solving a related simpler problem of finding the budget-reset-``reachable''
parts of the state space.
\fi

We presented empirical evidence of convergence
and illustrated other properties of our methods on several examples, including
path-planning under constraints on ``prolonged exposure'' to an enemy observer.
Even though all selected examples are isotropic in cost and dynamics,
only minor modifications (with no performance penalties) are needed to treat anisotropy
in secondary cost $\runcostb$.
Anisotropies in primary cost and/or dynamics can also be easily treated by switching to
a different (non-iterative or fast iterative) method on the safe set $\safe$.
Several other natural extensions are likely to be more computationally expensive,
but can be handled in the same framework:
\begin{itemize}
\item
In some applications the
resource restoration is more realistically modeled not as an instantaneous reset,
but as a continuous process on $\safe$.  Alternatively, resets
might be also modeled as conscious/optional control decisions
available on $\safe$, with an instantaneous penalty in primary cost.
\item
Differential games can be similarly modified to account for limited (and possibly renewable)
resource budgets.
\item
More generally,
both the dynamics and the budget changes might be affected by some random events,
leading to stochastic trajectories in the extended domain.
\item
It would be interesting to extend the method to problems with constraints on multiple
reset-renewable resources.
\end{itemize}
For problems with a fixed starting position, significant computational savings could be
attained by adopting A*-type domain restriction algorithms \cite{Clawson1}.
All of these extensions are of obvious practical importance in realistic applications,
and we hope to address them in the future.
Finally, more work is clearly needed to provide proofs of convergence of semi-Lagrangian schemes
to discontinuous viscosity solutions of hybrid systems.

\iffullversion
\appendix

\section{Determining the Reachable Set without expanding the state space.}
\label{s:Reachable}

For certain applications, one may be interested in computing only the \emph{reachable sets}:
\[
\domainreach = \{(\x,\bud) \mid \valr(\x,\bud)< \infty\}\subset\bar{\domain}.
\]
While $\domainreach$ is indeed a byproduct of Algorithm \ref{alg:budgetReset}, the goal of this
section is to recover $\domainreach$ by a ``lower dimensional'' algorithm;
e.g., using only computations on a grid $\grid$ in $\cdomain \subset \R^n$.

We note that the reachable set can be alternatively defined
as $\domainreach = \{(\x,\bud) \mid \mfl(\x) \le \bud\}$, where
$\mfl(\x)$ is the minimum needed starting budget, defined in \eqref{mfl_defn} and \eqref{mfl_bc2}.
Thus, it suffices to compute $\mfl$ from \eqref{mfl_hjb},
which in turn requires knowing $\safereach = \{\x\in\safe \mid \vals(\x)<\infty\}$ to impose the boundary condition \eqref{mfl_bc2}.
We take an approach similar to Algorithm \ref{alg:budgetReset}, by iteratively growing the
known reachable set.
The main difference is, instead of computing $\valu^{k}$ on $\domain_{\nsafe}$,
we solve a simpler boundary value problem on $\nsafe$ itself -- here
we only need to know whether $\target =\{\x\in\boundary \mid \term(\x)<\infty\}$ is reachable (from each point in $\nsafe$ and
starting with a given budget),
instead of finding the minimum cost of reaching it.

To this end, we introduce the auxiliary functions $\aux^k\colon\domain\rightarrow \R$,
computed iteratively in tandem with $\mfl^k$,
that can be used to extract the reachable sets in the safe set:
$\safereach^k = \{\x\in\safe\mid\aux^k(\x)<\infty\}$.
Similar to the recursive algorithm in section \ref{ss:iterative_brp}, we define initially
\begin{align}
\mfl^0(\x) = \aux^0(\x) =
\begin{cases}
0 & \x\in\target,\\
\infty & \text{otherwise};
\end{cases}
\end{align}
the $k$-th auxiliary function is defined as the (discontinuous) viscosity solution of the boundary value problem
\bq\label{aux_eikonal}
\begin{aligned}
\|\grad \aux^k(\x) \| &= 0 \qquad\x\in\text{int}(\safe),\\
\aux^k(\x) &=
\begin{cases}
0 & \x\in\target\\
\liminf\limits_{\x'\to\x \atop \x'\in\nsafe}\mfl^k(\x) & \x\in\interface\backslash\target\\
\infty & \text{otherwise}.
\end{cases}
\end{aligned}
\eq
(See Remark \ref{rem:gamma_interp} regarding the `liminf' in the boundary condition above.)
The $\text{MFL}^k$ function $\mfl^k$ is defined as described in section \ref{ss:mfl} except that $\safereach^{k-1}$ is no longer computed using $\valu^k$;
instead, it is computed from $\aux^{k-1}$ via the formula 
$\safereach^{k-1} = \{\x\in\safe\mid\aux^{k-1}(\x)<\infty\}$.
Note that \eqref{aux_eikonal} implies that for each (closed) connected component $\safe'$ of $\safe$,
\[
\aux^k(\x) = \min_{\x'\in\partial\safe'}\aux^k(\x'), \quad\text{ for all }\x\in\safe', k=0,1,2\dots
\]

Once the iterative procedure reaches a steady state, the approximation to $\mfl$ can be used to extract (an approximation to) the reachable set
$\domain_\reach$.
An important observation is that $\safereach^k$ represents all point in $\safe$ that can be reached from $\target$ by a feasible path 
containing at most $k$ contiguous segments through $\safe\backslash\target$.
Thus, under the assumption that the number of connected components of $\safe$ is finite,
the algorithm will converge in a finite number steps.

The numerical approximations $\Mfl^k$, $\Aux^k$ of $\mfl^k$, $\aux^k$, respectively, can be solved on $\grid$ using standard numerical methods described in section \ref{ss:hjb_semilag}.
The iterative method is outlined in Algorithm \ref{alg:reach}.

\begin{algorithm}
\textbf{Initialization:}\\
$\Mfl^0(\x) = \Aux^0(\x) =
\begin{cases}
0 & \forall\x\in\partial\grid,\\
\infty &  \forall\x\in\grid\backslash\partial\grid;\\
\end{cases}$\\
\vspace{0.4cm}

\textbf{Main Loop:}\\
\BlankLine
\ForEach{$k=1,2,\dots$ until $\Mfl^k$ and $\Aux^k$ stop changing} {
\BlankLine
Compute $\Mfl^k$, using $\Aux^{k-1}$ to specify $\grid\cap\safereach^{k-1}$;\\
\BlankLine
Compute $\Aux^k$ from equation \ref{aux_eikonal}, using $\Mfl^{k}$ on $\grid_{\nsafe}$;\\
}

\caption{Reachability algorithm for the budget reset problem.}
\label{alg:reach}
\end{algorithm}

As an illustration,
we implemented Algorithm \ref{alg:reach} and applied it to the same problem as in the left plot of Figure \ref{fig:varyB} (case $\maxb = 0.3$).
The first four iterations are shown in Figure \ref{fig:reachableNumerical}.
The domain $[-1,1]^2$ was discretized on a $400\times 400$ grid.
The algorithm halted after five iterations and took 0.45 seconds.

\begin{figure}[htbp] 
   \centering
   \includegraphics[width=1.4in]{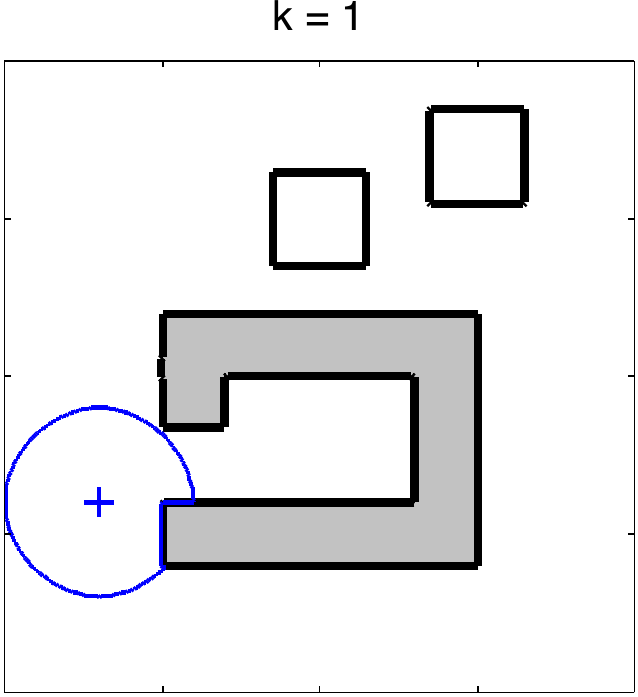}
   \includegraphics[width=1.4in]{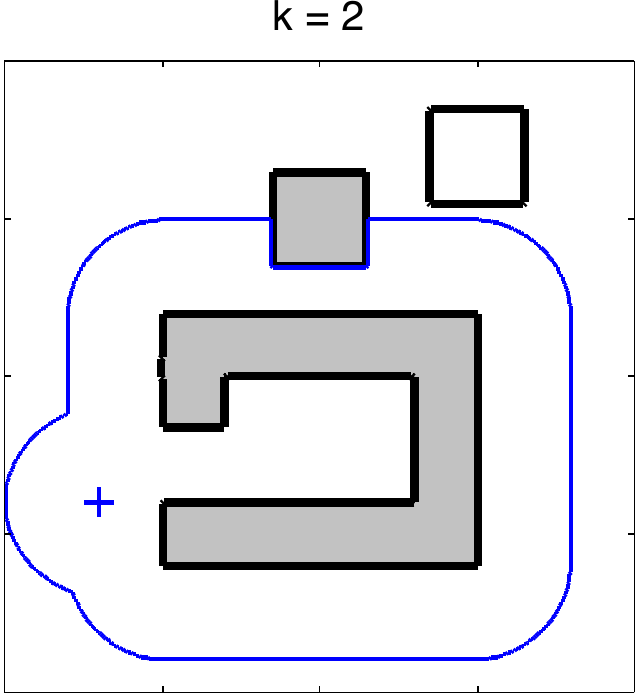}
   \includegraphics[width=1.4in]{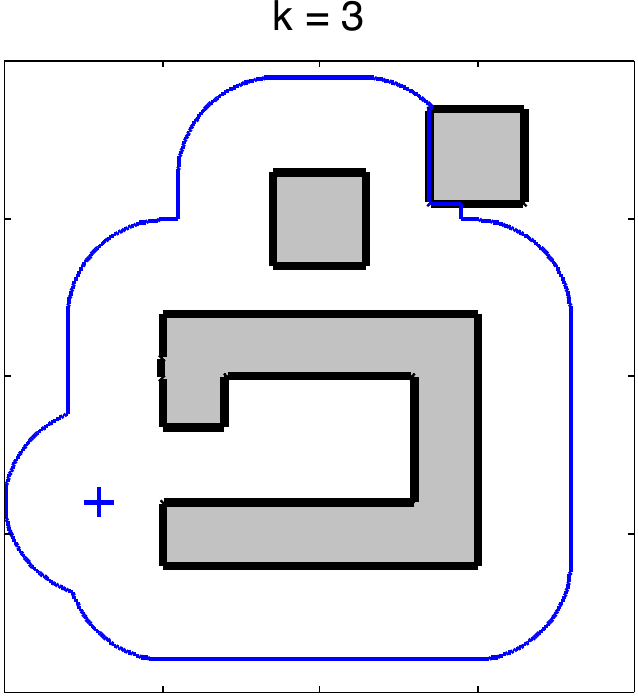}
   \includegraphics[width=1.4in]{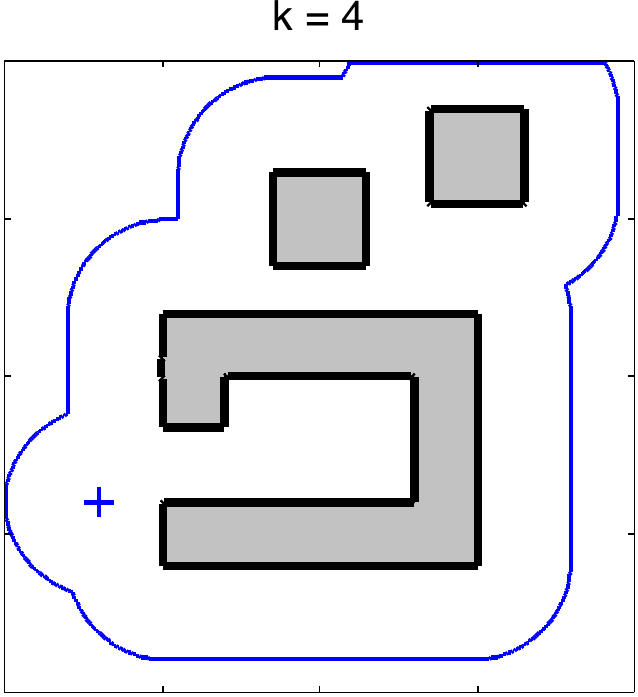}
   \caption{First four iterations for of Algorithm \ref{alg:reach}.
   The blue contours depict the level set $\{\Mfl^k(\x) = \maxb=0.3\}$ and the shaded region is the set $\safereach^{k} = \{\Aux^k(\x) < \infty\}$ at the $k$-th iteration.
   The target is shown by a `+' at $\target = (-0.7,-0.4)$.
   }
   \label{fig:reachableNumerical}
\end{figure}  
\fi

\vspace{.2in}
\noindent
{\bf{Acknowledgments:}}
The authors would like to thank Professor R. Tsai and Dr. Y. Landa, whose
work on visibility and surveillance-evasion problems (joint with Ryo Takei)
served as a starting point for this paper.  The authors are grateful to
them for helping to formulate the problem in section
\ref{visibility} and for many motivating discussions.
The authors are also grateful to Ajeet Kumar, whose source code
developed in \cite{KumarVlad} served as a starting point for our
implementation described in section \ref{s:continuous_reset}.
RT would also like to thank Professors S. Osher and C. Tomlin
for their encouragement and hospitality during the course of this work.


\end{document}